\documentclass[11pt]{article}
\usepackage{amsmath,amsthm,amssymb}
\usepackage[usenames,dvipsnames]{xcolor}
\usepackage{enumerate}
\usepackage{graphicx}
\usepackage{cite}
\usepackage{comment}
\usepackage{oands}
\usepackage{tikz}
\usepackage{changepage}
\usepackage{bbm}
\usepackage{mathtools}
\usepackage[margin=1in]{geometry}
\usepackage[pagewise,mathlines]{lineno}
\usepackage{appendix}
\usepackage{stmaryrd} 
\usepackage{multicol}
\usepackage{microtype}
\usepackage[colorinlistoftodos]{todonotes}
\usepackage{enumitem}

\usepackage[pdftitle={The fractal dimension of Liouville quantum gravity: universality, monotonicity, and bounds},
  pdfauthor={Jian Ding, Ewain Gwynne},
colorlinks=true,linkcolor=NavyBlue,urlcolor=RoyalBlue,citecolor=PineGreen,bookmarks=true,bookmarksopen=true,bookmarksopenlevel=2,unicode=true,linktocpage]{hyperref}

\theoremstyle{plain}
\newtheorem{thm}{Theorem}[section]

\newtheorem{lem}[thm]{Lemma}
\newtheorem{prop}[thm]{Proposition}

\def\@rst #1 #2other{#1}
\newcommand\MR[1]{\relax\ifhmode\unskip\spacefactor3000 \space\fi
  \MRhref{\expandafter\@rst #1 other}{#1}}
\newcommand{\MRhref}[2]{\href{http://www.ams.org/mathscinet-getitem?mr=#1}{MR#2}}

\newcommand{\arxiv}[1]{\href{http://arxiv.org/abs/#1}{#1}}

\theoremstyle{definition}
\newtheorem{defn}[thm]{Definition}
\newtheorem{remark}[thm]{Remark}

\numberwithin{equation}{section}

\newcommand{\dsb}{\begin{adjustwidth}{2.5em}{0pt}
\begin{footnotesize}}
\newcommand{\dse}{\end{footnotesize}
\end{adjustwidth}}

\newcommand{\ssb}{\begin{adjustwidth}{2.5em}{0pt}}
\newcommand{\sse}{\end{adjustwidth}}

\newcommand{\aryb}{\begin{eqnarray*}}
\newcommand{\arye}{\end{eqnarray*}}
\def\alb#1\ale{\begin{align*}#1\end{align*}}
\def\allb#1\alle{\begin{align}#1\end{align}}
\newcommand{\eqb}{\begin{equation}}
\newcommand{\eqe}{\end{equation}}
\newcommand{\eqbn}{\begin{equation*}}
\newcommand{\eqen}{\end{equation*}}

\newcommand{\BB}{\mathbbm}
\newcommand{\ol}{\overline}
\newcommand{\ul}{\underline}
\newcommand{\op}{\operatorname}

\newcommand{\re}{\operatorname{Re}}
\newcommand{\frk}{\mathfrak}
\newcommand{\eqD}{\overset{d}{=}}
\newcommand{\ep}{\epsilon}
\newcommand{\rta}{\rightarrow}

\newcommand{\wt}{\widetilde}
\newcommand{\wh}{\widehat} 
\newcommand{\mcl}{\mathcal}

\newcommand{\bdy}{\partial}

\newcommand{\tr}{{\mathrm{tr}}}
\newcommand{\LFPP}{{\textnormal{\tiny{\textsc{LFPP}}}}}

\let\originalleft\left
\let\originalright\right
\renewcommand{\left}{\mathopen{}\mathclose\bgroup\originalleft}
\renewcommand{\right}{\aftergroup\egroup\originalright}

\title{The fractal dimension of Liouville quantum gravity: universality, monotonicity, and bounds}

\date{  }
\author{
\begin{tabular}{c} Jian Ding\footnote{\textit{dingjian@wharton.upenn.edu}} \\[-5pt]\small University of Pennsylvania   \end{tabular} 
\begin{tabular}{c} Ewain Gwynne\footnote{\textit{eg558@cam.ac.uk}} \\[-5pt]\small MIT   \end{tabular}
}

\begin{document}

\maketitle
 
\begin{abstract}  
We prove that for each $\gamma \in (0,2)$, there is an exponent $d_\gamma > 2$, the ``fractal dimension of $\gamma$-Liouville quantum gravity (LQG)", which describes the ball volume growth exponent for certain random planar maps in the $\gamma$-LQG universality class, the exponent for the Liouville heat kernel, and exponents for various continuum approximations of $\gamma$-LQG distances such as Liouville graph distance and Liouville first passage percolation. We also show that $d_\gamma$ is a continuous, strictly increasing function of $\gamma$ and prove upper and lower bounds for $d_\gamma$ which in some cases greatly improve on previously known bounds for the aforementioned exponents. For example, for $\gamma=\sqrt 2$ (which corresponds to spanning-tree weighted planar maps) our bounds give $3.4641 \leq d_{\sqrt 2} \leq 3.63299$ and in the limiting case we get $4.77485 \leq \lim_{\gamma\rightarrow 2^-} d_\gamma \leq 4.89898$. 
\end{abstract}

\tableofcontents

\section{Introduction}
\label{sec-intro}

\subsection{Overview}
\label{sec-overview}

Let $\mcl D\subset\BB C$ be a simply connected domain and let $h$ be some variant of the Gaussian free field (GFF) on $\mcl D$. 
For $\gamma \in (0,2)$, the \emph{$\gamma$-Liouville quantum gravity (LQG)} surface parametrized by $\mcl D$ is, heuristically speaking, the random two-dimensional Riemannian manifold parametrized by $\mcl D$ with Riemannian metric tensor $e^{\gamma h} \, (dx^2 + dy^2)$, where $dx^2 + dy^2$ denotes the Euclidean metric tensor. The parameter $\gamma$ controls the ``roughness" of the surface, in the sense that it should in some ways behave more a smooth Euclidean surface the closer $\gamma$ is to zero. 

LQG surfaces were first introduced in the physics literature by Polyakov~\cite{polyakov-qg1,polyakov-qg2} in the context of string theory. Such surfaces are expected to describe the scaling limits of \emph{random planar maps}---random graphs embedded in the plane in such a way that no two edges cross, viewed modulo orientation-preserving homeomorphisms. 
The case when $\gamma = \sqrt{8/3}$ (sometimes called ``pure gravity") corresponds to uniform random planar maps (including uniform triangulations, quadrangulations, etc.) and other values of $\gamma$ correspond to random planar maps sampled with probability proportional to a $\gamma$-dependent statistical mechanics model, e.g., the uniform spanning tree ($\gamma =\sqrt 2$), a bipolar orientation on the edges ($\gamma = \sqrt{4/3}$), or the Ising model ($\gamma=\sqrt 3$). 

The above definition of a LQG surface does not make literal sense since the GFF $h$ is a random generalized function (distribution), so does not have well-defined pointwise values and hence cannot be exponentiated. However, one can make rigorous sense of certain objects associated with $\gamma$-LQG surfaces via regularization procedures. 
The first such object to be constructed is the \emph{$\gamma$-LQG area measure} $\mu_h$ associated with $h$, which is the a.s.\ weak limit of certain regularized versions of $e^{\gamma h(z)} \,dz$, where $dz$ denotes Lebesgue measure. This measure has been constructed in various equivalent ways in works by Kahane~\cite{kahane}, Duplantier and Sheffield~\cite{shef-kpz}, Rhodes and Vargas~\cite{rhodes-vargas-review}, and others. The construction of $\mu_h$ is a special case of the theory of \emph{Gaussian multiplicative chaos}; see~\cite{rhodes-vargas-review,berestycki-gmt-elementary} for overviews of this theory. 
For certain particular choices of $h$,\footnote{See, e.g.,~\cite{wedges,dkrv-lqg-sphere,hrv-disk,drv-torus,grv-higher-genus,remy-annulus} for definitions of the particular choices of $h$ corresponding to the scaling limits of random planar maps with different topologies. 
The $\gamma$-quantum cone, studied in Section~\ref{sec-planar-map} of the present paper, arises as the scaling limit of random planar maps with the topology of the whole plane.
We note that in the terminology of~\cite{dkrv-lqg-sphere}, etc., the term ``Liouville quantum gravity" is only used in the case when $h$ is one of these special random distributions. Here we follow the convention of~\cite{shef-kpz} and use the term ``Liouville quantum gravity" in the case when $h$ is any GFF-type distribution. }
the measure $\mu_h$ is conjectured (and in some cases proven~\cite{gms-tutte}) to describe the scaling limit of counting measure on the vertices of random planar maps embedded into the plane (e.g., via circle packing or harmonic embedding). See~\cite{shef-kpz,shef-zipper,dkrv-lqg-sphere,curien-glimpse} for conjectures of this type. 

It is expected that a $\gamma$-LQG surface also gives rise to a random metric on the domain $\mcl D$, which describes the Gromov-Hausdorff limit of random planar maps equipped with the graph distance. So far, such a metric has only been constructed in the special case when $\gamma=\sqrt{8/3}$ in a series of works by Miller and Sheffield~\cite{lqg-tbm1,lqg-tbm2,lqg-tbm3}. In this case, the $\sqrt{8/3}$-LQG metric induces the same topology as the Euclidean metric but has Hausdorff dimension 4. A certain special $\sqrt{8/3}$-LQG surface called the \emph{quantum sphere} is isometric to the \emph{Brownian map}, a random metric space which arises as the scaling limit of uniform random planar maps~\cite{legall-uniqueness,miermont-brownian-map}. 

For $\gamma\not=\sqrt{8/3}$, the metric structure of $\gamma$-LQG remains rather mysterious. Indeed, understanding this metric structure is arguably the most important problem in the theory of LQG. For $\gamma\not=\sqrt{8/3}$, a metric on $\gamma$-LQG has not been constructed, and the basic properties which the conjectural metric should satisfy --- such as its Hausdorff dimension --- are not known, even at a heuristic level. 
Nevertheless, there are a number of natural approximate random metrics which are expected to be related to the conjectural $\gamma$-LQG metric in some sense, so one can build an understanding of ``distances in $\gamma$-LQG" without rigorously constructing a metric.
\begin{itemize} 
\item \textbf{Random planar maps,} such as planar maps weighted by statistical mechanics models, as discussed above, or \emph{mated-CRT maps} as studied in~\cite{ghs-dist-exponent,gms-tutte}. 
\item \textbf{Liouville graph distance.} For $z,w\in \mcl D$ and $\ep > 0$, define the distance $  D_h^{\gamma,\ep}(z,w)$ to be the smallest $N\in\BB N$ for which there exists a continuous path from $z$ to $w$ in $\ol{\mcl D}$ which can be covered by $N$ Euclidean balls of $\gamma$-LQG mass\footnote{In the case of balls not entirely contained in $\mcl D$, we set $\mu_h \equiv 0$ outside of $\mcl D$ and for the purposes of defining the circle average we assume that $h$ vanishes outside of $\mcl D$.} at most $\ep$ with respect to $h$.  
\item \textbf{Liouville first passage percolation (LFPP).} For $\xi > 0$, $z,w\in \mcl D$, and $\delta>0$ define the distance $D_{h,\LFPP}^{\xi,\delta}(z,w)$ with parameter $\xi$ to be the infimum over all piecewise continuously differentiable paths $P : [0,T] \rta \ol{\mcl D}$ of the quantity $\int_0^T e^{\xi h_\delta(P(t))} |P'(t)| \,dt$, where $h_\delta(z)$ denotes the circle average of $h$ over $\bdy B_\delta(z)$ (as defined in~\cite[Section 3.1]{shef-kpz}). 
\item Various constructions using the so-called \textbf{Liouville heat kernel}, as defined in~\cite{grv-heat-kernel}, which is the heat kernel for \emph{Liouville Brownian motion}~\cite{berestycki-lbm,grv-lbm}. 
\end{itemize}  
We will sometimes drop the superscript $\gamma$ or $\xi$ in the notation for Liouville graph distance and LFPP when it is clear from the context. 

The above objects are defined in very different ways and it is not priori clear that they have any direct connection to each other.  
The goal of this paper is to show that there is a single exponent $d_\gamma > 2$, which we expect to be equal to the Hausdorff dimension of the conjectural $\gamma$-LQG metric, and which describes distances in all four of the above settings. Using the relationships between the exponents for the different models, we will also prove that $\gamma\mapsto d_\gamma$ is a continuous, strictly increasing function of $\gamma$ and prove new upper and lower bounds for $d_\gamma$ which (except for small values of $\gamma$) greatly improve on previously known bounds in the above settings (see Theorem~\ref{thm-d} and Figures~\ref{fig-d-bound} and~\ref{fig-exponent-table}). 

One can interpret our results as saying that even though we do not yet have a way to endow a $\gamma$-LQG surface with a metric, the fractal dimension of $\gamma$-LQG is well-defined in the sense that in each of the above settings, one has a notion of ``fractal dimension" and these notions all agree with one another. See Section~\ref{sec-related} for some additional quantities which we expect can be described in terms of $d_\gamma$, but which we do not treat in this paper.

The starting point of our analysis is a result of Ding, Zeitouni, and Zhang~\cite[Theorem 1.1]{dzz-heat-kernel} which shows the existence of a $\gamma$-dependent exponent which describes certain quantities related to Liouville graph distance and to the Liouville heat kernel. This exponent is called $\chi$ in~\cite{dzz-heat-kernel}. We set $d_\gamma := 2/\chi$. We also emphasize that some estimates in this paper differ by a factor of 2 from estimates in~\cite{dzz-heat-kernel} since the latter paper defines Liouville graph distance in terms of balls of mass $\ep^2$ instead of balls of mass $\ep$. 

\begin{thm}[\cite{dzz-heat-kernel}] \label{thm-dzz}
For each $\gamma \in (0,2)$, there exists $d_\gamma > 2$ (the \emph{fractal dimension of $\gamma$-Liouville quantum gravity}) such that the following is true. Let $\BB S = [0,1]^2$ be the unit square and let $h^{\BB S}$ be a zero-boundary Gaussian free field on $\BB S$. For any two distinct points $z$ and $w$ in the interior of $\BB S$, almost surely the $\gamma$-Liouville graph distance satisfies
\eqb \label{eqn-dzz-lgd}
\lim_{\ep\rta 0} \frac{ \log  D_{h^{\BB S}}^\ep\left( z , w   \right)    }{ \log \ep^{-1}   }  = \frac{1}{d_\gamma} .
\eqe
Furthermore, for each $\zeta>0$ there a.s.\ exists a random $C = C(z,w,\zeta,\gamma) > 1$ such that the $\gamma$-Liouville heat kernel satisfies
\eqb \label{eqn-dzz-heat-kernel}
C^{-1} \exp\left( - t^{-\frac{1 }{ d_\gamma-1} -\zeta } \right) \leq  \mathsf{p}_t^\gamma(z,w)   \leq   C  \exp\left( - t^{-\frac{1 }{ d_\gamma-1} + \zeta} \right) ,\quad \forall t > 0 .
\eqe
\end{thm}
 
We will not directly use the Liouville heat kernel, so we do not say anything further about it here and instead refer the interested reader to~\cite{grv-heat-kernel,mrvz-heat-kernel,andres-heat-kernel,dzz-heat-kernel} for additional background. 

The main contributions of the present paper are to prove monotonicity and bounds for the exponent $d_\gamma$ of Theorem~\ref{thm-dzz} and to prove that this exponent also describes distances with respect to LFPP and in certain random planar maps. 
\bigskip

\noindent\textbf{Acknowledgments.} We thank an anonymous referee for helpful comments on an earlier version of this paper. We thank Subhajit Goswami, Nina Holden, Josh Pfeffer, and Xin Sun for helpful discussions. J. Ding was supported in part by the NSF Grant DMS-1757479 and an Alfred Sloan fellowship.

\subsection{Main results}
\label{sec-results}

Let $d_\gamma$ be as in Theorem~\ref{thm-dzz}. We first record the properties which we prove are satisfied by $d_\gamma$. 

\begin{thm}[Monotonicity and bounds for $d_\gamma$] \label{thm-d}  
The fractal dimension $d_\gamma$ is a strictly increasing, locally Lipschitz continuous function of $\gamma \in (0,2)$ and satisfies 
\eqb \label{eqn-d-bound}
\ul d_\gamma \leq d_\gamma \leq \ol d_\gamma 
\eqe
for
\eqb \label{eqn-d-lower}
\ul d_\gamma := 
\begin{dcases}
\max\left\{ \sqrt 6 \gamma ,  \frac{2\gamma^2}{4+\gamma^2-\sqrt{16 +\gamma^4}}         \right\} ,\quad &\gamma \leq \sqrt{8/3}  \\
\frac13 \left( 4 + \gamma^2 +\sqrt{16 + 2 \gamma^2 + \gamma^4} \right) ,\quad &\gamma \geq \sqrt{8/3} 
\end{dcases}
\eqe 
and
\eqb  \label{eqn-d-upper}
\ol d_\gamma := 
\begin{dcases}
\min\left\{    \frac13 \left( 4 + \gamma^2 +\sqrt{16 + 2 \gamma^2 + \gamma^4} \right)  ,  2 + \frac{\gamma^2}{2} + \sqrt 2 \gamma \right\} ,\quad &\gamma \leq \sqrt{8/3}  \\
\sqrt 6 \gamma ,\quad &\gamma \geq \sqrt{8/3} 
\end{dcases} .
\eqe 
\end{thm}

Figure~\ref{fig-d-bound} shows graphs of our upper and lower bounds for $d_\gamma$. Figure~\ref{fig-exponent-table} shows a table of the upper and lower bounds for several special values of $\gamma$. 

Our upper and lower bounds match only for $\gamma = 0$ and $\gamma =\sqrt{8/3}$, in which case $d_{\sqrt{8/3}} = 4$. The fact that $d_{\sqrt{8/3}}=4$ is a new result in the setting of Theorem~\ref{thm-dzz}. In particular, we now know that the Liouville heat kernel exponent for $\gamma=\sqrt{8/3}$ is $1/3$. 

The bounds~\eqref{eqn-d-bound} are the best currently known for $d_\gamma$ except in the case of the lower bound when $\gamma$ is very small (see also Section~\ref{sec-related}).\footnote{Since this paper was posted to the arXiv, new bounds for $d_\gamma$ have been obtained in~\cite{gp-lfpp-bounds} which improve on our bounds in some regimes. As in the case of our bounds, the new bounds in~\cite{gp-lfpp-bounds} are based on Theorem~\ref{thm-lfpp-compare}, the fact that $d_{\sqrt{8/3}}=4$, and a certain monotonicity statement for LFPP.}
 In this latter regime, one gets from~\cite[Theorem 1.2]{ding-goswami-watabiki} that there is a universal constant $c>0$ such that for small enough $\gamma > 0$, 
\eqb \label{eqn-d-asymp}
d_\gamma \geq 2 +c \frac{\gamma^{4/3} }{\log\gamma^{-1}}  . 
\eqe
This is not implied by~\eqref{eqn-d-bound} since $\ul d_\gamma$ behaves like $ 2 + O_\gamma(\gamma^2)$ as $\gamma\rta 0^+$.
We will discuss the source of our bounds for $d_\gamma$ and their implications further in Section~\ref{sec-bounds}.

\begin{figure}[t!]
 \begin{center}
\includegraphics[scale=.6]{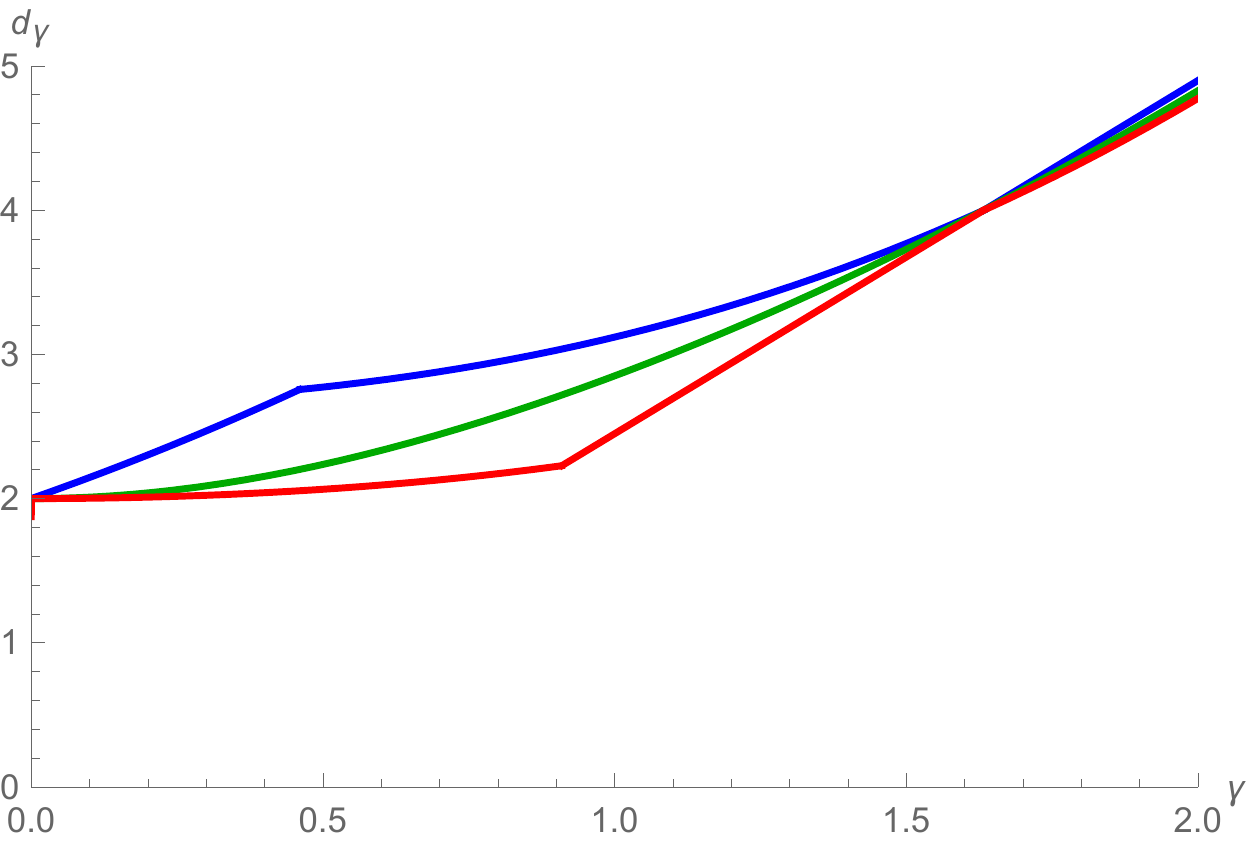} \hspace{15pt} \includegraphics[scale=.6]{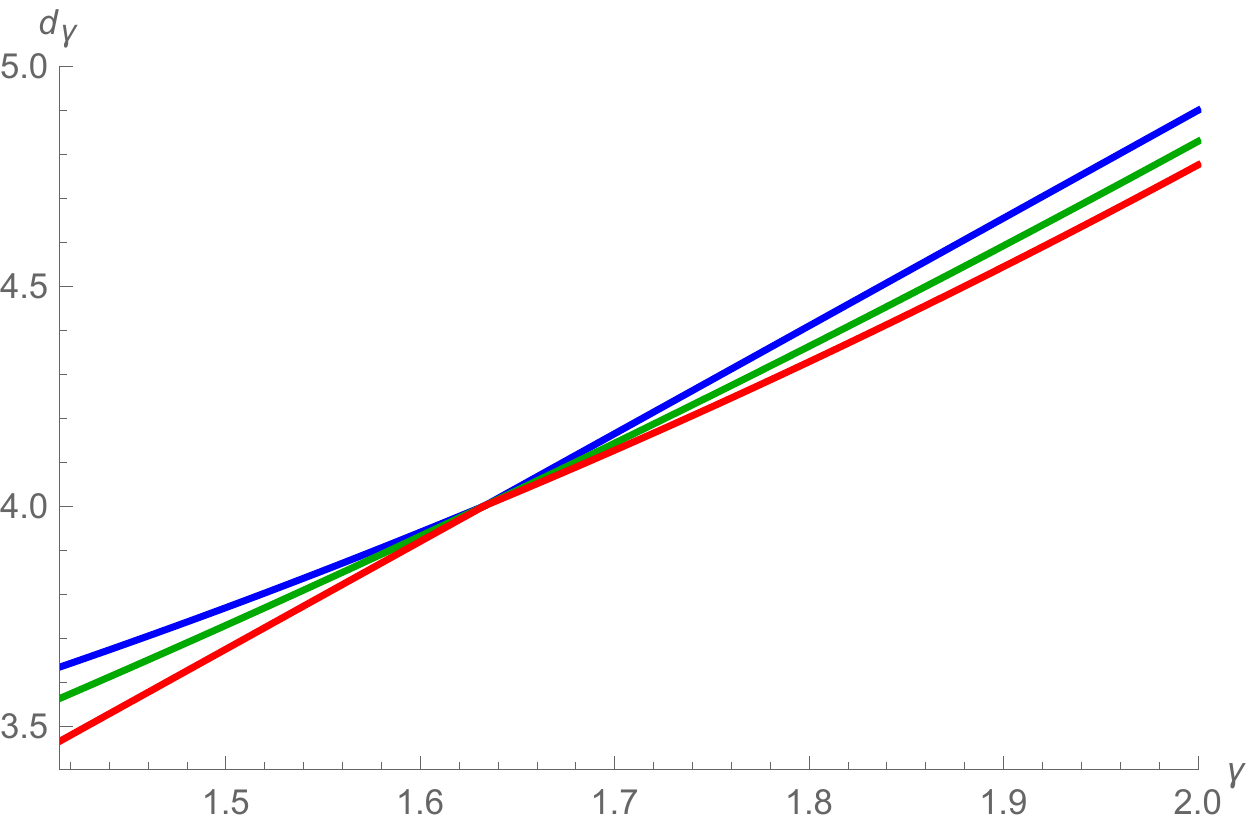}
\vspace{-0.01\textheight}
\caption{ \textbf{Left.} Graph of the lower bound $\ul d_\gamma$ (red) and the upper bound $\ol d_\gamma$ (blue) from Theorem~\ref{thm-d} together with the Watabiki prediction $d_\gamma^{\op{Wat}}$ from~\eqref{eqn-watabiki} (green). Note that the bounds $\ul d_\gamma \leq d_\gamma \leq \ol d_\gamma$ are consistent with the Watabiki prediction but the bound~\eqref{eqn-d-asymp} for the asymptotics as $\gamma\rta 0$ is not. The red and blue curves meet at $(\sqrt{8/3},4)$. The ``kink" in the red curve occurs at approximately  $( 0.909576, 2.228)$ and the ``kink" in the blue curve occurs at approximately $( 0.460149, 2.75662)$. 
\textbf{Right.} Graph of the same functions but restricted to the interval $[\sqrt 2 ,2]$. Graphs were produced using Mathematica.
}\label{fig-d-bound}
\end{center}
\vspace{-1em}
\end{figure}

\begin{figure}[t!]
 \begin{center}
\includegraphics[scale=1]{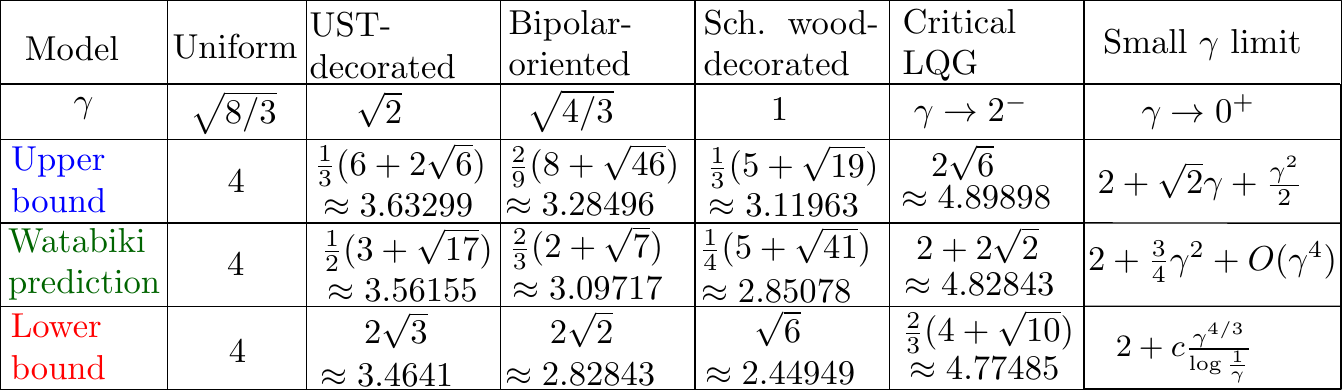}  
\vspace{-0.01\textheight}
\caption{ Table of known upper and lower bounds for $d_\gamma$ and the Watabiki prediction for several special values of $\gamma$. We emphasize that we do not treat the critical case $\gamma=2$ in this paper: the bounds shown in the table for critical LQG are bounds for $\lim_{\gamma\rta 2^-} d_\gamma$. The lower bound for the asymptotics as $\gamma\rta 0^+$ is the only place where known bounds are inconsistent with the Watabiki prediction. 
}\label{fig-exponent-table}
\end{center}
\vspace{-1em}
\end{figure}

We now express several other quantities in terms of $d_\gamma$. 
We start with a result to the effect that the exponent $d_\gamma$ describes not only point-to-point distances but also diameters and distances between sets. We can also require that the paths used in the definition of Liouville graph distance stay in a fixed open set. 

\begin{defn}[Restricted Liouville graph distance and LFPP] \label{def-restricted-lgd}
For a GFF-type distribution $h$ on $\mcl D\subset\BB C$, a domain $U\subset \mcl D$, $z,w\in U$, and $\ep > 0$, we define the \emph{restricted Liouville graph distance} $D_h^\ep(z,w;U)$ to be the smallest $N\in\BB N$ for which there is a collection of $N$ Euclidean balls \emph{contained in $\ol U$} which have $\gamma$-LQG mass at most $\ep$ with respect to $h$ and whose union contains a continuous path from $z$ to $w$. We similarly define the \emph{restricted LFPP distance} $D_{h,\LFPP}^\delta(z,w; U)$ for $\delta>0$ to be the infimum over all piecewise continuously differentiable paths $P : [0,T] \rta \ol U$ of the quantity $\int_0^T e^{\xi h_\delta(P(t))} |P'(t)| \,dt$. 
For $A,B\subset U$, we also define 
\eqb \label{eqn-lqg-set}
D_h^\ep(A,B ; U) := \min_{z\in A , w\in B} D_{h}^\ep(z,w ; U)  
\quad \op{and} \quad
D_{h,\LFPP}^\delta(A,B ; U) := \min_{z\in A , w\in B} D_{h,\LFPP}(z,w ; U) .
\eqe  
\end{defn}
 
To avoid unnecessary technicalities related to the boundary, in what follows (and throughout most of our proofs) we will consider the case when $D = \BB C$ and $h$ is a whole-plane GFF on $\BB C$ normalized so that its circle average over $\bdy\BB D$ is 0 (here and throughout the paper $\BB D$ denotes the open Euclidean unit disk).
It is easy to compare other variants of the GFF to $h$ away from the boundary of their respective domains using local absolute continuity; see Lemma~\ref{lem-whole-plane-compare}. 

\begin{thm}[Bounds for Liouville graph distance] \label{thm-diam}
Let $h$ be a whole-plane GFF normalized so that its circle average over $\bdy\BB D$ is zero. 
For each $z, w \in \BB C$, almost surely
\eqb \label{eqn-d-def}
\lim_{\ep\rta 0} \frac{\log D_h^\ep(z,w)}{\log\ep^{-1}} = \frac{1}{d_\gamma} .
\eqe
Furthermore, for each open set $U\subset\BB C$ and each compact connected set $K\subset U$, almost surely
\eqb \label{eqn-d-diam}
\lim_{\ep\rta 0} \frac{ \log \max_{z,w\in K} D_h^\ep\left( z , w ; U \right)    }{ \log \ep^{-1}   }
 = \lim_{\ep\rta 0} \frac{ \log   D_h^\ep\left( K , \bdy U \right)    }{ \log \ep^{-1}   }
 = \frac{1}{d_\gamma} . 
\eqe 
\end{thm}

The main difficulty in the proof of Theorem~\ref{thm-diam} is relating diameters and point-to-point distances. This is carried out in Section~\ref{sec-diam}. 
The convergence~\eqref{eqn-d-def} follows easily from the definition of $d_\gamma$ in Theorem~\ref{thm-dzz} and the relationship between the whole-plane and zero-boundary GFFs. The second convergence in~\eqref{eqn-d-diam} is also a relatively straightforward consequence of results from~\cite{dzz-heat-kernel}. 
 
Our next result says that distances with respect to the Liouville first passage percolation metric $D_{h,\LFPP}^\delta$ for $\delta>0$ can also be described in terms of $\gamma$ and $d_\gamma$. 
 
\begin{thm}[Bounds for Liouville first passage percolation] \label{thm-lfpp-compare}
Let $\gamma \in (0,2)$ and let $D_{h,\LFPP}^\delta$ for $\delta>0$ denote the LFPP distance with parameter $\xi = \gamma/d_\gamma$, for $h$ a whole-plane GFF normalized as above. 
For each pair of distinct points $z,w\in \BB C$, it holds with probability tending to 1 as $\delta \rta 0$ that 
\eqb  \label{eqn-lfpp-compare}
 D_{h,\LFPP}^\delta(z,w) =  \delta^{1 - \frac{2}{d_\gamma} - \frac{\gamma^2}{2 d_\gamma}  + o_\delta(1) } .
\eqe
Furthermore, for each open set $U\subset \BB C$ and each compact set $K\subset U$, it holds with probability tending to 1 as $\delta \rta 0$ that
\eqb \label{eqn-lfpp-diam}
 \max_{z,w\in K} D_{h,\LFPP}^\delta\left( z , w ; U \right) = \delta^{1 - \frac{2}{d_\gamma} - \frac{\gamma^2}{2 d_\gamma}  + o_\delta(1) } 
 \quad\op{and} \quad   D_{h,\LFPP}^\delta\left( K , \bdy U \right) = \delta^{1 - \frac{2}{d_\gamma} - \frac{\gamma^2}{2 d_\gamma}  + o_\delta(1) }  .
\eqe 
\end{thm}

See Section~\ref{sec-lfpp-exponent} a one-page heuristic explanation (using scaling properties of $\gamma$-LQG) of the choice $\xi = \gamma/d_\gamma$ and the exponent appearing in~\eqref{eqn-lfpp-compare}.
It was pointed out to us by R{\'e}mi Rhodes and Vincent Vargas that the relation $\xi = \gamma/d_\gamma$ is consistent with the physics literature, see, e.g.~\cite{watabiki-lqg}. 

We will prove slightly more quantitative variants of Theorems~\ref{thm-diam} and~\ref{thm-lfpp-compare} below, which give polynomial bounds on the rate of convergence of probabilities.

We also show that $d_\gamma$ describes distances in certain random planar maps. 
Consider the following infinite-volume random rooted planar maps $ (M, \BB v )  $, each equipped with its standard root vertex. In each case, the corresponding $\gamma$-LQG universality class is indicated in parentheses. 
\begin{enumerate}
\item The \emph{uniform infinite planar triangulation} (UIPT) of type II, which is the local limit of uniform triangulations with no self-loops, but multiple edges allowed~\cite{angel-schramm-uipt} ($\gamma=\sqrt{8/3}$). 
\item The \emph{uniform infinite spanning-tree decorated planar map}, which is the local limit of random spanning-tree weighted planar maps~\cite{shef-burger,chen-fk} ($\gamma = \sqrt 2$).
\item The \emph{uniform infinite bipolar oriented planar map}, as constructed in~\cite{kmsw-bipolar}\footnote{See~\cite[Section~3.3]{ghs-map-dist} for a careful proof that the infinite-volume bipolar-oriented planar maps considered in this paper exist as Benjamini-Schramm~\cite{benjamini-schramm-topology} limits of finite bipolar-oriented maps.} ($\gamma = \sqrt{4/3}$). 
\item More generally, one of the other distributions on infinite bipolar-oriented maps considered in~\cite[Section~2.3]{kmsw-bipolar} for which the face degree distribution has an exponential tail and the correlation between the coordinates of the encoding walk is $-\cos(\pi\gamma^2/4)$ (e.g., an infinite bipolar-oriented $k$-angulation for $k\geq 3$ --- in which case $\gamma=\sqrt{4/3}$ --- or one of the bipolar-oriented maps with biased face degree distributions considered in~\cite[Remark~1]{kmsw-bipolar} (see also~\cite[Section 3.3.4]{ghs-map-dist}), for which $\gamma \in (0,\sqrt 2)$).
\item The \emph{uniform infinite Schnyder-wood decorated triangulation}, as constructed in~\cite{lsw-schnyder-wood} ($\gamma = 1$).
\item The $\gamma$-mated-CRT map for $\gamma \in (0,2)$, as defined in Section~\ref{sec-planar-map-discussion}. 
\end{enumerate}

\begin{thm}[Ball volume exponent for random planar maps] \label{thm-ball-size}
Let $(\mcl M , \BB v)$ be any one of the above six rooted random planar maps and let $\gamma$ be the corresponding LQG parameter. For $r\in\BB N$, let $\mcl B_r^{\mcl M}(\BB v)$ be the graph distance ball of radius $r$ centered at $\BB v$ (i.e., the set of vertices lying at graph distance at most $r$ from $\BB v$) and write $\#\mcl B_r^{\mcl M}(\BB v)$ for its cardinality. Almost surely,
\eqb \label{eqn-ball-size}
\lim_{r \rta \infty} \frac{ \log \#\mcl B_r^{\mcl M}(\BB v)}{\log r} = d_\gamma .
\eqe  
\end{thm}

Theorem~\ref{thm-ball-size} is proven using the SLE/LQG representation of the mated-CRT map~\cite{wedges} together with the strong coupling between the mated-CRT map and other random planar maps~\cite{ghs-map-dist}. See Section~\ref{sec-planar-map-discussion} for more details. 

Building on Theorem~\ref{thm-ball-size} and the lower bound for the displacement of the random walk on $\mcl M$ from~\cite{gm-spec-dim}, it is shown in~\cite{gh-displacement} that the graph distance traveled by a simple random walk on $\mcl M$ run for $n$ steps is typically of order $n^{1/d_\gamma + o_n(1)}$. Since we know that $d_\gamma > 2$, this implies in particular that the simple random walk on each of the above maps is subdiffusive and that the subdiffusivity exponent is the reciprocal of the ball volume exponent.

We note that subdiffusivity in the case of the UIPT/UIPQ, with a non-optimal exponent, was previously established by Benjamini and Curien~\cite{benjamini-curien-uipq-walk}. Also, Theorem~\ref{thm-ball-size} combined with a recent result of Lee~\cite[Theorem 1.9]{lee-conformal-growth} implies subdiffusivity with the non-optimal exponent $1/(d_\gamma-1)$ in the case when $d_\gamma > 3$ (by Theorem~\ref{thm-d} this is the case for $\gamma > \sqrt{3/2}$).

\subsection{Discussion of bounds for $d_\gamma$} 
\label{sec-bounds}

As we will see in Section~\ref{sec-mono}, our bounds~\eqref{eqn-d-bound} for $d_\gamma$ turn out to be almost immediate consequences of the relationships between exponents from our other results. Indeed, our result for Liouville first passage percolation (Theorem~\ref{thm-lfpp-compare}) allows us to deduce that certain functions of $\gamma$ and $d_\gamma$ are increasing in $\gamma$. In particular, we have the following, which will be proven (via a two-page argument) in Section~\ref{sec-mono}. 
 
\begin{prop} \label{prop-increase}
The function 
\eqb \label{eqn-increase-exponent}
\gamma \mapsto \frac{\gamma}{d_\gamma} 
\eqe 
is strictly increasing on $(0,2)$ and the function 
\eqb \label{eqn-increase-dist}
\gamma \mapsto 1-\frac{2}{d_\gamma} - \frac{\gamma^2}{2d_\gamma} + \frac{\gamma^2}{2d_\gamma^2}
\eqe
is non-decreasing on $(0,2)$. 
\end{prop}

Theorem~\ref{thm-ball-size} together with known results for uniform triangulations~\cite{angel-peeling} shows that $d_{\sqrt{8/3}} = 4$. Combining this with Proposition~\ref{prop-increase} will yield the bounds~\eqref{eqn-d-bound} except in the case of small values of $\gamma$, in which case the bounds for the mated-CRT map obtained in~\cite[Theorem 1.10]{ghs-dist-exponent} are sharper than those obtained via monotonicity. This is the reason for the max and the min in the formulas for $\ul d_\gamma$ and $\ol d_\gamma$ in Theorem~\ref{thm-d}. We note that the lower bound for $d_\gamma$ in the small-$\gamma$ regime comes from the KPZ formula~\cite{shef-kpz} and coincides with the lower bound for $d_\gamma$ from~\cite{dzz-heat-kernel}. 
The monotonicity of $d_\gamma$ follows easily from the monotonicity of~\eqref{eqn-increase-dist} (Proposition~\ref{prop-d-mono}). 

We emphasize that the proof of our bounds for $d_\gamma$ relies crucially on the relationships between exponents. The monotonicity statements of Proposition~\ref{prop-increase} are not at all clear from the perspective of random planar maps, Liouville graph distance, and/or the Liouville heat kernel. Likewise, we do not have a direct proof that $d_{\sqrt{8/3}} = 4$ without using the theory of uniform random planar maps (the $\sqrt{8/3}$-LQG metric in~\cite{lqg-tbm1,lqg-tbm2,lqg-tbm3} is constructed in a rather indirect way which does not use Liouville graph distance or LFPP). 

If one could compute $d_{\gamma_0}$ for some $\gamma_0 \in (0,2) \setminus \{\sqrt{8/3}\}$, e.g., if one could find the volume growth exponent for metric balls in a spanning-tree weighted map (which we know is equal to $d_{\sqrt 2}$), then one could plug this into Proposition~\ref{prop-increase} to get improved bounds for $d_\gamma$ in some non-trivial interval of $\gamma$-values. 

Our results are contrary to certain predictions for the fractal dimension of $\gamma$-LQG from the physics literature. Let us first note that some physics articles have argued that the fractal dimension of $\gamma$-LQG satisfies $d_\gamma = 4$ for all $\gamma \in [\sqrt 2, 2)$ (which corresponds to central charge between $-2$ and 1); see, e.g.,~\cite{ajw-lqg-fractal,duplantier-d=4}.  
This paper is the first rigorous work to contradict this prediction: the bounds~\eqref{eqn-d-bound} show that $d_\gamma < 4$ for $\gamma \in (0,\sqrt{8/3})$ and $d_\gamma > 4$ for $\gamma \in (\sqrt{8/3}, 2)$.

The best-known prediction for the fractal dimension of $\gamma$-LQG is due to Watabiki~\cite{watabiki-lqg}, who predicted that this dimension is given by
\eqb \label{eqn-watabiki}
d_\gamma^{\op{Wat}} = 1 + \frac{\gamma^2}{4} + \frac14 \sqrt{(4+\gamma^2)^2 + 16\gamma^2} .
\eqe 
The bounds~\eqref{eqn-d-bound} are consistent with~\eqref{eqn-watabiki}, but the asymptotics~\eqref{eqn-d-asymp} as $\gamma\rta 0$ obtained in~\cite{ding-goswami-watabiki} are not. Indeed,~\eqref{eqn-watabiki} gives $d_\gamma^{\op{Wat}} = 2 + O (\gamma^2)$ as $\gamma\rta 0^+$. Theorem~\ref{thm-ball-size} shows that one has this same contradiction to Watabiki's prediction for small values of $\gamma$ for the ball volume exponent for certain random planar map models, and the results of~\cite{dzz-heat-kernel} (Theorem~\ref{thm-dzz}) shows that one also has the analogous contradiction for the Liouville heat kernel exponent. Taken together, this appears to be rather conclusive evidence that the Watabiki prediction is not correct for small values of $\gamma$. 

However, Watabiki's prediction appears to match up closely with numerical simulations (see, e.g.,~\cite{ambjorn-budd-lqg-dist}) and lies between our upper and lower bounds for $d_\gamma$ in~\eqref{eqn-d-bound}. This suggests that the true value of $d_\gamma$ should be numerically  close to $d_\gamma^{\op{Wat}}$. Since the known contradictions to Watabiki's prediction only hold for small values of $\gamma$, one possibility is that there is a $\gamma_* \in (0,2)$ such that $d_\gamma = d_\gamma^{\op{Wat}}$ for $\gamma \in [\gamma_* , 2)$ but not for $\gamma \in (0,\gamma_*)$. This would mean that $d_\gamma$ is not an analytic function of $\gamma$. Another possibility is that $d_\gamma$ is given by some other formula which is numerically close to $d_\gamma^{\op{Wat}}$. For example, all of our presently known results are consistent with $d_\gamma = d_\gamma^{\op{Quad}}$ for 
\eqb \label{eqn-quadratic}
d_\gamma^{\op{Quad}} = 2 + \frac{\gamma^2}{2} + \frac{\gamma}{\sqrt 6} ,
\eqe
although we have no theoretical reason to believe that this is actually the case. (The formula~\eqref{eqn-quadratic} was obtained by choosing a quadratic function of $\gamma$ which satisfies $d_0^{\op{Quad}} = 2$, $d_{\sqrt{8/3}}^{\op{Quad}} = 4$, and which has the simplest possible coefficients). 

\subsection{Discussion of random planar map connection}
\label{sec-planar-map-discussion}
 
The connection between Liouville graph distance and random planar maps (and thereby Theorem~\ref{thm-ball-size} and the fact that $d_{\sqrt{8/3}}=4$) comes by way of a one-parameter family of random planar maps called \emph{mated-CRT maps}. To define these maps, let $\gamma\in(0,2)$ and let $(L,R)$ be a pair of correlated, two-sided Brownian motions with correlation $-\cos(\pi\gamma^2/4)$ (the reason for the strange correlation parameter is that this makes it so that $\gamma$ is the LQG parameter). For $\ep > 0$, the \emph{$\gamma$-mated CRT map} associated with $(L,R)$ with increment size $\ep$ is the planar map whose vertex set is $\ep\BB Z$, with two such vertices $x_1 , x_2\in\ep\BB Z$ with $x_1<x_2$ connected by an edge if and only if
\eqb \label{eqn-inf-adjacency}
\left( \inf_{t\in [x_1 - \ep , x_1]} L_t \right) \vee \left( \inf_{t\in [x_2 - \ep , x_2]} L_t \right) 
\leq  \inf_{t\in [x_1 , x_2-\ep]} L_t ,
\eqe
or the same is true with $R$ in place of $L$. The vertices are connected by two edges if~\eqref{eqn-inf-adjacency} holds for both $L$ and $R$ but $|x_2-x_1| > \ep$. See Figure~\ref{fig-mated-crt-map}, left, for a more geometric definition of the mated-CRT map and an explanation of its planar map structure. We note that Brownian scaling shows that the law of $\mcl G^\ep$ as a planar map does not depend on $\ep$, but it will be convenient for our purposes to consider different values of $\ep$ for reasons which will become apparent just below.

\begin{figure}[t!]
 \begin{center}
\includegraphics[scale=.65]{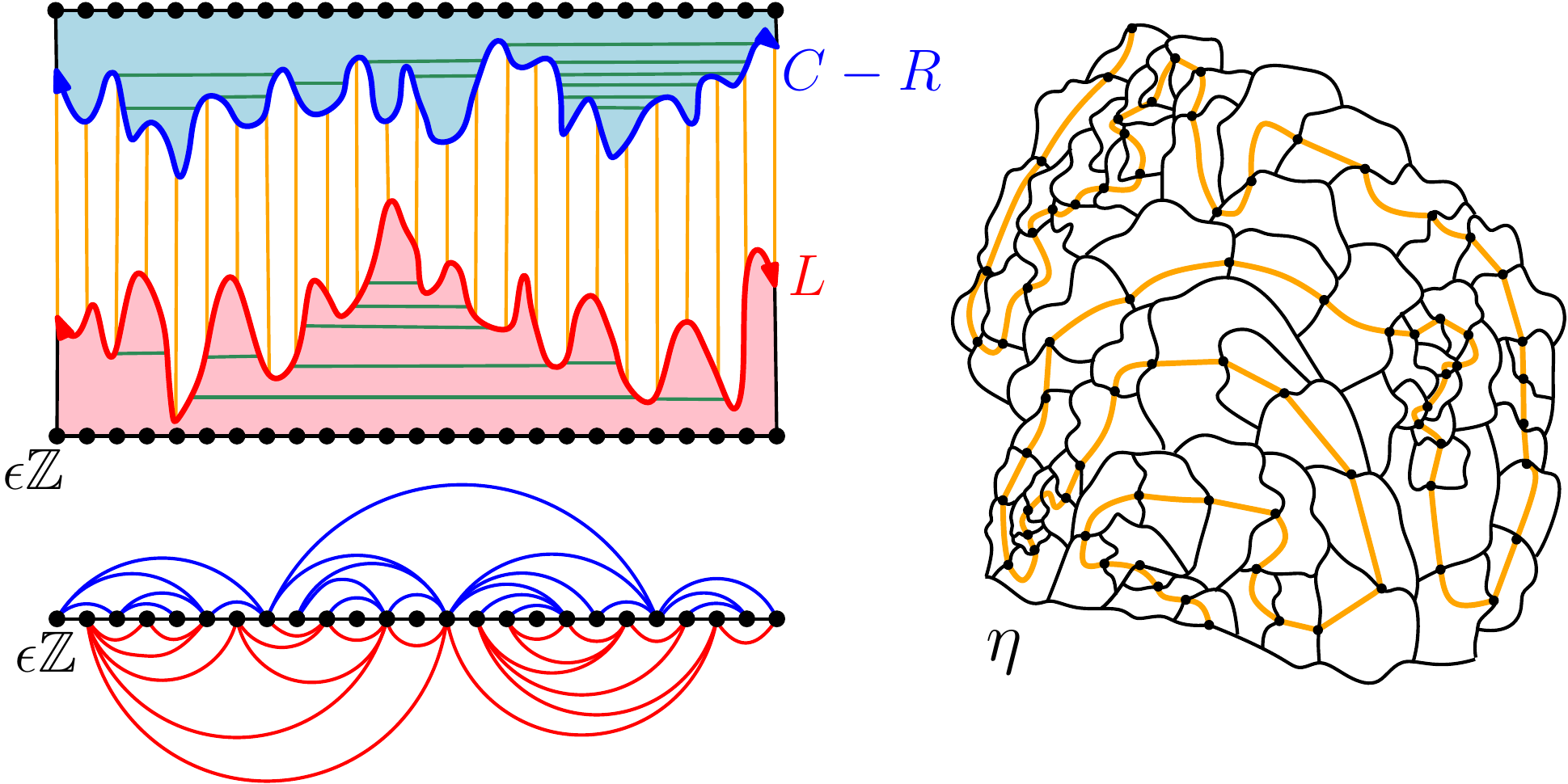}
\vspace{-0.01\textheight}
\caption{\textbf{Top Left.} To construct the mated-CRT map $\mcl G^\ep$ geometrically, draw the graph of $L$ (red) and the graph of $C-R$ (blue) for some large constant $C > 0$ chosen so that the parts of the graphs over some time interval of interest do not intersect. Then divide the region between the graphs into vertical strips (boundaries shown in orange) and identify each strip with the horizontal coordinate $x\in \ep \BB Z$ of its rightmost point. Vertices $x_1,x_2\in \ep \BB Z$ are connected by an edge if and only if the corresponding strips are connected by a horizontal line segment which lies under the graph of $L$ or above the graph of $C-R$. One such segment is shown in green in the figure for each pair of adjacent vertices.
\textbf{Bottom Left.} One can draw the graph $\mcl G^\ep $ in the plane by connecting two vertices $x_1,x_2 \in  \ep\BB Z$ by an arc above (resp.\ below) the real line if the corresponding strips are connected by a horizontal segment above (resp.\ below) the graph of $L$ (resp.\ $C-R$), and connecting each pair of consecutive vertices of $ \ep\BB Z$ by an edge. This gives $\mcl G^\ep$ a planar map structure under which it is a triangulation. 
\textbf{Right.} The mated-CRT map can be realized as the adjacency graph of \emph{cells} $\eta([x-\ep,x])$ for $x\in \ep\BB Z$, where $\eta$ is a space-filling SLE$_\kappa$ for $\kappa=16/\gamma^2$ parametrized by $\gamma$-LQG mass with respect to an independent $\gamma$-quantum cone. Here, the cells are outlined in black and the order in which they are hit by the curve is shown in orange. The three pictures do not correspond to the same mated-CRT map realization. 
Similar figures have appeared in~\cite{ghs-map-dist,gm-spec-dim,gh-displacement}. 
}\label{fig-mated-crt-map}
\end{center}
\vspace{-1em}
\end{figure}

There is a deep connection between mated-CRT maps and Liouville quantum gravity decorated by Schramm-Loewner evolution~\cite{schramm0} curves due to Duplantier, Miller, and Sheffield~\cite{wedges}, which is illustrated in Figure~\ref{fig-mated-crt-map}, right. 
We briefly review this connection here and refer to Section~\ref{sec-lqg-prelim} for a more detailed overview and a review of the definitions of the objects involved. 
Let $h$ be the variant of the whole-plane Gaussian free field corresponding to a so-called \emph{$\gamma$-quantum cone}, which can (roughly speaking) be thought of as describing the local behavior of a GFF-type distribution near a typical point sampled from its $\gamma$-LQG measure.
Independently from $h$, sample a whole-plane space-filling SLE$_\kappa$ curve $\eta$ from $\infty$ to $\infty$ with parameter $\kappa =16/\gamma^2 > 4$ --- this is just ordinary SLE$_\kappa$ for $\kappa\geq 8$ and for $\kappa \in (4,8)$ is obtained from ordinary SLE$_\kappa$ by iteratively filling in the ``bubbles" formed by the curve to get a space-filling curve. We then parametrize $\eta$ by $\gamma$-LQG mass with respect to $h$, so that $\eta(0) =0$ and $\mu_h(\eta([t_1,t_2])) = t_2-t_1$ for each $t_1<t_2$. 

It follows from~\cite[Theorem 1.9]{wedges} that for $\ep>0$, the adjacency graph of $\mu_h$-mass $\ep$ \emph{cells} $\eta([x-\ep,x])$ for $x\in\ep\BB Z$, with two cells considered to be adjacent if they share a non-trivial connected boundary arc, has exactly the same law as the $\gamma$-mated CRT map $\mcl G^\ep$. In other words, the distance from $x$ to $y$ in $\mcl G^\ep$ differs from the smallest $N\in\BB N$ for which there exists a Euclidean path from $\eta(x)$ to $\eta(y)$ which can be covered by $N$ of the cells $\eta([x-\ep,x])$ by at most a deterministic constant factor (depending on the maximal number of cells which can intersect at a single point). 

This gives us a representation of distances in the mated-CRT map which looks quite similar to the definition of Liouville graph distance.
Using basic estimates for space-filling SLE~\cite{ghm-kpz}, one can show that with very high probability each of the above space-filling SLE cells which intersects $\BB D$ is ``roughly spherical" in the sense that the ratio of its diameter to the largest Euclidean ball it contains is bounded above by $\ep^{o_\ep(1)}$. 
This allows us to compare Liouville graph distances to mated-CRT map distances (Proposition~\ref{prop-sle-metric-compare}) and thereby prove Theorem~\ref{thm-ball-size} in the case of the mated-CRT map.

The mated-CRT map is also related to various combinatorial random planar maps, including the other planar maps listed just above Theorem~\ref{thm-ball-size}. 
The reason for this is that each of these other planar maps can be bijectively encoded by a random walk on $\BB Z^2$ with a certain step size distribution depending on the model via an exact discrete analogue of the construction of the mated-CRT map from Brownian motion. For example, the infinite spanning-tree weighted map corresponds to a standard nearest-neighbor random walk~\cite{mullin-maps,bernardi-maps,shef-burger} and the UIPT corresponds to a walk whose increments are i.i.d.\ uniform samples from $\{(0,1), (1,0), (-1,-1)\}$~\cite{bernardi-dfs-bijection,bhs-site-perc}. 

Using these bijections and a strong coupling result for random walk and Brownian motion~\cite{kmt,zaitsev-kmt}, it was shown in~\cite{ghs-map-dist} that one can couple each of the above random planar maps with the $\gamma$-mated-CRT map (where $\gamma$ is determined by the correlation of the coordinates of the encoding walk) in such a way that with high probability, certain large subgraphs are roughly isometric, with a polylogarithmic distortion factor for distances. This allows us to transfer Theorem~\ref{thm-ball-size} from the case of the mated-CRT map to the case of these other maps. We do not need to use the bijections mentioned above directly: rather, we will just cite results from~\cite{ghs-map-dist}. 

\subsection{Related works}
\label{sec-related}

Several other works have proven bounds for the exponents which we now know can be described in terms of $d_\gamma$. Indeed, estimates for the Liouville heat kernel are proven in~\cite{andres-heat-kernel,mrvz-heat-kernel,dzz-heat-kernel}, estimates for the volume of graph distance balls in random planar maps are procen in~\cite{ghs-dist-exponent,ghs-map-dist}, and estimates for the Liouville graph distance are proven in~\cite{ding-goswami-watabiki,dzz-heat-kernel}. The estimates which come from Theorem~\ref{thm-d} are at least as sharp as all of these estimates except in the case of the lower bound as $\gamma\rta 0$, in which case~\cite{ding-goswami-watabiki} gives a stronger bound; see also~\eqref{eqn-d-asymp}. For $\gamma > 0.909576\dots$ (resp.\ $\gamma > 0.460149\dots$), our lower (resp.\ upper) bound for $d_\gamma$ is strictly sharper than any previously known bounds. 

Although this paper proves universality across different approximations of Liouville quantum gravity, 
it is known that the exponents associated with Liouville graph distance and the Liouville heat kernel are \emph{not} universal among all log-correlated Gaussian free fields: see~\cite{ding-zhang-fpp-gff,dzz-nonuniversality}. 

There is a different notion of the dimension of $\gamma$-LQG, besides the fractal (Hausdorff) dimension, called the \emph{spectral dimension}, which is expected to be equal to 2 for all values of $\gamma$. The spectral dimension can be defined in terms of the Liouville heat kernel, in which case it was proven to be equal to 2 in~\cite{rhodes-vargas-spec-dim,andres-heat-kernel}. Alternatively, it can be defined in terms of the return probability for random walk on random planar maps, in which case it was proven to be equal to 2 for all of the planar maps considered in the present paper in~\cite{gm-spec-dim}. 

Another interesting dimension associated with Liouville quantum gravity is the Euclidean Hausdorff dimension of the geodesics.  It was shown in~\cite{ding-zhang-geodesic-dim} that the geodesic length exponent associated with discrete LFPP (which should coincide with the Euclidean dimension of continuum LQG geodesics) is strictly larger than 1 when $\gamma$ is small. We expect this should be the case for all $\gamma \in (0,2)$, but we have no predictions for what the precise dimension should be, even for $\gamma = \sqrt{8/3}$ (see~\cite[Problem 9.2]{lqg-tbm2} for some discussion in this case). 
The recent paper~\cite{gp-lfpp-bounds} proves a non-trivial upper bound for the LFPP geodesic length exponent for all $\gamma \in(0,2)$.

In addition to the quantities considered in the present paper, there are several other quantities which we expect can be described in terms of our exponent $d_\gamma$, for example the following. 
\begin{itemize}
\item \textbf{Discrete Liouville first passage percolation.} Following, e.g.,~\cite{ding-dunlap-lqg-fpp,ding-goswami-watabiki,ding-zhang-geodesic-dim}, let $h$ be a discrete GFF on $\BB Z^2$ and for $\xi > 0$ and $x,y\in\BB Z^2$ define $D_{h,\LFPP}(x,y)$ to be the minimum of $\sum_{j=0}^m e^{\xi h(x_j)}$ over all paths $x = x_0 ,\dots,x_m = y$ in $\BB Z^2$ from $x$ to $y$. We expect that if $\xi = \gamma/d_\gamma$ and $|x-y|  =n$, then with high probability\footnote{To see why this should be the case, one can take as an ansatz that discrete LFPP distances are well-approximated by continuum LFPP distances with $\ep=1$. One can then re-scale by $1/n$, so that $|x-y|/n$ is of constant order, which shows that the discrete LFPP distance from $x$ to $y$ should be similar to $n$ times the continuum LFPP distance with $\delta = 1/n$ between points at constant-order Euclidean distance, as described in Theorem~\ref{thm-lfpp-compare}.}
\eqb \label{eqn-discrete-lfpp}
D_{h,\LFPP}(x,y)  = n^{\frac{2}{d_\gamma} + \frac{\gamma^2}{2d_\gamma} + o_n(1)} .
\eqe 
\item \textbf{Dimension of subsequential limiting metrics.} It is shown in~\cite{ding-dunlap-lqg-fpp} that for small enough values of $\xi > 0$, discrete LFPP admits non-trivial subsequential limiting metrics. We expect that for $\xi  =\gamma/d_\gamma$, the Hausdorff dimension of each such subsequential limiting metric is a.s.\ equal to $d_\gamma$. 
\item \textbf{Finite random planar maps.} Let $M_n$ be a finite-volume analogue of one of the planar maps considered in Theorem~\ref{thm-ball-size} with $n$ total edges. Then we expect that the graph-distance diameter of $M_n$ is typically of order $n^{1/d_\gamma + o_n(1)}$. We also expect that the same is true if $M_n$ is allowed to have a boundary of length at most $n^{1/2}$. 
\item \textbf{Mated-CRT map distance exponent.} It is shown in~\cite[Theorem 1.12]{ghs-dist-exponent} that if $\mcl G^\ep|_{(0,1]}$ denotes the sub-graph of the mated-CRT map induced by $(0,1]\cap (\ep\BB Z)$, then the limit $\chi := \lim_{\ep\rta 0} \log \BB E[\op{diam}(\mcl G^\ep|_{(0,1]})]/\log\ep^{-1}$ exists. As in~\cite[Conjecture 1.13]{ghs-dist-exponent}, we expect that there is a $\gamma_* \in (\sqrt 2 , \sqrt{8/3}]$ for which $\chi = 1/d_\gamma $ for $\gamma \in (0,\gamma_*]$.
\end{itemize}
It is likely possible to prove each of the above statements by building on the techniques of the present paper, but we do not carry this out here.
In the special case when $\gamma = \sqrt 2$, the last two statements discussed above are resolved in~\cite{gp-dla}.

\subsection{Outline} 
\label{sec-outline}

\begin{figure}[t!]
 \begin{center}
\includegraphics[scale=.80]{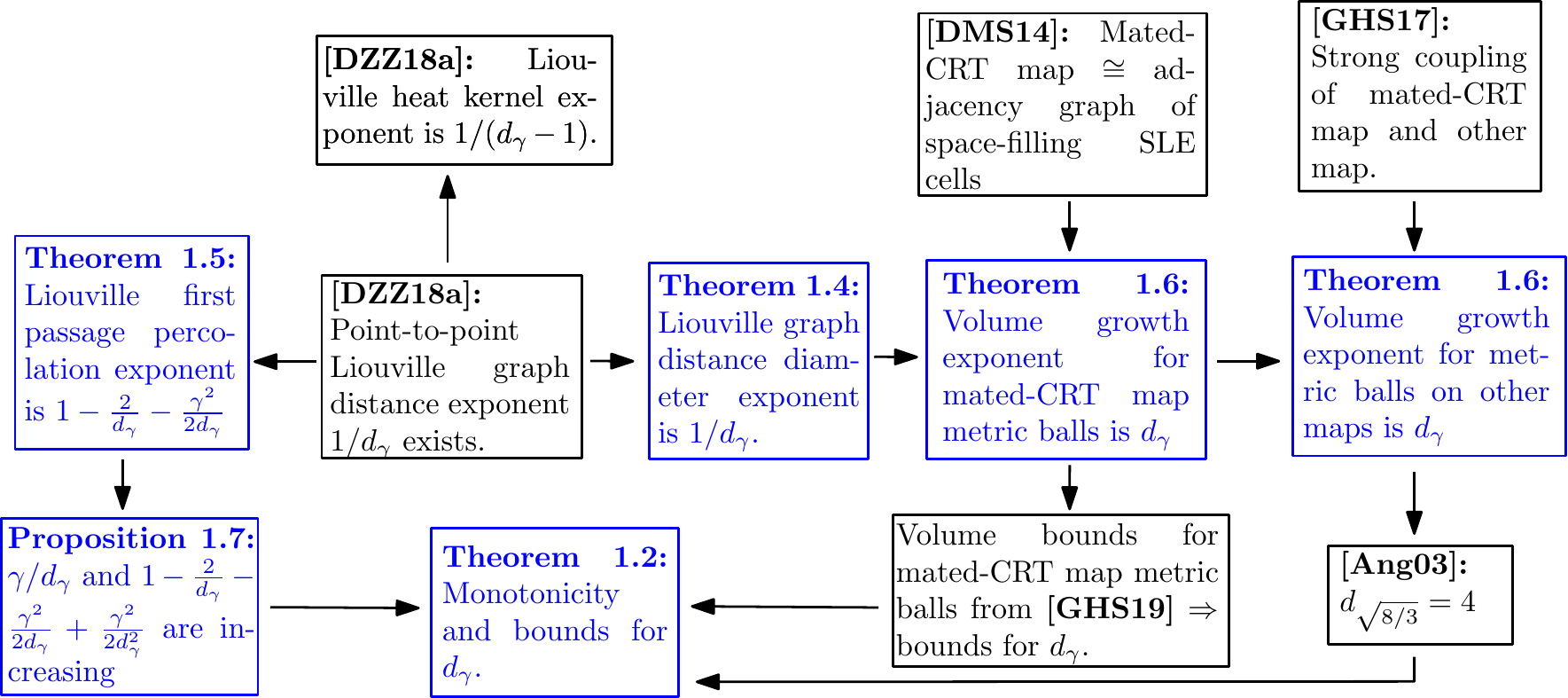}
\vspace{-0.01\textheight}
\caption{Schematic diagram of the logical relations between the results involved in this paper. Results proven in the present paper are in blue. 
}\label{fig-schematic}
\end{center}
\vspace{-1em}
\end{figure}

See Figure~\ref{fig-schematic} for a schematic diagram of how the results involved in this paper fit together. The remainder of the paper is structured as follows. 
\medskip

\noindent In \textbf{Section~\ref{sec-monotonicity}}, we first introduce some standard notation (Section~\ref{sec-notation}) and record some basic facts about the Gaussian free field which allow us to compare Liouville graph distances and LFPP defined with respect to GFF's on different domains (Section~\ref{sec-gff}). We then provide a short heuristic argument for why one should expect the relationship between Liouville graph distance and Liouville first passage percolation exponents asserted in Theorem~\ref{thm-lfpp-compare} (Section~\ref{sec-lfpp-exponent}). Finally, in Section~\ref{sec-mono} we explain why the relationships between exponents given in Theorems~\ref{thm-lfpp-compare} and~\ref{thm-ball-size} imply the properties of $d_\gamma$ asserted in Theorem~\ref{thm-d}, using the ideas discussed at the beginning of Section~\ref{sec-bounds}. 
\medskip

\noindent In \textbf{Section~\ref{sec-lfpp-compare}} we prove our theorems concerning relationships between Liouville graph distance and LFPP exponents, Theorems~\ref{thm-diam} and~\ref{thm-lfpp-compare}. We first introduce in Section~\ref{sec-wn} various approximations to the GFF defined in terms of the white noise decomposition which are in some ways easier to work with than the GFF itself. We then prove several lemmas which allow us to estimate these approximations and to compare Liouville graph distance and LFPP distances defined in terms of these approximations to distances defined in terms of the GFF. 
In Section~\ref{sec-diam}, we prove that the Liouville graph distance diameter of a fixed compact subsets of $\BB C$ is with high probability at most $\ep^{-1/d_\gamma + o_\ep(1)}$, which together with results from~\cite{dzz-heat-kernel} allows us to prove Theorem~\ref{thm-diam}. In Sections~\ref{sec-lfpp-lower} and~\ref{sec-lfpp-upper}, respectively, we prove the lower and upper bounds for LFPP distances asserted in Theorem~\ref{thm-lfpp-compare} by comparing LFPP and Liouville graph distance. See the beginnings of these subsections for outlines of the arguments involved.
\medskip

\noindent In \textbf{Section~\ref{sec-planar-map}}, we relate Liouville graph distance to distances in random planar maps and thereby prove Theorem~\ref{thm-ball-size}, using the ideas discussed in Section~\ref{sec-planar-map-discussion}. We first provide some relevant background on SLE, LQG, and their connection to the mated-CRT map (Section~\ref{sec-lqg-prelim}). We then prove a result relating several variants of Liouville graph distance, including one defined in terms of LQG-mass $\ep$ SLE cells, which we know is equivalent to the mated-CRT map (Section~\ref{sec-sle-compare}). In Section~\ref{sec-ball-size}, we use this to prove Theorem~\ref{thm-ball-size}. We first show that the diameter (in the adjacency graph) of the set of $\ep$-mass cells in the SLE/LQG representation of the mated-CRT map which intersect the Euclidean unit ball is of order $\ep^{-1/d_\gamma + o_\ep(1)}$ with high probability (Proposition~\ref{prop-crt-ball}), using the comparison results of the preceding subsection and the bounds for Liouville graph distance from Theorem~\ref{thm-diam}. We then use this to show that the volume of the graph distance ball of radius $r$ in the mated-CRT map is of order $r^{d_\gamma + o_r(1)}$ (essentially by taking $\ep = 1/r^{d_\gamma}$), and finally transfer to other planar maps using the coupling results of~\cite{ghs-map-dist}.

We emphasize that Section~\ref{sec-planar-map} is the only section of the paper which uses SLE theory. The reader does not need any knowledge of this theory to understand Section~\ref{sec-planar-map} beyond the background we provide, so long as he or she is willing to take certain results as black boxes.

\section{Preliminaries}
\label{sec-monotonicity}

\subsection{Basic notation}
\label{sec-notation}

\noindent
We write $\BB N = \{1,2,3,\dots\}$ and $\BB N_0 = \BB N \cup \{0\}$. 
\medskip

\noindent
For $a < b$, we define the discrete interval $[a,b]_{\BB Z}:= [a,b]\cap\BB Z$. 
\medskip

\noindent
If $f  :(0,\infty) \rta \BB R$ and $g : (0,\infty) \rta (0,\infty)$, we say that $f(\ep) = O_\ep(g(\ep))$ (resp.\ $f(\ep) = o_\ep(g(\ep))$) as $\ep\rta 0$ if $f(\ep)/g(\ep)$ remains bounded (resp.\ tends to zero) as $\ep\rta 0$. We similarly define $O(\cdot)$ and $o(\cdot)$ errors as a parameter goes to infinity. 
\medskip

\noindent
If $f,g : (0,\infty) \rta [0,\infty)$, we say that $f(\ep) \preceq g(\ep)$ if there is a constant $C>0$ (independent from $\ep$ and possibly from other parameters of interest) such that $f(\ep) \leq  C g(\ep)$. We write $f(\ep) \asymp g(\ep)$ if $f(\ep) \preceq g(\ep)$ and $g(\ep) \preceq f(\ep)$. 
\medskip

\noindent
Let $\{E^\ep\}_{\ep>0}$ be a one-parameter family of events. We say that $E^\ep$ occurs with
\begin{itemize}
\item \emph{polynomially high probability} as $\ep\rta 0$ if there is a $p > 0$ (independent from $\ep$ and possibly from other parameters of interest) such that  $\BB P[E^\ep] \geq 1 - O_\ep(\ep^p)$. 
\item \emph{superpolynomially high probability} as $\ep\rta 0$ if $\BB P[E^\ep] \geq 1 - O_\ep(\ep^p)$ for every $p>0$. 
\item \emph{exponentially high probability} as $\ep\rta 0$ if there exists $\lambda >0$ (independent from $\ep$ and possibly from other parameters of interest) $\BB P[E^\ep] \geq 1 - O_\ep(e^{-\lambda/\ep})$. 
\end{itemize}
We similarly define events which occur with polynomially, superpolynomially, and exponentially high probability as a parameter tends to $\infty$. 
\medskip

\noindent
We will often specify any requirements on the dependencies on rates of convergence in $O(\cdot)$ and $o(\cdot)$ errors, implicit constants in $\preceq$, etc., in the statements of lemmas/propositions/theorems, in which case we implicitly require that errors, implicit constants, etc., appearing in the proof satisfy the same dependencies.

\subsection{Gaussian free field}
\label{sec-gff} 

Here we give a brief review of the definition of the zero-boundary and whole-plane Gaussian free fields. We refer the reader to~\cite{shef-gff} and the introductory sections of~\cite{ig1,ig4} for more detailed expositions. 

For a proper open domain $U\subset \BB C$, let $\mcl H(U)$ be the Hilbert space completion of the set of smooth, compactly supported functions on $U$ with respect to the \emph{Dirichlet inner product},
\eqb \label{eqn-dirichlet}
(\phi,\psi)_\nabla = \frac{1}{2\pi} \int_U \nabla \phi(z) \cdot \nabla \psi(z) \,dz .
\eqe
In the case when $U= \BB C$, constant functions $c$ satisfy $(c,c)_\nabla = 0$, so to get a positive definite norm in this case we instead take $\mcl H(\BB C)$ to be the Hilbert space completion of the set of smooth, compactly supported functions $\phi$ on $\BB C$ with $\int_{\BB C} \phi(z) \,dz = 0$, with respect to the same inner product~\eqref{eqn-dirichlet}.  
 
The \emph{(zero-boundary) Gaussian free field} on $U$ is defined by the formal sum
\eqb \label{eqn-gff-sum}
h^U = \sum_{j=1}^\infty X_j \phi_j 
\eqe
where the $X_j$'s are i.i.d.\ standard Gaussian random variables and the $\phi_j$'s are an orthonormal basis for $\mcl H(U)$. The sum~\eqref{eqn-gff-sum} does not converge pointwise, but it is easy to see that for each fixed $\phi \in \mcl H(U)$, the formal inner product $(h^U ,\phi)_\nabla$ is a Gaussian random variable and these random variables have covariances $\BB E [(h^U,\phi)_\nabla (h^U,\psi)_\nabla] = (\phi,\psi)_\nabla$. In the case when $U \not=\BB C$ and $U$ has harmonically non-trivial boundary (i.e., a Brownian motion started from a point of $U$ a.s.\ hits $\bdy U$), one can use integration by parts (Green's identities) to define the ordinary $L^2$ inner products $(h^U,\phi) := -2\pi (h^U,\Delta^{-1}\phi)_\nabla$, where $\Delta^{-1}$ is the inverse Laplacian with zero boundary conditions, whenever $\Delta^{-1} \phi \in \mcl H(U)$. 

In the case $U=\BB C$ we typically write $h = h^{\BB C}$. In this case one can similarly define $(h ,\phi) := -2\pi(h ,\Delta^{-1}\phi)_\nabla$ where $\phi$ is the inverse Laplacian normalized so that $\int_{\BB C} \Delta^{-1} \phi(z) \, dz = 0$ (in the case $U= \BB C$). With this definition, one has $(h+c , \phi) = (h ,\phi) + (c,\phi) = (h,\phi)$ for each $\phi \in \mcl H(\BB C)$, so the whole-plane GFF is only defined modulo a global additive constant. We will typically fix this additive constant by requiring that the circle average $h_1(0)$ over $\bdy\BB D$ is zero. We refer to~\cite[Section 3.1]{shef-kpz} for more on the circle average. 

An important property of the GFF is the Markov property, which we state in the whole-plane case. If $U\subset \BB C$, then we can write $h|_U = h^U + \frk h$ where $h^U$ is a zero-boundary GFF on $U$ and $\frk h$ is an independent random harmonic function on $U$. We call $h^U$ and $\frk h$ the \emph{zero-boundary part} and \emph{harmonic part} of $h|_U$, respectively. 

The following lemma allows us to compare the approximate LQG distances associated with whole-plane GFF and the zero-boundary GFF. For the statement, we recall from Definition~\ref{def-restricted-lgd} that $D_h^\ep(z,w ; V)$ denotes the Liouville graph distance defined with respect to paths which stay in $V$, and similarly for LFPP. 

\begin{lem} \label{lem-whole-plane-compare}
Let $U \subset \BB C$ be a proper simply connected domain and let $V$ be a bounded connected domain with $\ol V\subset U$. 
Let $h$ be a whole-plane GFF normalized so that its circle average over $\bdy\BB D$ is zero.
Write $h = h^U + \frk h$ where $h^U$ is a zero-boundary GFF on $U$ and $\frk h$ is an independent random harmonic function on $U$. There are constants $a_0,a_1 > 0$ depending only on $U$ and $V$ such that for each $A>1 $,
\eqb \label{eqn-whole-plane-compare}
\BB P \left[ \max_{z\in \ol V} |\frk h(z)| \leq A \right] \geq 1-a_0 e^{-a_1 A^2} .
\eqe 
In particular, for each $\gamma \in (0,2)$, each $\ep\in (0,1)$, and each $C > 3$ it holds with probability at least $1-a_0 e^{-a_1 (\log C)^2/\gamma^2}$ that the $\gamma$-Liouville graph distance metrics satisfy
\eqb \label{eqn-whole-plane-compare-lgd}
D^{\ep/C}_h\left( z , w ; V \right) \leq   D^{ \ep}_{h^U}\left( z , w ; V \right)  \leq  D^{C\ep}_h\left( z , w ; V \right)   ,\quad \forall z,w \in V  
\eqe 
and for each $\xi  > 0$ and $\delta \in (0,1)$, it holds with probability at least $1-a_0 e^{-a_1 (\log C)^2 / \xi^2}$ that the $\xi$-LFPP metrics satisfy
\eqb \label{eqn-whole-plane-compare-lfpp}
C^{-1} D_{h,\LFPP}^\delta\left( z,w ; V \right) \leq    D_{h^U,\LFPP}^\delta\left( z , w ; V \right)  \leq C D_{h,\LFPP}^\delta\left( z , w ; V \right)   ,\quad \forall z,w \in V   .
\eqe
\end{lem}
\begin{proof} 
By, e.g.,~\cite[Lemma 6.4]{ig1}, the harmonic function $\frk h$ is a centered Gaussian random function with $\op{Var}(\frk h(z)) \leq \log \op{CR}(z;U)^{-1} + O(1)$ for each $z\in V$, where $\op{CR}(z;U)$ denotes the conformal radius and the $O(1)$ depends only on $U$. In particular, $\max_{z\in \ol V} |\frk h(z)|$ is a.s.\ finite (since $\frk h$ is harmonic, hence continuous) and $\max_{z\in \ol V} \op{Var}(\frk h(z))$ is bounded above by a constant depending only on $U$ and $V$. By the Borell-TIS inequality~\cite{borell-tis1,borell-tis2} (see, e.g.,~\cite[Theorem 2.1.1]{adler-taylor-fields}), we obtain~\eqref{eqn-whole-plane-compare} for appropriate constants $a_0,a_1 > 0$ as in the statement of the lemma (note that we absorbed $\BB E[\max_{z\in \ol V} |\frk h(z)|]$, which is finite by the Borell-TIS inequality, into the constants $a_0,a_1$). 
Since $d\mu_h = e^{\gamma \frk h} \, d\mu_{h^U}$, we obtain~\eqref{eqn-whole-plane-compare-lgd} by applying~\eqref{eqn-whole-plane-compare} with $A = \frac{1}{\gamma} \log C$. We similarly obtain~\eqref{eqn-whole-plane-compare-lfpp} by applying~\eqref{eqn-whole-plane-compare} with $A = \frac{1}{\xi} \log C$.  
\end{proof}

As an immediate consequence of Lemma~\ref{lem-whole-plane-compare}, we get that the exponent $d_\gamma$ from Theorem~\ref{thm-dzz} can equivalently be defined in the whole-plane case.

\begin{lem} \label{lem-whole-plane-d}
Let $d_\gamma$ be as in Theorem~\ref{thm-dzz}. For each connected open set $U\subset\BB C$, each distinct $z,w\in U$, and each $\zeta \in (0,1)$, it holds with polynomially high probability as $\ep\rta 0$ (at a rate which is allowed to depend on $U,z,w,\zeta$, and $\gamma$) that
\eqb \label{eqn-whole-plane-d}
\ep^{ - \frac{1}{d_\gamma + \zeta}} \leq  D_h^\ep\left( z , w  ; U  \right)  \leq \ep^{ - \frac{1}{d_\gamma-\zeta}} .
\eqe 
\end{lem}
\begin{proof}
Let $S = S_{z,w}$ be the square centered at $(z+w)/2$, with side length $2|z-w|$ and sides parallel to the segment from $z$ to $w$. Also let $S(1)$ be the square with the same center as $S$ and three times the side length and let $h^{S(1)}$ be a zero-boundary GFF on $S(1)$. If we re-scale and rotate space so that $S(1)$ is mapped to the unit square and apply~\cite[Propositions 3.17 and Lemmas 5.3 and 5.4]{dzz-heat-kernel} (see also~\cite[Remark 5.2]{dzz-heat-kernel}), we obtain that with polynomially high probability as $\ep\rta 0$, 
\eqb \label{eqn-zero-bdy-dist}
\ep^{ - \frac{1}{d_\gamma +\zeta}} \leq  D_{h^{S(1)}}^\ep\left( z , w  ; S  \right)  \leq \ep^{ - \frac{1}{d_\gamma-\zeta}} 
\quad\op{and} \quad
  D_{h^{S(1)}}^\ep\left( \{z,w\}  ,\bdy S  \right)  \geq \ep^{ - \frac{1}{d_\gamma+\zeta}} .
\eqe   
Combining~\eqref{eqn-zero-bdy-dist} with Lemma~\ref{lem-whole-plane-compare} gives the lower bound in~\eqref{eqn-whole-plane-d} for any choice of $U$ and the upper bound in~\eqref{eqn-whole-plane-d} in the case when $S\subset U$. To get the upper bound for a general choice of $U$ and $z,w\in U$, we can choose points $z = z_0,\dots,z_k = w$ in $U$ such that the square $S_{z_{j-1},z_{j }}$ is contained in $U$ for each $j=1,\dots,k$, then apply the triangle inequality. 
\end{proof}

\subsection{Heuristic derivation of the LFPP exponent}
\label{sec-lfpp-exponent}

In this subsection we provide a short heuristic explanation of why one should expect the relationship between LFPP exponents and $d_\gamma$ described in Theorem~\ref{thm-lfpp-compare}, using scaling properties which we expect to be true for the $\gamma$-LQG metric. The argument here is very different from the rigorous proof of Theorem~\ref{thm-lfpp-compare}, but the main source of the relation (the behavior of LQG distances and measures under scaling) is the same. 
Our explanation is based on the following elementary observation about possible scaling limits of LFPP distances (which we do not yet know exist).  

\begin{prop} \label{prop-lfpp-lim} 
Assume that for some $\xi   > 0$, LFPP with exponent $\xi$ converges pointwise to a metric in the scaling limit, i.e., there exists $\lambda = \lambda(\xi)\in \BB R$ such that for each random distribution $h$ on $\BB C$ whose law is locally absolutely continuous with respect to the GFF, the limit
\eqb \label{eqn-lqg-metric}
\frk d_h(z,w) = \lim_{\delta \rta 0} \delta^{-\lambda} D_{h,\LFPP}^{ \xi , \delta}(z,w) = \lim_{\delta\rta0} \delta^{-\lambda} \inf_{P\in\mcl P_{z,w}} \int_0^1 e^{\xi h_\delta(P(t))} |P'(t)| \,dt ,\quad\forall z,w\in\BB C   
\eqe
exists and defines a metric on $\BB C$, where $\mcl P_{z,w}$ is the set of all piecewise continuously differentiable paths $P$ from $z$ to $w$. 
Then for $C>0$, the limiting metric satisfies the following scaling relations:
\eqb \label{eqn-metric-add}
\frk d_{h+ \log C }(z,w) = C^\xi \frk d_h(z,w) ,\quad\forall z,w\in \BB C   
\eqe
and
\eqb \label{eqn-metric-Q}
\frk d_{h(\cdot/C) + Q \log (1/C)}(Cz,Cw) = \frk d_h(z,w)  ,\quad\forall z,w\in\BB C \quad \text{for} \quad Q = Q(\xi) = (1-\lambda)/\xi . 
\eqe
\end{prop}

We emphasize that we are very far from actually proving that~\eqref{eqn-lqg-metric} holds (although subsequential limits for a closely related metric are shown to exist when $\xi$ is small in~\cite{ding-dunlap-lqg-fpp}). 

\begin{proof}[Proof of Proposition~\ref{prop-lfpp-lim}]
The relation~\eqref{eqn-metric-add} is immediate from~\eqref{eqn-lqg-metric}.  
To derive~\eqref{eqn-metric-Q}, fix $C>0$ and write $h^C: = h(C^{-1} \cdot)$. Then
\eqb \label{eqn-circle-avg}
h_{\delta/C}(P(t)) = h_{ \delta}^C(C P(t)) .
\eqe
Moreover, $\mcl P_{Cz,Cw} = C \mcl P_{z,w}$. Therefore, applying~\eqref{eqn-lqg-metric} to $h^C$ gives
\alb
\frk d_{h^C}(C z, C w) 
&= \lim_{\delta\rta 0} \delta^{-\lambda} \inf_{P\in\mcl P_{z,w}} \int_0^1 e^{\xi h_{ \delta}^C( C P(t))} C |P'(t)| \, dt \\
&= C  \lim_{\delta\rta 0} \delta^{-\lambda} \inf_{P\in\mcl P_{z,w}} \int_0^1 e^{\xi h_{ \delta / C} (  P(t))}  |P'(t)| \, dt \quad \text{(by~\eqref{eqn-circle-avg})} \\
&= C^{1-\lambda}  \lim_{\delta' \rta 0} (\delta')^{-\lambda} \inf_{P\in\mcl P_{z,w}} \int_0^1 e^{\xi h_{ \delta'} (  P(t))}  |P'(t)| \, dt \quad \text{(by setting $\delta' = \delta/C$)} \\
&= C^{1-\lambda} \frk d_h(z,w) .
\ale
Re-arranging this gives~\eqref{eqn-metric-Q}. 
\end{proof}

We now explain why Proposition~\ref{prop-lfpp-lim} suggests the relations between exponents given in Theorem~\ref{thm-lfpp-compare}.
Indeed, suppose that for some $\xi = \xi(\gamma) >0$, the metric~\eqref{eqn-lqg-metric} is the ``correct" metric on $\gamma$-LQG (which can be described, e.g., as the one which is the scaling limit of graph distances on random planar maps).  
We will argue that 
\eqb \label{eqn-xi-lambda}
\xi = \frac{\gamma}{d_\gamma} \quad \op{and} \quad \lambda = 1 - \frac{2}{d_\gamma} - \frac{\gamma^2}{2d_\gamma}  . 
\eqe 
Indeed, if (as expected) $d_\gamma$ is the Hausdorff dimension of $\gamma$-LQG, then scaling LQG areas by $A > 0$ should correspond to scaling LQG distances by $A^{1/d_\gamma}$. The former is the same as adding $\frac{1}{\gamma} \log A$ to $h$, so by~\eqref{eqn-metric-add} we get $\frk d_{h + \gamma^{-1} \log A}(z,w) = A^{\xi/\gamma} \frk d_h(z,w)$. Hence we should have $\xi = \gamma/d_\gamma$. 

To see why the formula for $\lambda$ in~\eqref{eqn-xi-lambda} should hold, we recall the scaling relation for the $\gamma$-LQG measure $\mu_h$~\cite[Proposition 2.1]{shef-kpz}, which says that 
\eqbn
\mu_{h(\cdot/C) + Q\log(1/C)}(C X) =  \mu_h(X) \quad \text{for} \quad Q = \frac{2}{\gamma} + \frac{\gamma}{2} .
\eqen
We expect that the $\gamma$-LQG metric satisfies an analogous scaling relation, with the same value of $Q$. From~\eqref{eqn-metric-Q}, we therefore have $2/\gamma + \gamma/2 = (1-\lambda)/\xi$. Setting $\xi = \gamma/d_\gamma$ and re-arranging gives the formula for $\lambda$ in~\eqref{eqn-xi-lambda}.

\begin{remark}[$\xi > 2/d_2$ and $c >1$]
Proposition~\ref{prop-lfpp-lim} is true for any $\xi  > 0$, not just for the values $\xi\in (0,2/d_2)$ which are related to $\gamma$-LQG for $\gamma\in (0,2)$. 
It is proven in~\cite[Lemma 4.1]{gp-lfpp-bounds} that in the notation of~\eqref{eqn-metric-Q} one has $Q(\xi) = (1-\lambda)/\xi \in [0,2)$ whenever $\xi > 2/d_2$ (we know $Q(\gamma /d_\gamma) = 2/\gamma +\gamma/2 >2$ for $\gamma \in (0,2)$ by Theorem~\ref{thm-lfpp-compare}). 
The parameter $Q$ is expected to be related to the so-called \emph{central charge} $c$ by $c  = 25 -6Q^2$~\cite{polyakov-qg3,kpz-scaling,david-conformal-gauge,dk-qg,shef-kpz}. 
Therefore, Proposition~\ref{prop-lfpp-lim} suggests that LFPP for $\xi > 2/d_2$ might provide an approximation to a metric on LQG with central charge $c \in (1,25)$. LQG with $c \in (1,25)$ is much less well-understood than the case when $c \leq 1$ (which corresponds to $\gamma \in (0,2]$).
See~\cite{ghpr-central-charge} for more on LQG with $c \in (1,25)$.
We believe that the case when $\xi>2/d_2$ is of substantial interest, but it is outside the scope of the current paper. Some results for LFPP with $\xi >2/d_2$ are proven in~\cite{gp-lfpp-bounds}.
\end{remark}

\subsection{Proof of monotonicity, continuity, and bounds, assuming universality}
\label{sec-mono}

In this subsection we will explain why the monotonicity and continuity of of $\gamma \mapsto d_\gamma$ and the bounds~\eqref{eqn-d-bound} follow from our universality results, in particular Theorems~\ref{thm-lfpp-compare} and~\ref{thm-ball-size}. Throughout, we let $h$ be a whole-plane GFF normalized so that $h_1(0) = 0$ and we assume that the limit~\eqref{eqn-d-def} exists and that the conclusions of Theorems~\ref{thm-lfpp-compare} and~\ref{thm-ball-size} are satisfied. Aside from these results, the key input in our proofs is the following elementary monotonicity observation for LFPP distances re-scaled by a quantity proportional to $1/\BB E[e^{\xi h_\delta(z)}]$.

\begin{lem}[Monotonicity of re-scaled LFPP distances] \label{lem-lfpp-mono}
For $0 < \xi < \wt\xi$, there is a coupling of two whole-plane GFFs $h \eqD \wt h$ such that for each bounded connected open set $U\subset \BB C$, each $z,w\in U$, each $\delta>0$, the LFPP distances with exponents $\xi$ and $\wt\xi$ satisfy
\eqb \label{eqn-lfpp-mono}
\BB P\left[   \delta^{ \wt \xi^2/2} D_{h,\LFPP}^{\wt\xi ,\delta}(z,w ; U)  \leq C \delta^{ \xi^2/2} D_{ h  ,\LFPP}^{ \xi,\delta}(z,w ; U)  \right] \geq 1 - O_C(1/C) 
\eqe
as $C\rta\infty$, at a rate which is uniform in $\delta$. 
\end{lem}
\begin{proof}
Let $h'$ be an independent GFF with the same law as $h $. Then the field
\eqbn
\wt h := \wt\xi^{-1} \left( \xi h + \sqrt{\wt\xi^2 - \xi^2} h' \right)
\eqen
has the same law as $h$, as can be seen by computing $\BB E[(\wt h , f) (\wt h , g)]$ for smooth compactly supported functions $f,g$. We now need to compare LFPP distances with respect to $h $ and $\wt h$.

By the definition of LFPP, for $\delta > 0$ and $z,w\in U$, we can find a piecewise continuously differentiable path $P : [0,T] \rta \ol{U}$ from $z$ to $w$ which is a measurable function of $h $ and which satisfies
\eqbn
\delta^{\xi^2/2} \int_0^T e^{\xi h_\delta(P(t) )} |P'(t)| \, dt \leq 2 \delta^{ \xi^2/2} D_{h,\LFPP}^{\xi,\delta}(z,w ; U) .
\eqen
We have
\allb
\BB E\left[ \delta^{ \wt\xi^2 / 2} D_{\wt h,\LFPP}^{\wt\xi,\delta}(z,w ; U)  \,\big|\, h \right] 
&\leq  \BB E\left[  \delta^{ \wt\xi^2 / 2} \int_0^T   e^{\wt\xi \wt h_\delta(P(t) ) } |P'(t)| \, dt  \,\big|\, h \right] \notag \\
&= \delta^{ \wt\xi^2 / 2} \int_0^T   e^{\xi   h_\delta(P(t) ) }  |P'(t)|  \BB E\left[  e^{ \sqrt{\wt\xi^2 - \xi^2} h'_\delta(P(t))}  \,\big|\, h   \right]   \, dt   .
\alle
By the calculations in~\cite[Section 3.1]{shef-kpz}, the circle average $h'_\delta(u)$ is independent from $h $ and is centered Gaussian with variance at most $\log \delta^{-1} + O_\delta(1)$ (with the $O_\delta(1)$ uniform over all $u\in U$), so $\BB E\left[  e^{ \sqrt{\wt\xi^2 - \xi^2} h'_\delta(P(t))}  \,\big|\, h   \right]$ above is bounded above by a deterministic constant (depending only on $\xi$ and $\wt\xi$) times $\delta^{(\wt\xi^2 - \xi^2)/2}$. Therefore,
\alb
\BB E\left[ \delta^{\wt\xi^2/2} D_{\wt h,\LFPP}^{\wt\xi , \delta}(z,w ; U)  \,\big|\, h  \right] 
\preceq \delta^{\xi^2/2} \int_0^T   e^{\xi  h_\delta(P(t) ) }  |P'(t)|    \, dt 
\preceq \delta^{\xi^2/2} D_{h ,\LFPP}^{\xi,\delta}(z,w  ; U) ,
\ale
with a deterministic implicit constant. We now conclude by means of Markov's inequality.
\end{proof}

By Theorem~\ref{thm-lfpp-compare} and Lemma~\ref{lem-lfpp-mono}, for $\gamma_1,\gamma_2 \in (0,2)$, 
\eqb \label{eqn-exponent-mono}
\frac{\gamma_1}{d_{\gamma_1}} \leq \frac{\gamma_2}{d_{\gamma_2}} 
\quad  \Rightarrow \quad 
 1-\frac{2}{d_{\gamma_1} } - \frac{\gamma_1^2}{2d_{\gamma_1}}  +  \frac{\gamma_1^2}{2d_{\gamma_1}^2} 
   \leq 1-\frac{2}{d_{\gamma_2} } - \frac{\gamma_2^2}{2d_{\gamma_2}} + \frac{\gamma_2^2}{2d_{\gamma_2}^2} .
\eqe
We will now use the relation~\eqref{eqn-exponent-mono} to prove the properties of $d_\gamma$ stated in Theorem~\ref{thm-d}.

\begin{prop} \label{prop-d-mono}
The function $\gamma\mapsto d_\gamma$ is strictly increasing on $(0,2)$.
\end{prop}
\begin{proof}  
Suppose by way of contradiction that there exists $0 < \gamma_1 < \gamma_2 < 2$ such that $d_{\gamma_1} \geq d_{\gamma_2}$. 
Then $\gamma_1/d_{\gamma_1}  < \gamma_2/d_{\gamma_2}$. 
We will argue that 
\eqb \label{eqn-d-mono-show}
 1-\frac{2}{d_{\gamma_1} } - \frac{\gamma_1^2}{2d_{\gamma_1}}   +  \frac{\gamma_1^2}{2d_{\gamma_1}^2}    >  1-\frac{2}{d_{\gamma_2} } - \frac{\gamma_2^2}{2d_{\gamma_2}} +  \frac{\gamma_2^2}{2d_{\gamma_2}^2}
\eqe  
which will contradict~\eqref{eqn-exponent-mono}. 
To this end, choose a non-increasing continuously differentiable function $f : [\gamma_1,\gamma_2] \rta [d_{\gamma_2} , d_{\gamma_1}]$ with $f(\gamma_1) = d_{\gamma_1}$ and $f(\gamma_2) = d_{\gamma_2}$ and set
\eqbn
g(\gamma) :=  1-\frac{2}{f(\gamma) } - \frac{\gamma^2}{2 f(\gamma)}  + \frac{\gamma^2}{2 f(\gamma)^2}
\eqen
so that~\eqref{eqn-d-mono-show} is the same as $g(\gamma_1)  > g(\gamma_2) $. Implicit differentiation gives
\eqbn
g'(\gamma) = \left( \frac{ \gamma}{f (\gamma)^2} - \frac{\gamma}{f(\gamma)}  \right)  + f'(\gamma) \left( \frac{\gamma^2 }{2f(\gamma)^2}  + \frac{2 }{f(\gamma)^2}  -  \frac{\gamma^2   }{f(\gamma)^3} \right)  .
\eqen
Since $d_\gamma \geq 2$, we have $f(\gamma) \geq 2$, 
so
\eqbn
\frac{ \gamma}{f (\gamma)^2} - \frac{\gamma}{f(\gamma)} < 0 
\quad \op{and} \quad
\frac{\gamma^2 }{2 f(\gamma)^2}  + \frac{2 }{f(\gamma)^2}  -  \frac{\gamma^2   }{f(\gamma)^3} \geq \frac{2}{f(\gamma)^2}  > 0 .
\eqen
Since $f'(\gamma) \leq 0$, it follows that $g'(\gamma)   < 0$, and in particular $g(\gamma_1) > g(\gamma_2)$, which is the desired contradiction.
\end{proof}

\begin{proof}[Proof of Proposition~\ref{prop-increase}]
Since $\gamma\mapsto d_\gamma$ is increasing (Proposition~\ref{prop-d-mono}), the function $\gamma\mapsto \gamma/d_\gamma$ is continuous except possibly for countably many downward jumps. (It is also not hard to check directly that $d_\gamma$, and hence also $\gamma/d_\gamma$, is continuous, but this is not necessary for our argument here. We will check that $d_\gamma$ is continuous in the proof of Theorem~\ref{thm-d} below.)
Clearly, $\gamma/d_\gamma \rta 0$ as $\gamma\rta 0$. By Lemma~\ref{lem-cont-increase} just below, to show that $\gamma\mapsto \gamma/d_\gamma$ is strictly increasing it therefore suffices to show that this function is injective. 
To this end, suppose $0<\gamma_1 \leq \gamma_2 < 2$ such that $\gamma_1/d_{\gamma_1} = \gamma_2/d_{\gamma_2}$. We will show that $\gamma_1=\gamma_2$. By Theorem~\ref{thm-lfpp-compare}, 
\eqbn
 1-\frac{2}{d_{\gamma_1} } - \frac{\gamma_1^2}{2d_{\gamma_1}}  = 1-\frac{2}{d_{\gamma_2} } - \frac{\gamma_2^2}{2d_{\gamma_2}}  .
\eqen
Writing $\xi = \gamma_1/d_{\gamma_1} = \gamma_2/d_{\gamma_2}$, subtracting $1$ from both sides, then dividing by $-\xi$ gives 
\eqbn
 \frac{2  }{\gamma_1 } + \frac{\gamma_1 }{2}  =  \frac{2   }{ \gamma_2  } + \frac{  \gamma_2 }{2}  .
\eqen
Since $0<\gamma_1\leq \gamma_2<2$, this implies that $\gamma_1 = \gamma_2$. Hence $\gamma\mapsto \gamma/d_\gamma$ is strictly increasing. Combining this with~\eqref{eqn-exponent-mono} shows that $\gamma\mapsto 1-\frac{2}{d_{\gamma } } - \frac{\gamma^2}{2d_{\gamma}} + \frac{\gamma^2}{2d_\gamma^2}$ is non-decreasing. 
\end{proof}

We now prove the following elementary lemma which was used in the proof of Proposition~\ref{prop-increase}. 
 
\begin{lem} \label{lem-cont-increase}
Let $f : [0,1] \rta [0,\infty)$ be an injective function such that $f(0) = 0$ and $f$ has no upward jumps, i.e., $\liminf_{y\rta x^-} f(y) \geq f(x)$ and $\limsup_{y\rta x^+} f(y) \leq f(x)$ for each $x \in [0,1]$. Then $f$ is continuous and strictly increasing. 
\end{lem}
\begin{proof}
We claim that the range of $f$ is an interval. Indeed, suppose $b\in (0,\max_{x\in[0,1]} f(x) )$ and let $x_* := \sup\{x\in [0,1] : f(x) \leq b\}$. By left upper semicontinuity, $b \geq \liminf_{y\rta x_*^-} f(y) \geq f(x_*)$ and by right lower semicontinuity, $b \leq \limsup_{y\rta x_*^+} f(y) \leq f(x_*)$, so $f(x_*) = b$. 
The same applies to the restriction of $f$ to  $[0,x]$ for any $x\in [0,1]$. Consequently, if $0 \leq x < y \leq 1$, then $f(x) < f(y)$ since otherwise we would have $f(y) \in [0,f(x)] \subset f([0,x])$ which would contradict the injectivity of $f$. This shows that $f$ is strictly increasing, so since $f$ has no upward jumps $f$ must be continuous.
\end{proof}

\begin{proof}[Proof of Theorem~\ref{thm-d}]
The monotonicity of $d_\gamma$ was proven in Proposition~\ref{prop-d-mono}. 
The lower bound~\eqref{eqn-d-asymp} for the asymptotics as $\gamma\rta 0^+$ follows from~\cite[Theorem 1.1]{ding-goswami-watabiki}.  Since $\gamma \mapsto d_\gamma$ and $\gamma \mapsto \gamma/d_\gamma$ are increasing, for $0 < \gamma_1 < \gamma_2 < 2$ we have
\eqbn
d_{\gamma_1} \leq d_{\gamma_2} \leq \frac{\gamma_2}{\gamma_1} d_{\gamma_1} ,
\eqen
which gives the desired local Lipschitz continuity of $d_\gamma$. 

To prove the bounds~\eqref{eqn-d-bound}, we argue as follows. 
By Theorem~\ref{thm-ball-size} (applied in the case of the UIPT) and~\cite[Theorem 1.2]{angel-peeling}, we get $d_{\sqrt{8/3}} = 4$. Hence the monotonicity of~\eqref{eqn-increase-exponent} of Proposition~\ref{prop-increase} shows that 
\eqb \label{eqn-increase-upper}
\frac{\gamma}{d_\gamma} \geq \frac{1}{\sqrt 6} , \quad \forall \gamma \in (\sqrt{8/3} , 2) \quad \text{and} \quad
\frac{\gamma}{d_\gamma} \leq \frac{1}{\sqrt 6} ,\quad \forall \gamma \in (0,\sqrt{8/3}) .
\eqe 
Similarly, using the monotonicity of~\eqref{eqn-increase-dist} we get
\eqb \label{eqn-increase-lower}
1-\frac{2}{d_\gamma} - \frac{\gamma^2}{2d_\gamma} + \frac{\gamma^2}{2d_\gamma^2} \geq  \frac14 , \quad \forall \gamma \in (\sqrt{8/3} , 2) \quad \text{and} \quad
1-\frac{2}{d_\gamma} - \frac{\gamma^2}{2d_\gamma} + \frac{\gamma^2}{2d_\gamma^2} \leq \frac14 ,\quad \forall \gamma \in (0,\sqrt{8/3}) .
\eqe 
Finally, from the bounds for the volume of a metric ball in the mated-CRT map from~\cite[Theorem 1.10]{ghs-dist-exponent}, we get
\eqb \label{eqn-basic-bounds}
\frac{2\gamma^2}{4+\gamma^2-\sqrt{16+\gamma^4}} \leq d_\gamma \leq 2 + \frac{\gamma^2}{2} + \sqrt 2 \gamma ,\quad\forall \gamma \in (0,2) .
\eqe
Combining~\eqref{eqn-increase-upper},~\eqref{eqn-increase-lower}, and~\eqref{eqn-basic-bounds} gives~\eqref{eqn-d-bound}. 
\end{proof}

\section{Estimates for Liouville graph distance and LFPP}
\label{sec-lfpp-compare}

The goal of this section is to prove Theorems~\ref{thm-diam} and~\ref{thm-lfpp-compare}. 
For most of our arguments, instead of working with the GFF we will work with two approximations of the GFF defined by integrating the transition density of Brownian motion against a white noise which we introduce in Section~\ref{sec-wn}. The process $\wh h$, defined in~\eqref{eqn-wn-decomp}, possesses several exact scale and translation invariance properties which make it especially suitable for multi-scale analysis. The process $\wh h^\tr$, defined in~\eqref{eqn-wn-truncate}, is a truncated version of $\wh h$ which is no longer scale invariant in law but satisfies a local independence property which will be useful in various ``percolation"-style arguments below. We will prove in Lemmas~\ref{lem-tr-compare-square} and~\ref{lem-circle-avg-approx}, respectively, that Liouville graph distance and LFPP with respect to either $\wh h$ or $\wh h^\tr$ can be compared to the analogous distances with respect to a GFF. 

In Section~\ref{sec-diam}, we prove Theorem~\ref{thm-diam} by first establishing an upper concentration estimate for the Liouville graph distance between the two sides of a rectangle (using a percolation argument). We then apply this estimate at several scales and take a union bound to get an upper bound on the distance between the two sides of many different rectangles simultaneously. This then leads to an upper bound for the Liouville graph distance diameter of the unit square by concatenating paths within these rectangles in an appropriate manner. 

In Sections~\ref{sec-lfpp-lower} and~\ref{sec-lfpp-upper}, respectively, we prove the upper and lower bounds for LFPP from Theorem~\ref{thm-lfpp-compare}. The basic idea of the proofs in both cases is to fix a small parameter $\beta \in (0,1)$ (it turns out that any $0 < \beta < 2/(2+\gamma)^2$ will suffice) and compare LFPP with circle-average radius $\delta = \ep^\beta$ to Liouville graph distance defined using balls of LQG mass at most $\ep$. We know the latter distance can be described in terms of $d_\gamma$ by Theorems~\ref{thm-dzz} and~\ref{thm-diam}. To carry out the comparison, we will first condition on the field at scale $\ep^\beta$ (in the sense of the white-noise approximation process $\wh h$). We will then estimate the Liouville graph distance within each sub-square of the unit square of side length approximately $\ep^\beta$. This will be done using our known estimates for Liouville graph distance and the scaling properties of this distance when one re-scales space and adds a constant to the field (this ``constant" will depend on the values of the field at scale $\ep^\beta$). 

For the proofs in this section, it will often be convenient to consider decompositions into dyadic squares and rectangles, so here we introduce some notation to describe rectangles. All of the rectangles we consider will be closed.
\begin{itemize}
\item We write $\BB S = [0,1]^2$ for the unit square. 
\item For a square $S\subset \BB C$, we write $|S|$ for its side length and $v_S$ for its center. 
\item For a rectangle $R \subset \BB C$ and $r > 0$, we write $R(r)$ for the closed $r$-neighborhood of $R$ with respect to the $L^\infty$ metric, i.e., the rectangle with the same center as $R$ whose sides are parallel to $R$ and have length $1 + 2r$ times the side lengths of $R$. 
\end{itemize}

\subsection{White noise approximation}
\label{sec-wn}

In this subsection we will introduce various white-noise approximations of the Gaussian free field which are often more convenient to work with than the GFF itself and discuss several properties of these processes, many of which were proven in~\cite{ding-goswami-watabiki,dzz-heat-kernel}. Let $W$ be a space-time white noise on $\BB C\times [0,\infty)$, i.e., $\{(W,f) : f\in L^2(\BB C\times [0,\infty))\}$ is a centered Gaussian process with covariances $\BB E[(W,f) (W,g) ]  = \int_\BB C\int_0^\infty f(z,s) g(z,s) \,ds \, dz$. For $f\in L^2(\BB C\times [0,\infty))$ and Borel measurable sets $A\subset\BB C$ and $I\subset [0,\infty)$, 
we slightly abuse notation by writing 
\eqbn
\int_A\int_I f(z,s) \, W(dz,ds) := (W , f \BB 1_{A\times I} ) .
\eqen

For an open set $U \subset \BB C$, we write $p_U(s ; z,w)$ for the transition density of Brownian motion killed upon exiting $U$, so that for $s\geq 0$, $z\in \BB C$, and $A\subset \ol U$, the integral $\int_A p_U(s;z,w) \,dw$ gives the probability that a standard planar Brownian motion $\mcl B$ started from $z$ satisfies $\mcl B([0,s]) \subset U$ and $\mcl B_s \in A$. We also write 
\eqbn
p(s;z,w) := p_{\BB C}(s;z,w) =  \frac{1}{2\pi s} \exp\left( - \frac{|z-w|^2}{2s} \right) .
\eqen  

Following~\cite[Section 3]{ding-goswami-watabiki}, we define the centered Gaussian process
\eqb \label{eqn-wn-decomp}
\wh h_t (z) := \sqrt\pi \int_{\BB C} \int_{t^2}^1 p (s/2 ;z,w) \, W(dw,ds)  ,\quad \forall t \in [0,1] , \quad \forall z\in \BB C .
\eqe 
We write $\wh h  := \wh h_0$.
Note that $\wh h_t$ is called $\eta_t^1$ in~\cite{ding-goswami-watabiki}.
By~\cite[Lemma 3.1]{ding-goswami-watabiki} and Kolmogorov's criterion, each $\wh h_t$ for $t \in (0,1]$ admits a continuous modification. Henceforth whenever we work with $\wh h_t$ we will assume that it has been replaced with such a modification. 
The process $\wh h$ does not admit a continuous modification, but makes sense as a distribution: indeed, it is easily checked that its integral against any smooth compactly supported test function is Gaussian with finite variance.
This distribution is not itself a Gaussian free field, but it does approximate a Gaussian free field in several useful respects (see in particular Lemmas~\ref{lem-gff-compare} and~\ref{lem-circle-avg-approx}).  
We record the formula
\eqb \label{eqn-wn-var}
\op{Var}\left( \wh h_{\wt t}(z) - \wh h_t(z) \right) = \log (\wt t/t),\quad\forall z \in \BB C, \quad\forall 0 < t <\wt t < 1 ,
\eqe

The process $\wh h$ is in some ways more convenient to work with than the GFF thanks to the following symmetries, which are immediate from the definition. 
\begin{itemize}
\item \textit{Rotation/translation/reflection invariance.} The law of $\{\wh h_t : t\in [0,1]  \}$ is invariant with respect to rotation, translation, and reflection of the plane.
\item \textit{Scale invariance.} For $\delta \in (0,1]$, one has $\{(\wh h_{\delta t } - \wh h_\delta)(\delta \cdot)  : t \in [0,1]  \} \eqD \{\wh h_t : t\in [0,1]\}$. 
\item \textit{Independent increments.} If $0 \leq t_1\leq t_2 \leq t_3 \leq t_4 \leq 1$, then $\wh h_{t_2} - \wh h_{t_1}$ and $\wh h_{t_4} - \wh h_{t_3}$ are independent. 
\end{itemize}

One property which $\wh h$ does not possess is spatial independence. To get around this, we will sometimes work with a truncated variant of $\wh h $ where we only integrate over a ball of finite radius. For $t\in [0,1]$, we define 
\eqb \label{eqn-wn-truncate}
\wh h_t^\tr(z) := \sqrt\pi \int_{t^2}^1 \int_{\BB C} p_{B_{1/10}(z)}(s/2; z,w)   \, W(dw,dt) .
\eqe 
We also set $\wh h^\tr := \wh h^\tr_0$. 
As in the case of $\wh h$, it is easily seen from the Kolmogorov continuity criterion that each $\wh h^\tr_t$ for $t\in (0,1]$ a.s.\ admits a continuous modification (see \cite[Lemmas 2.3 and 2.5]{dzz-heat-kernel} for a proof of a very similar statement).
The process $\wh h^\tr$ does not admit a continuous modification and is instead viewed as a random distribution. 

The key property enjoyed by $\wh h^\tr$ is spatial independence: if $A,B\subset \BB C$ with $\op{dist}(A,B) \geq 1/5$, then $\{\wh h^\tr_t|_A : t\in [0,1]\}$ and $\{\wh h^\tr_t|_B : t\in [0,1]\}$ are independent. Indeed, this is because $\{\wh h^\tr_t|_A : t\in [0,1]\}$ and $\{\wh h^\tr_t|_B : t\in [0,1]\}$ are determined by the restrictions of the white noise $W$ to the disjoint sets $B_{1/10}(A) \times \BB R_+$ and $B_{1/10}(B)\times \BB R_+$, respectively.  
Unlike $\wh h$, the distribution $\wh h^\tr$ does not possess any sort of scale invariance but its law is still invariant with respect to rotations, translations, and reflections of $\BB C$. 
We note that our definition of $\wh h^\tr$ is simpler than the definition of the truncated white-noise decomposition used in~\cite{dzz-heat-kernel} since we do not need to have the spatial independence property at all scales.

The following lemma will allow us to use $\wh h^\tr$ or $\wh h$ in place of the GFF in many of our arguments. 

\begin{lem} \label{lem-gff-compare}
Suppose $U\subset \BB C$ is a bounded Jordan domain and let $K$ be the set of points in $U$ which lie at Euclidean distance at least $1/10$ from $\bdy U$. 
There is a coupling $(h , h^U,\wh h , \wh h^\tr)$ of a whole-plane GFF normalized so that $h_1(0) = 0$, a zero-boundary GFF on $U$, and the fields from~\eqref{eqn-wn-decomp} and~\eqref{eqn-wn-truncate} such that the following is true. For any $h^1,h^2 \in \{h , h^U,\wh h , \wh h^\tr\}$, the distribution $(h^1-h^2)|_K$ a.s.\ admits a continuous modification and there are constants $c_0,c_1 > 0$ depending only on $U$ such that for $A>1$, 
\eqb \label{eqn-gff-compare}
\BB P\left[\max_{z\in K} |(h^1-h^2)(z)| \leq A \right] \geq 1 - c_0 e^{-c_1 A^2} .
\eqe
In fact, in this coupling one can arrange so that $\wh h$ and $\wh h^\tr$ are defined using the same white noise and $h - h^U$ is harmonic on $U$. 
\end{lem}

Lemma~\ref{lem-gff-compare} is proven in Appendix~\ref{sec-gff-compare} via elementary calculations for the transition density $p_U(t;z,w)$ which allow us to check the Kolmogorov continuity criterion for $h^1-h^2$. Once we establish the continuity of $h^1-h^2$, the bound~\eqref{eqn-gff-compare} comes from the Borell-TIS inequality.  

\subsubsection{LQG measures and Liouville graph distances for $\wh h$ and $\wh h^\tr$}

Lemma~\ref{lem-gff-compare} allows us to define for each $\gamma \in (0,2)$ the $\gamma$-LQG measures $\mu_{\wh h}$ and $\mu_{\wh h^\tr}$ associated with the fields $\wh h$ and $\wh h^\tr$. Indeed, one way to do this is as follows. If $h$ is a GFF and $f$ is a (possibly random) continuous function, then for any $z\in\BB C$ and any $\ep > 0$ we can define the average $(h+f)_\ep(z) = h_\ep(z) + f_\ep(z)$ of $h+f$ over the circle $\bdy B_\ep(z)$. We can then define $\mu_{h+f}$ as the a.s.\ weak limit $\lim_{\ep\rta 0} \ep^{\gamma^2/2} e^{\gamma (h+f)_\ep(z)} \,dz$, following~\cite[Proposition 1.1]{shef-kpz}. With this definition, one has $d\mu_{h+f} = e^{\gamma f} \, d\mu_h$ a.s. 
Applying this with $f = \wh h - h$ or $\wh h^\tr - h$, when the fields are coupled as in Lemma~\ref{lem-gff-compare}, allows us to define $\mu_{\wh h}$ and $\mu_{\wh h^\tr}$.  
 
The measures $\mu_{\wh h}$ and $\mu_{\wh h^\tr}$ are a.s.\ non-atomic and assign positive mass to every open set. 
Furthermore, for any open set $U\subset \BB C$, we have that $\mu_{\wh h}$ and $\mu_{\wh h^\tr}$ are determined by the restrictions of $\wh h$ and $\wh h^\tr$, respectively, to $U$.  

As in the case of a GFF, for $z,w\in \BB C$ and $\ep>0$, we define the \emph{Liouville graph distance} $D_{\wh h}^\ep(z,w)$ with respect to $\wh h$ to be the smallest $N\in\BB N$ for which there is a continuous path from $z$ to $w$ which can be covered by at most $N$ Euclidean balls of $\mu_{\wh h}$-mass at most $\ep$. 
We extend the definitions of the localized Liouville graph distance and the Liouville graph distance between sets from Definition~\ref{def-restricted-lgd} to $D_{\wh h}^\ep$ in the obvious manner.
We similarly define $D_{\wh h^\tr}^\ep$. 

As a consequence of Lemma~\ref{lem-gff-compare}, we have the following lemma, which will be a key tool in our proofs.

\begin{lem} \label{lem-tr-compare-square}
Suppose $U \subset \BB C$ and $K\subset U$ are as in Lemma~\ref{lem-gff-compare} and that $(h , h^U,\wh h , \wh h^\tr)$ are coupled as in Lemma~\ref{lem-gff-compare}.
For each $\gamma \in (0,2)$, there are constants $a_0,a_1 >0$, depending only on $\gamma$, such that for each $\ep \in (0,1)$, each pair of fields $h^1,h^2 \in \{h,h^{U} , \wh h, \wh h^\tr\}$, and each $C>1$, 
\eqb \label{eqn-tr-compare-square}
\BB P\left[ D_{h^1}^{C\ep}\left( z , w ; K \right) \leq   D_{h^2}^\ep\left( z , w ; K \right) \leq   D_{h^1}^{\ep/C}\left( z , w ; K \right)  ,\: \forall z,w\in K \right] \geq 1 - a_0 e^{-a_1(\log C)^2} . 
\eqe 
\end{lem}
\begin{proof}
This follows from Lemma~\ref{lem-gff-compare} applied with $A = \frac{1}{\gamma} \log C$ and the fact that for $h^1,h^2 \in \{h,h^U , \wh h, \wh h^\tr\}$, we have $d\mu_{h^1} = e^{\gamma(h^1-h^2)} \, d\mu_{h^2}$. 
\end{proof}

Due to the scale invariance and independent increments properties of $\wh h$, it is convenient to understand how Liouville graph distances with respect to $\wh h$ transform under scaling.  
The basic properties of $\wh h$ listed above show that for $\gamma \in (0,2)$, $\delta \in (0,1)$, and $b \in\BB C$, the $\gamma$-LQG measures and Liouville graph distances associated with the fields $\wh h$ and $\wh h(\delta\cdot + b) - \wh h_\delta(\delta \cdot)$ satisfy
\eqb \label{eqn-scaled-measures}
\mu_{ ( \wh h  -\wh h_\delta )(\delta \cdot + b) } \eqD \mu_{\wh h}  \quad \op{and} \quad
D_{( \wh h  -\wh h_\delta )(\delta \cdot + b)  }^\ep \eqD D_{\wh h}^\ep ,\quad\forall \ep> 0.
\eqe
Furthermore, these measures and distances are related in the following deterministic manner. 

\begin{lem}  \label{lem-measure-scale}
For each $b\in\BB C$ and each $\delta \in (0,1)$, a.s.\ 
\eqb \label{eqn-measure-scale}
\mu_{\wh h}(X)
=  \delta^{ 2 + \gamma^2/2}  \int_{\delta^{-1}(X-b)} e^{\gamma \wh h_\delta(\delta \cdot + b)} \, d \mu_{( \wh h  -\wh h_\delta )(\delta \cdot + b)   }(z)   , \quad \forall \: \text{Borel set $X\subset \BB C$} .
\eqe
Furthermore, if $U\subset \BB C$ is a bounded, open, connected set and we set\footnote{Note here that $\ol T \leq \ul T$, which might be slightly unintuitive. The reason for the notation is that $\ol T$ corresponds to a larger distance function. A similar notational convention is used for variants of Liouville graph distance in Section~\ref{sec-sle-compare}.}
\eqbn
\ul T := \delta^{-2 -\gamma^2/2} \exp\left(- \min_{z\in U}  \wh h_{\delta }(z)   \right) \quad \op{and} \quad \ol T := \delta^{ -2-\gamma^2/2} \exp\left(- \max_{z\in U}  \wh h_{\delta }(z)   \right)
\eqen
then a.s.\  the restricted Liouville graph distances satisfy
\eqb \label{eqn-dist-scaling}
 D_{\wh h}^\ep(z,w; U )        
 \leq  D_{( \wh h  -\wh h_\delta )(\delta \cdot + b) }^{\ol T \ep} \left( \delta^{-1}(z-b) , \delta^{-1}(w-b) ; \delta^{-1}(U - b)  \right)     , \quad  \forall \ep > 0,\quad \forall z,w\in U ,   
\eqe
and the reverse inequality holds with $\ul T$ in place of $\ol T$. 
\end{lem}
\begin{proof}
By the $\gamma$-LQG coordinate change formula~\cite[Proposition 2.1]{shef-kpz}, a.s.\ $\mu_{\wh h(\delta\cdot+b) + Q \log \delta}(\delta^{-1}(X-b)) = \mu_{\wh h}(X)$ for all Borel sets $X\subset\BB C$, where $Q = 2/\gamma + \gamma/2$ (this is also easy to see directly from the circle average or white-noise approximations of the measures).
This together with the relation $d\mu_{\wh h + f} = e^{\gamma f} \, d\mu_{\wh h}$ yields~\eqref{eqn-measure-scale}.
 The relation~\eqref{eqn-dist-scaling} follows from~\eqref{eqn-measure-scale} applied to Euclidean balls contained in $U$.  
\end{proof} 

In the remainder of this section we record some basic estimates for the above processes $\wh h$ and $\wh h^\tr$, building on estimates from~\cite{ding-goswami-watabiki,dzz-heat-kernel}. The reader may wish to skip these estimates on a first read and refer back to them as they are used.

\subsubsection{Estimates for $\wh h_\delta$}

We start with estimates for the modulus of continuity and maximum value of the process $\wh h_\delta$ from~\eqref{eqn-wn-decomp}. 

\begin{lem} \label{lem-use-btis} 
For each $\zeta \in (0,1)$ and each bounded domain $U \subset \BB C$, it holds with superpolynomially high probability as $\delta \rta 0$ that
\eqb \label{eqn-use-btis}
  \max_{z,w \in U : |z-w| \leq \delta} |\wh h_\delta(z) -\wh h_\delta(w)| \leq \zeta \log \delta^{-1} .
\eqe 
\end{lem}
\begin{proof}
It is easily seen (see~\cite[Lemma 3.1]{ding-goswami-watabiki}) that for $\delta > 0$, $\op{Var}(\wh h_\delta(z) - \wh h_\delta(w)) \leq |z-w|^2/\delta^2$, which is of course smaller than $|z-w| /\delta$ whenever $|z-w| \leq \delta$.  
By Fernique's criterion~\cite{fernique-criterion} (see~\cite[Theorem 4.1]{adler-gaussian} or~\cite[Lemma 2.3]{dzz-heat-kernel} for the version we use here), we find that for each square $S\subset\BB C$ with side length $\delta/2$, 
\eqbn
\BB E\left[ \max_{z,w\in S} |\wh h_\delta(z) - \wh h_\delta(w)| \right] \leq C ,
\eqen
for a universal constant $C>0$.
Combining this with the Borell-TIS inequality~\cite{borell-tis1,borell-tis2} (see, e.g.,~\cite[Theorem 2.1.1]{adler-taylor-fields}), we get that for each such square $S$, 
\eqbn
\BB P\left[ \max_{z,w\in S} |\wh h_\delta(z) - \wh h_\delta(w)| \leq \zeta \log \delta^{-1} \right]  \geq 1 - e^{- \frac{\zeta^2}{2} (\log\delta^{-1})^2} .
\eqen
A union bound over $O_\delta(\delta^{-2})$ such squares whose union contains $U$ concludes the proof. 
\end{proof}

\begin{lem} \label{lem-one-scale-max}
For $\zeta \in (0,1)$ and each bounded domain $U\subset \BB C$, it holds with polynomially high probability as $\delta\rta 0$ that 
\eqb \label{eqn-one-scale-max}
  \max_{z\in U} |\wh h_{\delta }(z)  | \leq  (2+\zeta) \log \delta^{-1}   .
\eqe  
\end{lem}
\begin{proof}
Since each $\wh h_\delta(z)$ is centered Gaussian of variance $\log \delta^{-1}  $, a union bound shows that
\eqbn
\BB P\left[ \max_{z\in \left(\frac{\delta}{2} \BB Z^2 \right) \cap U} |\wh h_\delta(z)| \leq \left(2 + \frac{\zeta}{2} \right) \log \delta^{-1} \right] \geq 1 -  \delta^{\frac{(2+\zeta/2)^2}{2} - 2 + o_\delta(1)} .
\eqen
Combining this with Lemma~\ref{lem-use-btis} and the triangle inequality concludes the proof. 
\end{proof}

\begin{lem} \label{lem-mid-scale-max}
For each bounded domain $U\subset \BB C$, each $\zeta \in (0,1)$, each $\delta\in (0,1)$,  each $A \in \left(1, e^{( \log\delta^{-1})^{1-\zeta} }\right)$, and each $C\geq 1$, 
\eqb \label{eqn-mid-scale-max}
\BB P\left[ \max_{z,w\in U : |z-w| \leq  C \delta} |\wh h_{ \delta / A}(z) - \wh h_\delta(w)| \leq \zeta \log \delta^{-1} \right]  \geq 1 - O_\delta(\delta^p) ,\: \forall p > 0  ,
\eqe
with the rate of the $O_\delta(\delta^p)$ depending on $U$, $\zeta$, $C$, and $\gamma$ but uniform over all of the possible choices of $A$. 
\end{lem}
\begin{proof} 
The random variables $\wh h_{ \delta /A}(z) - \wh h_\delta(z)$ for $z\in U$ are jointly centered Gaussian with variances $\log A  \leq  (\log \delta^{-1})^{1-\zeta}  $. By the Gaussian tail bound and a union bound, 
\eqb \label{eqn-mid-scale-union}
\BB P\left[ \max_{z \in \left(\frac{\delta}{2} \BB Z^2 \right)\cap U} |\wh h_{ \delta /A}(z) - \wh h_\delta(z)| \leq (\log \delta^{-1})^{1-\zeta/3} \right]  \geq 1 - O_\delta(\delta^p) ,\: \forall p > 0  .
\eqe 
The estimate~\eqref{eqn-mid-scale-max} follows by combining Lemma~\ref{lem-use-btis}, applied for $\wh h_\delta$ and with $ \delta/A$ in place of $\delta$, with $\zeta/(2C)$ in place of $\zeta$, with~\eqref{eqn-mid-scale-union} and the triangle inequality. 
\end{proof}
 
Finally, we record a lemma which serves an analogous purpose to Lemma~\ref{lem-tr-compare-square} but for LFPP instead of Liouville graph distance.

\begin{lem} \label{lem-circle-avg-approx}
Let $h^{\BB S(1)}$ be a zero-boundary GFF on the square $\BB S(1)  $. 
There is a coupling of $\wh h$ and $h^{\BB S(1)}$ such that for each $C > 0$ and each $\zeta\in (0,1)$, it holds with superpolynomilally high probability as $\delta \rta 0$ that
\eqb \label{eqn-circle-avg-approx}
\max_{z,w\in \BB S: |z-w| \leq C \delta} |h_\delta^{\BB S(1)}(z) -\wh h_\delta(w)| \leq \zeta \log\delta^{-1} .
\eqe 
\end{lem}
\begin{proof}
This follows from the uniform comparison between $h_\delta^{\BB S(1)}(z)$ and $\wh h_\delta(z)$ established in~\cite[Proposition 3.2]{ding-goswami-watabiki} together with the continuity estimate for $\wh h_\delta$ from Lemma~\ref{lem-use-btis} (applied with $\zeta/(2C)$ in place of $\zeta$). 
\end{proof}

\subsubsection{Maximal and minimal radii of balls of LQG mass $\ep$}
 
We next record a basic estimate for the maximal and minimal radii of Euclidean balls with $\mu_h$-mass $\ep$ when $h$ is any of the fields considered in Lemma~\ref{lem-tr-compare-square}. The significance of this lemma is that if $z$ and $w$ lie in the same ball of mass $\ep$, then $D_h^\ep(z,w) \leq 1$. 

\begin{lem} \label{lem-max-ball-radius}
Suppose that $h$ is either a whole-plane GFF normalized so that $h_1(0) = 0$, a zero-boundary GFF on $\BB S(1)$, or one of the white noise fields $\wh h$ or $\wh h^\tr$ defined above. 
For each $\ul\beta \in \left(0,\frac{2}{(2+\gamma)^2} \right)$ and each $\ol\beta > \frac{2}{(2-\gamma)^2}$, it holds with polynomially high probability as $\ep\rta 0$ that 
\eqb \label{eqn-max-ball-radius} 
\inf_{z\in \BB S} \mu_{ h}(B_{\ep^{\ul\beta}}(z)) \geq \ep  \quad \op{and} \quad \sup_{z\in \BB S} \mu_{  h}(B_{\ep^{\ol\beta} }(z)) \leq \ep .
\eqe 
\end{lem}
\begin{proof}
By Lemma~\ref{lem-gff-compare}, it suffices to prove the lemma in the case when $h$ is a whole-plane GFF. This, in turn, follows from standard estimates for the $\gamma$-LQG measure. In particular, the first estimate in~\eqref{eqn-max-ball-radius} holds with polynomially high probability by, e.g.,~\cite[Lemma 2.5]{gms-harmonic} applied with $\delta =\ep^{\ul\beta}$. To prove the second estimate, we first use a standard moment estimate for the $\gamma$-LQG measure (see~\cite[Theorem 2.14]{rhodes-vargas-review} or~\cite[Lemma 5.2]{ghm-kpz}) to get that for $z\in \BB S$, $p \in [0,4/\gamma^2)$, and $\delta \in (0,1)$,  
\eqbn
\BB E\left[ \mu_h(B_\delta(z))^p \right] \leq \delta^{f(p) + o_\delta(1) } \quad \text{where} \quad f(p):= \left(2 + \frac{\gamma^2}{2} \right) p - \frac{\gamma^2}{2} p^2 
\eqen
with the rate of the $o_\delta(1)$ uniform over all $z\in \BB S$. 
By Markov's inequality, if $\ol\beta$ is as in the statement of the lemma then for $p\in [0,4/\gamma^2)$, 
\eqbn
\BB P\left[ \mu_h(B_\delta(z)) > \delta^{\ol\beta^{-1}} \right] \leq \delta^{ f(p)  - \ol\beta^{-1} p    +o_\delta(1) } .
\eqen
The exponent on the right is maximized over all values of $p \in [0,4/\gamma^2)$ when $p =   (4+\gamma^2 - 2\ol\beta^{-1} )/(2\gamma^2) $. Choosing this value of $p$ gives 
\eqb \label{eqn-min-ball0}
\BB P\left[ \mu_h(B_\delta(z)) > \delta^{\ol\beta^{-1}} \right] \leq \delta^{\tfrac{(4 + \gamma^2 - 2 \ol\beta^{-1} )^2}{8\gamma^2} + o_\delta(1) } .
\eqe  
We obtain the second estimate in~\eqref{eqn-max-ball-radius} with polynomially high probability by applying~\eqref{eqn-min-ball0} with $\delta =\ep^{\ol\beta}$ then taking a union bound over all $z\in \left(\frac12 \ep^{\ol\beta} \BB Z^2 \right)\cap \BB S$. 
\end{proof}

\subsection{Comparison of diameter and point-to-point distance}
\label{sec-diam}

In this subsection we will prove Theorem~\ref{thm-diam}. 
The main step in the proof is Proposition~\ref{prop-diam} just below. In the course of the proof, we will also establish some estimates which are needed for the proof of the lower bound for LFPP distances in Theorem~\ref{thm-lfpp-compare}.

\begin{prop} \label{prop-diam}
Let $\wh h$ be as in~\eqref{eqn-wn-decomp}. 
For each $\zeta \in (0,1)$, it holds with polynomially high probability as $\ep\rta 0$ that
\eqb \label{eqn-diam}
\max_{z,w\in \BB S} D_{\wh h}^\ep\left(z,w ; \BB S(1/2) \right) \leq \ep^{-\frac{1}{d_\gamma-\zeta}} ,
\eqe
where here we recall that $\BB S(1/2) $ is the expanded square $ [-1/2,3/2]^2$. 
\end{prop}

We now give an overview of the proof of Proposition~\ref{prop-diam}. We will first establish a concentration estimate (Lemma~\ref{lem-rectangle-perc}) which says that the $\wh h$-Liouville graph distance between two sides of a large rectangle is superpolynomially unlikely to be larger than the area of the rectangle times $\ep^{-\frac{1}{d_\gamma-\zeta}}$. To prove this estimate, we use a percolation argument to construct a ``path" of squares from one side of the rectangle to the other with the property that the distance between the midpoints of the sides of the squares in the path is bounded above by $\ep^{-\frac{1}{d_\gamma-\zeta}}$ (several similar percolation arguments are used in~\cite{ding-dunlap-lqg-fpp,ding-goswami-watabiki,dzz-heat-kernel}). For the proof, we will need to work with the truncated field $\wh h^\tr$ of~\eqref{eqn-wn-truncate} since we will need exact local independence in order to carry out the percolation argument (one can do this due to Lemma~\ref{lem-tr-compare-square}). 

By the scale invariance properties of $\wh h$ (see Lemma~\ref{lem-measure-scale}), if $\delta \in (0,1)$ is at least some $\gamma$-dependent positive power of $\ep$ and $R\subset \BB S$ is a $2\delta\times\delta$ or $\delta\times 2\delta$ rectangle, then the conditional law given $\wh h_\delta$ of the $D_{\wh h}^\ep$-distance between the two shorter sides of $R$ is stochastically dominated by the law of the $D_{\wh h}^{T_R\ep}$-distance between the left and right sides of $[0,2]\times [0,1]$ for $T_R = \delta^{2+\gamma^2/2} \exp\left( - \max_{z\in R} \wh h_\delta(z) \right)$ (actually, for technical reasons instead of $[0,2]\times [0,1]$ we will consider a rectangle whose side lengths are of order $(\log\delta^{-1})^{3/2}$). By the aforementioned concentration estimate, a union bound, and our continuity estimate for $\wh h_\delta$ (Lemma~\ref{lem-use-btis}), this allows us to show that with polynomially high probability as $\ep\rta 0$, one has a \emph{simultaneous} upper bound for the distance between the sides of a large number of different $2\delta\times \delta$ or $\delta \times 2\delta$ rectangles $R\subset \BB S$ in terms of $\ep$, $\delta$, and the value of the exponential of $\gamma/d_\gamma$ times the white-noise field $\wh h_\delta$ at any point of the rectangle (Lemma~\ref{lem-rectangle-dist}). More precisely, the distance between the two shorter sides of $R$ is bounded above by
\eqbn
\ep^{-\frac{1}{d_\gamma}} \delta^{\frac{1}{d_\gamma} \left( 2+\frac{\gamma^2}{2} \right) } \exp\left( \frac{\gamma}{d_\gamma} \min_{z\in R} \wh h_\delta(z) \right) , 
\eqen
up to $o(1)$ errors in the exponents (this estimate will also be important for our lower bound for LFPP distances). 

Using Lemma~\ref{lem-one-scale-max}, one can eliminate the dependence on the coarse field $\wh h_\delta$ in the above estimate by replacing $\wh h_\delta(z)$ by the maximum value of $\wh h_\delta(z)$ on $\BB S$ (Lemma~\ref{lem-rectangle-dist-multi}). 
One can then concatenate a logarithmic number of paths between the sides of $2\delta\times \delta$ or $\delta\times 2\delta$ rectangles for dyadic values of $\delta$ to construct a path between any two points in $\BB S$ which can be covered by at most $\ep^{-\frac{1}{d_\gamma-\zeta} + o_\ep(1)}$ disks of $\mu_{\wh h}$-mass at most $\ep$ (see Figure~\ref{fig-diam}, right). This gives Proposition~\ref{prop-diam}. 

In this subsection and the next, we will use the following notation for rectangles.

\begin{defn} \label{def-rectangle}
For a rectangle $R = [a,b] \times [c,d] \subset \BB C$ with sides parallel to the coordinate axes, we write $\bdy_{\op{L}} R$, $\bdy_{\op{R}} R$, $\bdy_{\op{T}} R$, and $\bdy_{\op{B}} R$, respectively, for its left, right, top, and bottom boundaries.  
We also define the associated \emph{stretched rectangle} $R'$ as follows. If the horizontal side length $b-a$ is larger than the vertical side length $d-c$, we let $R' = [a- \frac12(b-a) , b + \frac12(b-a)] \times [c,d]$ be the rectangle with the same center as $R$, twice the horizontal side length as $R$, and the same vertical side length as $R$. If $d-c > b-a$, we define $R'$ analogously with ``horizontal" and ``vertical" interchanged. 
\end{defn}

The following is our concentration bound for the distance across a rectangle.

\begin{lem} \label{lem-rectangle-perc} 
For $n\in\BB N$, let $\mcl R_n := [0,2n]\times [0,n]$, so that $\mcl R_n' =  [-n,3n] \times [0,n ]$.  
For each fixed $\zeta  \in (0,1)$, there exist $a_0,a_1  , A   > 0$ (depending only on $\zeta$ and $\gamma$) such that for $n\in\BB N$ and $\ep > 0$, we have (in the notation of Definition~\ref{def-rectangle}) 
\eqb \label{eqn-rectangle-perc}
\BB P\left[ D_{\wh h }^\ep\left(\bdy_{\op{L}} \mcl R_n , \bdy_{\op{R}} \mcl R_n ; \mcl R_n' \right) \leq  n^2 \max\left\{ A ,  e^{n^{1/2}}  \ep^{-\frac{1}{d_\gamma- \zeta } } \right\} \right] \geq 1 -  a_0 e^{-a_1 n  }  .
\eqe 
\end{lem}

When we apply Lemma~\ref{lem-rectangle-perc}, we will typically take $n \approx (\log\ep^{-1})^p$ for $p \in (1,2)$, so that the $ n^2$ and $e^{n^{1/2}}$ terms are negligible in comparison to $\ep^{-\frac{1}{d_\gamma- \zeta } } $. 
We note that~\eqref{eqn-rectangle-perc} implies that there is a continuous path in $\mcl R_n$ between the left and right boundaries of $\mcl R_n$ which can be covered by at most $ n^2 \max\left\{ A ,  e^{n^{1/2}}  \ep^{-\frac{1}{d_\gamma- \zeta } } \right\}$ Euclidean balls of $\mu_{\wh h}$-mass at most $\ep$ which are contained in $\mcl R_n'$ (this is because any path between the two connected components of $\mcl R_n' \setminus \mcl R_n$ must cross the left and right boundaries of $\mcl R_n$).
However, we need to take distances relative to $\mcl R_n'$ instead of $\mcl R_n$ since some of the Euclidean balls in the covering might not be contained in $\mcl R_n$. 

The starting point of the proof of Lemma~\ref{lem-rectangle-perc} is the following estimate from~\cite{dzz-heat-kernel}, which we will apply to each $1\times 1$ square in $\mcl R_n$ with corners in $\BB Z^2$.

\begin{lem} \label{lem-local-dist}
Recall the truncated field $\wh h^\tr$ from~\eqref{eqn-wn-truncate} and its associated Liouville graph distance. 
Also let $\BB S = [0,1]^2$ and $\BB S(1) = [-1 , 2]^2$ be the squares as defined at the beginning of this section and let $u_{\BB S}^1,\dots,u_{\BB S}^4$ be the midpoints of the four corners of $\BB S$. 
For each $\zeta \in (0,1)$, it holds with probability tending to 1 as $\ep\rta 0$ that
\eqb \label{eqn-local-dist}
 D^\ep_{\wh h^\tr }\left( u_{\BB S}^i, u_{\BB S}^j ; \BB S(1) \right) \leq  \ep^{- \frac{1}{d_\gamma  - \zeta } }  ,\quad \forall i,j \in \{1,\dots,4\}. 
\eqe
\end{lem}
\begin{proof}
The analogue of~\eqref{eqn-local-dist} with a zero-boundary GFF on $\BB S(1)$ in place of $\wh h^\tr$ is proven in~\cite{dzz-heat-kernel} (see, in particular,~\cite[Proposition 3.17 and Lemma 5.3]{dzz-heat-kernel} and note that $\wt D_{\gamma,\delta,\eta}(u,v)$ in~\cite{dzz-heat-kernel} denotes $\delta^2$-Liouville graph distance restricted to paths of disks which lie in the box of side length $2|u-v|$ centered at $(u+v)/2$, with sides parallel to the segment through $[u,v]$). 
The bound~\eqref{eqn-local-dist} follows from this and Lemma~\ref{lem-tr-compare-square}.
\end{proof}

\begin{figure}[t!]
 \begin{center}
\includegraphics[scale=1]{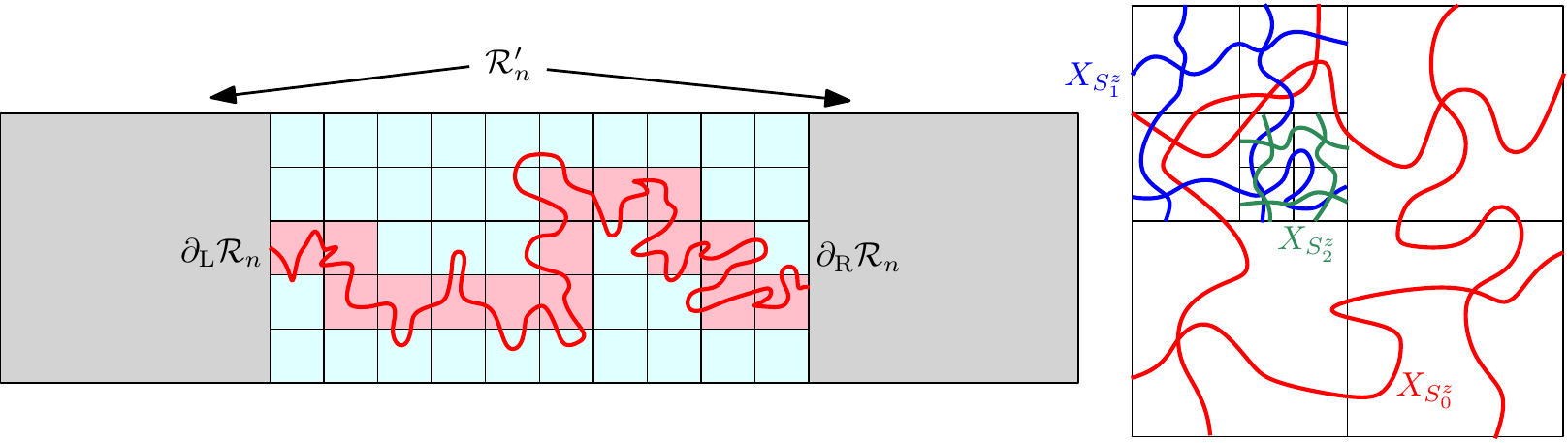}
\vspace{-0.01\textheight}
\caption{ \textbf{Left.} Illustration of the proof of Lemma~\ref{lem-rectangle-perc}. We show that there must exist a path from $\bdy_{\op{L}} \mcl R_n$ to $\bdy_{\op{R}} \mcl R_n$ consisting of unit side-length squares $S$ with the property that the $D_{\wh h^\tr}^\ep$-distance between the midpoints of any of the two sides of $S$, restricted to paths of disks which lie in the slightly expanded square $S(1)$, is at most $\ep^{-\frac{1}{d_\gamma-\zeta}}$ (squares in this path are shown in pink). This gives a Euclidean path (red) from $\bdy_{\op{L}} \mcl R_n$ to $\bdy_{\op{R}} \mcl R_n$ which can be covered by at most $2n^2 \ep^{-\frac{1}{d_\gamma-\zeta}}$ disks of $\mu_{\wh h^\tr}$-mass at most $\ep$ which are contained in $\mcl R_n'$ (larger rectangle).  
\textbf{Right.} To prove Proposition~\ref{prop-diam}, we use Lemma~\ref{lem-rectangle-perc} to find a Euclidean path across each $2^{-m-1}\times 2^{-m}$ or $2^{-m}\times 2^{-m-1}$ rectangle in $S$ with $m\geq \log_2 \ep^{-\beta}$ which can be covered by a bounded number of disks of $\mu_{\wh h^\tr}$-mass at most $\ep$. For each dyadic square $S\subset \BB S$ with side length at most $\ep^\beta$, we consider the set $X_S$ which is the union of the paths crossing four rectangles contained in $S$. The sets $X_S$ corresponding to successive dyadic squares containing $z\in \BB S$ must intersect, which allows us to bound the $D_{\wh h}^\ep$-diameter of each dyadic square of side length at most $\ep^\beta$.  
}\label{fig-diam}
\end{center}
\vspace{-1em}
\end{figure}

\begin{proof}[Proof of Lemma~\ref{lem-rectangle-perc}]
We will show that there are constants $a_0,a_1 , A > 0$ as in the statement of the lemma such that for $n\in\BB N$ and $\ep > 0$, 
\eqb \label{eqn-rectangle-perc-truncated}
\BB P\left[ D_{\wh h^\tr}^\ep\left(\bdy_{\op{L}} \mcl R_n , \bdy_{\op{R}} \mcl R_n  ; \mcl R_n'  \right) \leq  n^2 \max\left\{ A ,    \ep^{-\frac{1}{d_\gamma- \zeta } } \right\} \right] \geq 1 -  a_0 e^{-a_1 n}  .
\eqe 
Combining this with~\eqref{eqn-gff-compare} (applied with $A =  c n^{1/2} $ for an appropriate constant $c>0$) and taking a union bound of $O_n(n^2)$ Euclidean balls of radius 1 whose union covers $\mcl R_n'$ yields~\eqref{eqn-rectangle-perc}. 

See Figure~\ref{fig-diam}, left, for an illustration of the proof of~\eqref{eqn-rectangle-perc-truncated}. 
Let $p\in (0,1)$ be a small universal constant to be chosen later. 
We assume without loss of generality that $n\geq 3$ and let $\mcl S(\mcl R_n)$ be the set of unit side length squares\footnote{The reason for considering $[0,2n]\times [1,n-1]$ instead of $\mcl R_n $ is so that the expanded square $S(1)$ is contained in $[-1,2n+1]\times [0,n ] \subset \mcl R_n'$ instead of in $[-1,2n+1]\times [-1,n+1]$.} 
$S\subset [0,2 n] \times [1,n-1]$ with corners in $\BB Z^2$. 

For each square $S\in \mcl S(\mcl R_n)$ and $\ep \in (0,1)$, define the event
\eqbn
E_S^\ep := \left\{ D^\ep_{\wh h^\tr}\left( u_{S }^i, u_{S }^j ; S(1) \right) \leq  \ep^{- \frac{1}{d_\gamma - \zeta   }} ,\: \forall  i,j\in \{1,\dots,4\} \right\}    
\eqen
where $u_S^1,\dots,u_S^4$ denote the four corners of $S$.  
For each $S\in \mcl S(\mcl R_n)$, the re-centered field $\wh h^\tr(\cdot - v_S + v_{\BB s})  $ agrees in law with $\wh h^\tr$. 
By Lemma~\ref{lem-local-dist}, it therefore follows that we can find $\ep_* = \ep_*(p,\zeta,\gamma) > 0$ such that  
\eqb \label{eqn-perc-prob}
\BB P[E_S^\ep ] = \BB P[E_{\BB S}^\ep] \geq 1-p , \quad \forall S\in \mcl S(\mcl R_n) , \quad \forall \ep\in (0,\ep_*] .
\eqe 
Note that we are only asserting that \emph{each} $E_S^\ep$ individually has probability at least $1-p$ --- we are not yet claiming anything about the probabilities of the intersections of these events.

View $\mcl S(\mcl R_n)$ as a graph with two squares considered to be adjacent if they share an edge. We define the \emph{left boundary} of $\mcl S(\mcl R_n)$ to be the set of squares in $\mcl S(\mcl R_n)$ which intersect the left boundary of $[0,2n]\times [1,n-1]$. We similarly define the right, top, and bottom boundaries of $\mcl S(\mcl R_n)$.  

We claim that if $p$ is chosen sufficiently small, then for appropriate constants $a_0,a_1 > 0$ as in the statement of the lemma, it holds for each $\ep \in (0,\ep_*]$ and $n\in\BB N$ that with probability at least $1- a_0 e^{-a_1 n}$, we can find a path in $\mcl S(\mcl R_n)$ from the left boundary of $\mcl S(\mcl R_n)$ to the right boundary of $\mcl S(\mcl R_n)$ consisting of squares for which $E_S^\ep$ occurs. 

Assume the claim for the moment. Since each square $S(1)$ for $S\in \mcl S(\mcl R_n)$ is contained in $\mcl R_n'$, the definition of $E_S^\ep$ and the triangle inequality show that if a path as in the claim exists then the $ D_{ \wh h^\tr}^\ep$-distance between the left and right boundaries 
of $\mcl R_n$ along paths of disks which are contained in $\mcl R_n'$ is at most $ 2 n^2 \ep^{-1/(d_\gamma-\zeta )}$. 
This shows that~\eqref{eqn-rectangle-perc-truncated} holds for $\ep \in (0,\ep_*]$. 
Since $  D_{ \wh h_1^{\tr}}^\ep $ can only increase when $\ep$ decreases (since by definition we are taking an infimum over a smaller collection of sets of balls), it follows that~\eqref{eqn-rectangle-perc-truncated} is true for general $\ep > 0$ with $A =2 \ep_*^{-1/(d_\gamma - \zeta)} $. 

It remains only to prove the above claim. 
Let $\mcl S^*(\mcl R_n)$ be the graph whose squares are the same as the squares of $\mcl S(\mcl R_n)$, but with two squares considered to be adjacent if they share a corner or an edge, instead of only considering squares to be adjacent if they share an edge.
By planar duality, it suffices to show that if $p ,a_0, a_1$ are chosen appropriately, then for $\ep\in (0,\ep_*]$ it holds with probability at least $1-a_0 e^{-a_1 n}$ that there does \emph{not} exist a simple path in $\mcl S^*(\mcl R_n)$ from the top boundary to the bottom boundary of $\mcl S (\mcl R_n)$ consisting of squares for which $E_S^\ep$ does not occur. This will be proven by a standard argument for subcritical percolation.
By the definition~\eqref{eqn-wn-truncate} of $\wh h^{\tr}$, the event $E_S^\ep$ is a.s.\ determined by the restriction of the white noise $W$ to $S(2) \times \BB R_+$. In particular, $E_S^\ep$ and $E_{\wt S}^\ep$ are independent whenever $S(2)\cap \wt S(2) = \emptyset$. 
For each fixed deterministic simple path $ P$ in $\mcl S^*(\mcl R_n)$, we can find a set of at least $|P|/100$ squares hit by $P$ for which the expanded squares $S(2)$ are disjoint. 
By~\eqref{eqn-perc-prob}, applied once to each of these $|P|/100$ squares, if $\ep\in (0,\ep_*]$ then the probability that $E_S^\ep$ fails to occur for every square in $ P$ is at most $p^{| P|/100}$.

We now take a union bound over all simple paths $P$ in $\mcl S^*(\mcl R_n)$ connecting the top and bottom boundaries. For $k \in [n,2n^2]_{\BB Z}$, the number of such paths with $|P| = k$ is at most $ n 8^{k+1}$ since there are $2 n$ possible initial squares adjacent in the top boundary of $\mcl R_n$ and 8 choices for each step of the path. Combining this with the estimate in the preceding paragraph, we find that for $\ep \in (0,\ep_*]$ the probability of a top-bottom crossing of $\mcl S^*(\mcl R_n)$ consisting of squares for which $E_S^\ep$ does not occur is at most
\eqbn
  n \sum_{k=n}^{2 n^2} p^{k/100} 8^{k+1} , 
\eqen
which is bounded above by an exponential function of $n$ provided we take $p < 8^{-100}$.  
\end{proof}

From Lemma~\ref{lem-rectangle-perc}, the scaling properties of the field $\wh h$, and a union bound, we get the following. 

\begin{lem} \label{lem-rectangle-dist}
For each $\zeta\in (0,1)$, there exists $\lambda= \lambda(\zeta,\gamma) > 0$ such that the following is true.  
For $m\in\BB N$ and $\ep  > 0$, it holds with probability $1-O_m(e^{-\lambda m})$ as $m\rta\infty$, at a rate which is uniform in $\ep$, that for each $ 2^{-m+1} \times 2^{-m}$ rectangle $R\subset \BB S$ with corners in $2^{-m} \BB Z^2$, 
\eqb \label{eqn-rectangle-dist}
 D_{\wh h}^\ep\left( \bdy_{\op{L}} R , \bdy_{\op{R}} R ; R' \right)  
\leq \max\left\{ m^3 , \, \ep^{-\frac{1}{d_\gamma - \zeta  } }  2^{- \frac{1}{d_\gamma } \left(2  + \frac{\gamma^2}{2} - \zeta \right) m  }  \exp\left( \frac{\gamma}{d_\gamma}  \min_{z\in R'} \wh h_{2^{-m} }(z)  \right) \right\}.
\eqe  
\end{lem}

Note that there is a \emph{minimum}, rather than a maximum, inside the exponential on the right side of~\eqref{eqn-rectangle-dist}. 
The minimum and maximum values of $\wh h_{2^{-m}}$ on $R$ typically differ by a small multiple of $\log m$ due to continuity estimates for $\wh h_{2^{-m}}$ (Lemma~\ref{lem-use-btis}).

\begin{proof}[Proof of Lemma~\ref{lem-rectangle-dist}]
Fix $\wt\zeta\in (0,1)$ to be chosen later, in a manner depending only on $\zeta$ and $\gamma$. 
Also set 
\eqbn
n_m := \lfloor \log_2   m^{3/2} \rfloor ,\quad \forall m \in \BB N  
\eqen
(in fact, $n_m = \lfloor \log_2   m^p \rfloor$ for any $1 < p < 2$ would suffice). 

By Lemmas~\ref{lem-one-scale-max} and~\ref{lem-mid-scale-max} (the latter is applied with $\delta = 2^{-m}$ and $A = 2^{n_m} \asymp m^{3/2}$), it holds with exponentially high probability as $m\rta\infty$ that 
\eqb \label{eqn-use-mid-scale-compare}  
\max_{z\in \BB S} |\wh h_{2^{-m}}(z)| \leq (2+\wt\zeta) \log 2^m \quad \op{and} \quad
  \max_{z,w\in \BB S : |z-w| \leq   2^{-m+2}   } |\wh h_{   2^{-m - n_m  } }(z) - \wh h_{2^{-m}}(w)| \leq \wt\zeta \log 2^m .
\eqe
If $R$ is a $  2^{-m+1}  \times 2^{-m}$ rectangle and $u_R$ denotes its bottom left corner, then $   2^{m+n_m} (R-u_R) $ is the rectangle $\mcl R_{2^{n_m} }$ of Lemma~\ref{lem-rectangle-perc} and $2^{m + n_m} (R' - u_R)$ is the rectangle $\mcl R_{2^{n_m}}'$ of Lemma~\ref{lem-rectangle-perc}. Moreover, the field $ ( \wh h - \wh h_{2^{-m-n_m}})(2^{-m-n_m}\cdot + u_R)$ agrees in law with $\wh h$ and is independent from $\wh h_{2^{-m-n_m}   }$, which means that the associated Liouville graph distance $D_{( \wh h - \wh h_{2^{-m-n_m}})(2^{-m-n_m}\cdot + u_R)}^\ep$ agrees in law with $D_{\wh h}^\ep$ and is independent from $\wh h_{2^{-m-n_m}}$.
Using~\eqref{eqn-dist-scaling} with $\delta = 2^{-m-n_m}$ and $U $ equal to the interior of $ R'$, we therefore get that the conditional law of $ D_{\wh h}^\ep\left( \bdy_{\op{L}} R , \bdy_{\op{R}} R ; R' \right)$ given $\wh h_{  2^{-m - n_m} }$ is stochastically dominated by the law of 
\eqbn
 D_{\wh h }^{T_R\ep}\left(\bdy_{\op{L}} \mcl R_{2^{n_m} } , \bdy_{\op{R}} \mcl R_{2^{n_m} } ; \mcl R_{2^{n_m} }' \right) 
\quad \text{for} \quad 
T_R := 2^{\left( 2+\frac{\gamma^2}{2}\right) ( m+n_m) } \exp\left( -  \gamma \max_{z\in R'}  \wh h_{ 2^{-m -  n_m} }(z)  \right) .
\eqen
If~\eqref{eqn-use-mid-scale-compare} holds, then 
\allb \label{eqn-rectangle-dist-T}
T_R  
&\geq 2^{ \left(2  + \frac{\gamma^2}{2}  + o_m(1) + o_{\wt\zeta}(1) \right) m  }  \exp\left( -  \gamma \min_{z\in R'} \wh h_{2^{-m} }(z) \right) \notag \\
&\geq 2^{ \left(2  + \frac{\gamma^2}{2}  + o_m(1) + o_{\wt\zeta}(1) \right) m  }  \exp\left( - \frac{d_\gamma - \wt\zeta}{d_\gamma}  \gamma \min_{z\in R'} \wh h_{2^{-m} }(z) \right) , 
\alle
where the $o_{\wt\zeta}(1)$ and $o_m(1)$ are each deterministic and independent of $\ep$ and the $o_{\wt\zeta}(1)$ error is also independent of $m$. Note that in the first line, we switched from the maximum of $\wh h_{  2^{-m - n_m}}(z)$ to the minimum of $\wh h_{2^{-m}}(z)$ (which gives a stronger estimate than the maximum) using the second inequality in~\eqref{eqn-use-mid-scale-compare}. Also, in the second line, we absorbed a small power of $e^{\wh h_{2^{-m}}(z)}$ into a factor of $2^{m o_{\wt\zeta}(1)}$ using the first inequality in~\eqref{eqn-use-mid-scale-compare}. By Lemma~\ref{lem-rectangle-perc} (applied with $T_R \ep$ in place of $\ep$ and $2^{n_m}$ in place of $n$) and a union bound over $O_m(2^{2m})$ rectangles $R$, we obtain the statement of the lemma upon choosing $\wt\zeta$ sufficiently small, in a manner depending only on $\zeta$ and $\gamma$. 
\end{proof}

The proof of Proposition~\ref{prop-diam} will use the following consequence of Lemma~\ref{lem-rectangle-dist}.

\begin{lem} \label{lem-rectangle-dist-multi}
Fix $\zeta\in (0,1)$ and $\beta \in (0,1)$. 
It holds with polynomially high probability as $\ep \rta 0$ that for each $m\in\BB N$ with $m \geq \log_2 \ep^{-\beta}$ and each $  2^{-m+1} \times 2^{-m}$ rectangle $R\subset \BB S$ with corners in $2^{-m} \BB Z^2$, 
\eqb \label{eqn-rectangle-dist-multi}
D_{\wh h}^\ep\left( \bdy_{\op{L}} R , \bdy_{\op{R}} R ; R' \right)  
\leq   \max\left\{ m^3, \ep^{-\frac{1}{d_\gamma -\zeta}}  2^{- \frac{1}{d_\gamma } \left(2  + \frac{\gamma^2}{2} - 2\gamma - \zeta\right) m  }  \right\} ;
\eqe 
and the same holds with $2^{-m} \times   2^{-m+1}$ rectangles but with $\bdy_{\op{B}}$ and $\bdy_{\op{T}}$ in place of $\bdy_{\op{L}}$ and $\bdy_{\op{R}}$. 
\end{lem}
\begin{proof}
By Lemma~\ref{lem-one-scale-max}, it holds with exponentially high probability as $m\rta\infty$ that
\eqbn
 \max_{z\in \BB S} |\wh h_{2^{-m} }(z)  | \leq  \left( 2 + \zeta \right) \log 2^m  .
\eqen
Combining this with Lemma~\ref{lem-rectangle-dist}, taking a union bound over all integers $m\geq \log_2 \ep^{-\beta}$, and possibly shrinking $\zeta$ concludes the proof. 
\end{proof}

\begin{proof}[Proof of Proposition~\ref{prop-diam}] 
See Figure~\ref{fig-diam}, right, for an illustration of the proof. 
Fix $\wt\zeta \in (0,1)$ and $\beta \in (0,1)$, to be chosen later in a manner depending only on $\zeta$, and let $E^\ep = E^\ep(\wt\zeta,\beta)$ be the event of Lemma~\ref{lem-rectangle-dist-multi} with $\wt\zeta$ in place of $\zeta$, so that $E^\ep$ occurs with polynomially high probability as $\ep\rta 0$. On the event $E^\ep$, we can choose for each $m\geq \log_2 \ep^{-\beta}$ and each $2^{-m+1}\times 2^{-m}$ (resp.\ $2^{-m}\times 2^{-m+1}$) rectangle $R\subset \BB S$ with corners in $2^{-m}\BB Z^2$ a simple path $P_R$ in $R$ from $\bdy_{\op{L}} R$ to $\bdy_{\op{R}} R$ (resp.\ $\bdy_{\op{B}} R$ to $\bdy_{\op{T}} R$) which can be covered by at most
\eqbn
\max\left\{ m^3, \ep^{-\frac{1}{d_\gamma -\wt\zeta}}  2^{- \frac{1}{d_\gamma } \left(2  + \frac{\gamma^2}{2} - 2\gamma - \wt\zeta\right) m  }  \right\} 
\eqen
Euclidean balls of $\mu_{\wh h}$-mass at most $\ep$ which are contained in $R' \subset \BB S(1/2)$. For a dyadic square $S\subset \BB S$ with side length at most $\ep^\beta$, let $X_S$ be the \#-sign shaped set which is the union of the paths $P_R$ corresponding to the four $2^{-m-1}\times 2^{-m}$ or $2^{-m}\times 2^{-m-1}$ rectangles $R$ as above which are contained in $S$. Then $X_S$ is connected and contained in $S$. Furthermore, if $\wt S$ is one of the four dyadic children of $S$, then $X_{\wt S} \cap X_S\not= \emptyset$. We will prove the proposition by constructing connected paths between points of $\BB S$ using the $X_S$'s. 

Consider a dyadic square $S$ with $|S| = 2^{- \lceil \log_2   \ep^{-\beta} \rceil}$ and a point $z\in S$. Let $S = S_0^z , S_1^z,\dots$ be the sequence of dyadic descendants of $S$ containing $z$ (enumerated so that $S_{j-1}^z$ is the dyadic parent of $S_j^z$ for each $j \in\BB N$).
The preceding paragraph shows that on $E^\ep$, it holds for each $j \in \BB N$ that 
\eqb \label{eqn-dyadic-seq-diam}
\max_{ u \in X_{S_j^z} } D_{\wh h}^\ep \left( u , X_{S_{j-1}^z} ; \BB S(1/2) \right)\leq 4 \max\left\{ (\log_2 (1/|S_j^z|) )^3, \ep^{-\frac{1}{d_\gamma -\wt\zeta}}  |S_{j}^z|^{ \frac{1}{d_\gamma } \left(2  + \frac{\gamma^2}{2} - 2\gamma - \wt\zeta\right)   }  \right\} .
\eqe  

By Lemma~\ref{lem-max-ball-radius}, there exists $A=A(\gamma) > 0$ such that with polynomially high probability as $\ep\rta 0$, each Euclidean ball of $\mu_{\wh h}$-mass $\ep$ which intersects $\BB S$ is contained in $\BB S(1/2)$ and has Euclidean radius at least $2\ep^A$. This in particular implies that whenever $\wt S \subset\BB S$ is a dyadic square with $|\wt S| \leq \ep^{A}$, we have $D_{\wh h}^\ep\left( u ,v ; \BB S(1/2) \right) = 1$ for each $u,v\in \wt S$. If this is the case, we may sum the estimate~\eqref{eqn-dyadic-seq-diam} over all $j \in [1, \log_2 \ep^{\beta - A}]_{\BB Z}$ to find that 
\eqb \label{eqn-dist-to-X}
D_{\wh h}^\ep\left( z , X_S;  \BB S(1/2) \right) 
\preceq  \ep^{-\frac{1}{d_\gamma -\wt\zeta} + \frac{\beta}{d_\gamma } \left(2  + \frac{\gamma^2}{2} - 2\gamma - \wt\zeta\right)   }   +  (\log_2 \ep^{\beta - A} )^4 
\preceq  \ep^{ -\frac{1}{d_\gamma -\wt\zeta} + \frac{\beta}{d_\gamma } \left(2  + \frac{\gamma^2}{2} - 2\gamma - \wt\zeta\right)   } ,
\eqe 
with the implicit constant in $\preceq$ deterministic and independent of $\ep$ and $z$. 

The bound~\eqref{eqn-dist-to-X} holds simultaneously for every dyadic square $S$ of side length $ 2^{- \lceil \log_2  \ep^{-\beta} \rceil}$ and every $z\in S$ with polynomially high probability as $\ep\rta 0$. Furthermore, with polynomially high probability as $\ep\rta 0$ each $X_S$ for dyadic squares $S$ with $|S| =2^{- \lceil\log_2   \ep^{-\beta} \rceil}$ can be covered by at most $4 \ep^{ -\frac{1}{d_\gamma -\wt\zeta} + \frac{\beta}{d_\gamma } \left(2  + \frac{\gamma^2}{2} - 2\gamma - \wt\zeta\right)   }$ Euclidean balls contained in $\BB S(1/2)$, each of which has $\mu_{\wh h}$-mass at most $\ep$. It follows that with polynomially high probability as $\ep\rta 0$, 
\eqb \label{eqn-dyadic-diam}
\max_{z,w\in S} D^\ep_{\wh h}\left( z , w ; \BB S(1/2) \right) \preceq \ep^{ -\frac{1}{d_\gamma -\wt\zeta} + \frac{\beta}{d_\gamma } \left(2  + \frac{\gamma^2}{2} - 2\gamma - \wt\zeta\right)   } ,
\quad \forall \text{dyadic $S\subset \BB S$ with $|S| = 2^{-  \lceil \log_2  \ep^{-\beta} \rceil}$}.
\eqe 
By summing the bound~\eqref{eqn-dyadic-diam} over $O_\ep(\ep^{-\beta})$ dyadic squares of side length $2^{- \lceil \log_2  \ep^{-\beta} \rceil}$ whose union contains a path between two given points of $\BB S$, we get that with polynomially high probability as $\ep\rta 0$, 
\eqbn
\max_{z,w\in S}   D^\ep_{\wh h}\left( z , w ; \BB S(1/2) \right) \preceq \ep^{ -\frac{1}{d_\gamma -\wt\zeta} + \frac{\beta}{d_\gamma } \left(2  + \frac{\gamma^2}{2} - 2\gamma - \wt\zeta\right) - \beta} .
\eqen
We now obtain~\eqref{eqn-diam} by choosing $\beta$ and $\wt\zeta$ sufficiently small, in a manner depending only on $\zeta$.
\end{proof}

\begin{proof}[Proof of Theorem~\ref{thm-diam}]
The bound for point-to-point distance~\eqref{eqn-d-def} was already proven in Lemma~\ref{lem-whole-plane-d}, so we only need to prove~\eqref{eqn-d-diam}. 
For a compact set $K$ and an open set $U$ with $K \subset U$ as in the theorem statement, choose finitely many squares $S_1,\dots,S_k \subset U$ whose union covers $K$ and such that each of the expanded squares $S_j(1/2)$ for $j=1,\dots,k$ is also contained in $U$. Also fix $\zeta\in (0,1)$. 

By Lemma~\ref{lem-tr-compare-square}, the conclusion of Proposition~\ref{prop-diam} remains true with the white-noise field $\wh h$ replaced with the whole-plane GFF $h$. 
If $C> 0$ and $z\in\BB C$, then $h(C^{-1}(\cdot - z)) - h_C(z)$ agrees in law with $h$, equivalently the Liouville graph distance satisfies $D_{h(C^{-1}(\cdot-z))}^{\ep e^{\gamma h_C(z)}} \eqD D_h^\ep$. Since each $h_C(z)$ is a Gaussian random variable, we find that the conclusion of Proposition~\ref{prop-diam} remains true with $h$ in place of $\wh h$ and with $\BB S$ replaced with any other square $S\subset \BB C$ (with the rate of convergence of the probability as $\ep\rta 0$ depending on the square). 
Applying this to each of the squares $S_j$ above, we find that with polynomially high probability as $\ep\rta 0$,
\eqb \label{eqn-diam-proof-upper}
\max_{z,w\in K} D_h^\ep\left( z , w ; U \right) \leq \ep^{-\frac{1}{d_\gamma - \zeta}} .
\eqe 

To bound $ D_h^\ep\left( K , \bdy U \right)$, we first use~\cite[Lemma 6.1]{dzz-heat-kernel} to get that if $h^{\BB S(1)}$ is a zero-boundary GFF on $\BB S(1)$, then with polynomially high probability as $\ep\rta 0$, 
\eqb
D_{h^{\BB S(1)}}^\ep\left( \bdy \BB S , \bdy \BB S(1/2) \right) \geq \ep^{-\frac{1}{d_\gamma+\zeta}} .
\eqe
By the same argument as in the preceding paragraph, the same is true with $h$ in place of $h^{\BB S(1)}$ and with $\BB S$ replaced with any other square $S\subset\BB C$. Any path from $K$ to $\bdy U$ must cross $S_j(1/2) \setminus S_j$ for one of the squares $S_j$, $j=1,\dots,k$, above. We therefore obtain that with polynomially high probability as $\ep\rta 0$, 
\eqb\label{eqn-diam-proof-lower}
D_h^\ep\left( K , \bdy U \right)  \geq \ep^{-\frac{1}{d_\gamma + \zeta}}  .
\eqe 

By Lemma~\ref{lem-whole-plane-d} (applied for $z,w \in K$ and for $z\in K$ and $w\in \bdy U$, respectively) we also get a lower bound of $\ep^{-\frac{1}{d_\gamma + \zeta}}$ for the left side of~\eqref{eqn-diam-proof-upper} and an upper bound of $\ep^{-\frac{1}{d_\gamma - \zeta}}$ for the right side of~\eqref{eqn-diam-proof-lower} which each hold with polynomially high probability as $\ep\rta 0$. Taking a union bound over dyadic values of $\ep$ concludes the proof.
\end{proof}

\subsection{Lower bound for LFPP distances}
\label{sec-lfpp-lower}

In this subsection we will prove the lower bound for LFPP distances from Theorem~\ref{thm-lfpp-compare}, building on the estimates proven in Section~\ref{sec-diam}. In fact, we will prove the following slightly more quantitative statement.

\begin{prop} \label{prop-lfpp-lower}
Let $h$ be a whole-plane GFF normalized so that $h_1(0) = 0$. 
Also let $U\subset \BB C$ be a bounded open set and let $K\subset U$ be a compact set. 
For each $\zeta \in (0,1)$, it holds with polynomially high probability as $\delta \rta 0$ that the LFPP distance with exponent $\xi=\gamma/d_\gamma$ satisfies
\eqb \label{eqn-lfpp-lower}
D_{h,\LFPP}^\delta\left( K , \bdy U \right) \geq \delta^{1 - \frac{2}{d_\gamma}  - \frac{\gamma^2}{2 d_\gamma}  + \zeta } .
\eqe
\end{prop}

The basic idea of the proof of Proposition~\ref{prop-lfpp-lower} is as follows. We choose $\delta=\delta_\ep$ to be comparable to a small (but fixed) power of $\ep$ and consider a path from $K$ to $\bdy U$ along which the integral of $e^{\gamma h_\delta(z)}$ is close to minimal. We then concatenate the crossings of the  $2\delta_\ep \times \delta_\ep$ and $\delta_\ep\times 2\delta_\ep$ rectangles traversed by this path, as afforded by Lemma~\ref{lem-rectangle-dist}, to produce another path from $K$ to $\bdy U$ such that the number of $\mu_{\wh h}$-mass $\ep$ disks needed to cover this second path can be bounded above in terms of $D_{h,\LFPP}^{\delta_\ep}\left( K , \bdy U \right)$ (see Figure~\ref{fig-lfpp-lower}). Plugging in our known lower bound for $D_{\wh h}^\ep\left( K , \bdy U \right)$ (which comes from Theorem~\ref{thm-diam} and Lemma~\ref{lem-tr-compare-square}) then gives a lower bound for $D_{h,\LFPP}^{\delta_\ep}\left( K , \bdy U \right)$.  

For most of the proof, we will work with a zero-boundary GFF $h^{\BB S(1)}$ on the square $\BB S(1) = [-1,2]^2$ instead of a whole-plane GFF (mostly because of Lemma~\ref{lem-circle-avg-approx}). It will also be convenient to work with an approximate version of LFPP distances for which the paths interact with squares in a nice way (this is a LFPP analogue of the approximate Liouville graph distance considered in~\cite[Section 3]{dzz-heat-kernel}).
  
For $\delta>0$, let $m_\delta := \lceil \log_2  \delta^{-1} \rceil$ be the smallest integer with $2^{m_\delta} \geq \delta^{-1}$ and let $\mcl S_{2^{-m_\delta}}$ be the set of dyadic squares contained in $\BB S$ with side length $2^{-m_\delta}$.
For $z,w\in\BB S$ and $\delta>0$, define the \emph{approximate $\delta$-LFPP distance} from $z$ to $w$ with respect to $\wh h$ by
\eqb
\wh D_{\wh h ,\LFPP}^\delta(z,w ; \BB S) := \min_{S_0,\dots,S_k} \sum_{j =0}^{k} \delta e^{\xi \wh h_\delta(v_{S_j})}
\eqe 
where the minimum is over all sequences of distinct squares $S_0 ,\dots,S_k \in \mcl S_{2^{-m_\delta}}$ such that $z\in S_0$, $w\in S_k$, and $S_j$ and $S_{j-1}$ share a side for each $j=1,\dots,k$. Here we recall that $v_S$ denotes the center of $S$.

\begin{prop} \label{prop-lfpp-approx}
There is a coupling of $\wh h$ and $h^{\BB S(1)}$ such that the following is true. For each $\zeta\in (0,1)$ and each $\xi>0$, it holds with polynomially high probability as $\delta\rta 0$ that for each $z,w\in\BB S$, 
\eqb \label{eqn-lfpp-approx}
\delta^\zeta \left( \wh D_{\wh h ,\LFPP}^\delta(z,w ; \BB S) - \delta e^{\xi \wh h_\delta(v_{S_z}) } \right) \leq  D_{h^{\BB S(1)} ,\LFPP}^\delta(z,w;\BB S)  \leq \delta^{-\zeta} \wh D_{\wh h ,\LFPP}^\delta(z,w ; \BB S)  , 
\eqe 
where $S_z$ is the square of $\mcl S_{2^{-m_\delta}}$ containing $z$ for which $\wh h_\delta(v_{S_z})$ is maximized (this is the unique square containing $z$ if $z$ is not on the boundary of a square).
\end{prop}

The reason for the $-\delta e^{\xi \wh h_\delta(z)}$ in the lower bound in~\eqref{eqn-lfpp-approx} is that if $z$ and $w$ are contained in the same square of $\mcl S_{2^{-m_\delta}}$, then $\wh D_{h^{\BB S(1)} ,\LFPP}^\delta(z,w ; \BB S) = \delta e^{\xi \wh h_\delta(v_{S_z}))} $, whereas $D_{h^{\BB S(1)} ,\LFPP}^\delta(z,w;\BB S)$ might be much smaller than $ \delta e^{\xi \wh h_\delta(v_{S_z}))} $ (e.g., if $z=w$).

\begin{proof}[Proof of Proposition~\ref{prop-lfpp-approx}]
By Lemma~\ref{lem-use-btis} and~\ref{lem-circle-avg-approx}, and the triangle inequality, we can couple $\wh h$ and $h^{\BB S(1)}$ in such a way that for each $\zeta \in (0,1)$, it holds with polynomially high probability as $\delta \rta 0$ that
\eqb \label{eqn-use-circle-avg-approx} 
\max_{z,w\in \BB S: |z-w| \leq 4 \delta} \left( |h^{\BB S(1)}_\delta(z) - \wh h_\delta(w)| \vee |\wh h_\delta(z)  - \wh h_\delta(w)| \right) \leq \frac{\zeta}{2\xi} \log\delta^{-1} .
\eqe
Henceforth assume that~\eqref{eqn-use-circle-avg-approx} holds. We will show that~\eqref{eqn-lfpp-approx} holds. 
\medskip

\noindent\textit{Upper bound.} We first prove the second inequality in~\eqref{eqn-lfpp-approx}, which is easier. For $z,w\in \BB S$, we can find distinct squares $S_0,\dots,S_k \in \mcl S_{2^{-m_\delta}}$ such that $z\in S_0$, $w\in S_k$, $S_j$ and $S_{j-1}$ share a side for each $j=1,\dots,k$, and 
\eqbn
\sum_{j=0}^k \delta e^{\xi \wh h_\delta(v_{S_j})} \leq 2 \wh D_{\wh h ,\LFPP}^\delta(z,w;\BB S) . 
\eqen
Let $z_0 := z$, let $z_{k+1} := w$, and choose $z_j \in S_j \cap S_{j-1}$ for each $j=1,\dots,k$. 
Let $P$ be the concatenation of the line segments $[z_j,z_{j+1}]$ for $j=0,\dots,k$, traversed at unit speed. The segment $[z_j,z_{j+1}]$ is contained in $S_j$, so~\eqref{eqn-use-circle-avg-approx} implies that the maximum value of the circle average $h^{\BB S(1)}_\delta$ on this line segment is at most $\wh h_\delta(v_{S_j}) +\frac{\zeta}{2\xi} \log \delta^{-1}$. Summing over all such segments gives the desired bound (up to a deterministic constant factor which can be ignored by slightly shrinking $\zeta$). 
\medskip

\noindent\textit{Lower bound.}
Fix $z,w\in \BB S$ and let  $P : [0,T]\rta \BB S$ be a piecewise continuously differentiable simple path from $z$ to $w$, parametrized by Euclidean unit speed, with 
\eqb \label{eqn-lfpp-geo}
\int_0^T e^{\xi h^{\BB S(1)}_{\delta}(P (t))}   \, dt \leq 2 D_{h^{\BB S(1)},\LFPP}^{\delta} \left(z,w ; \BB S \right) .
\eqe 
We first construct an approximation $\wt P : [0,\wt T] \rta \BB S$ of $P$ such that $\wt P^{-1}(S)$ is either empty or a single connected interval for each square $S\in\mcl S_{2^{-m_\delta}}$ via the following inductive ``loop erasing" procedure.  
Let $t_0 = z = P(0)$ and let $S_0$ be a square of $\mcl S_{2^{-m_\delta}}$ containing $z$ (we make an arbitrary choice if there is more than one). 
Inductively, suppose that $j\in\BB N$ and times $0\leq t_0\leq \dots \leq t_{j-1} \leq T$ and squares $S_0,\dots,S_{j-1} \in \mcl S_{2^{-m_\delta}}$ have been defined in such a way that $t_i$ is the last time $t \in [0,T]$ with $P(t) \in  S_{i-1}$ for each $i = 1,\dots,j-1$. 
Let $t_j$ be the last time $t \in [0,T]$ for which $P(t) \in S_{j-1}$. 
If $t_j = T$, let $S_j = S_{j-1}$. 
If $t_j \not=T$, then since each $t_i$ for $i\leq j$ is the \emph{last} time that $P$ is in $S_i$, there must be a square of $\mcl S_{2^{-m_\delta}}$ other than $S_0,\dots,S_{j-1}$ with $P(t_j)$ on its boundary (so that $P$ has somewhere to go after time $t_j$). Let $S_j$ be such a square, chosen in such a way that $P((t_j,t_j+\ep])$ intersects $S_j$ for each $\ep >0$ (we make an arbitrary choice if there is more than one such square). 

Let $\mcl J$ be the smallest $j\in\BB N$ with $t_{j+1} = T$. 
Let $\wt P : [0,\wt T ] \rta \BB S$ be the concatenation of the straight line segments $[P(t_j), P(t_{j+1})] \subset S_j$ for $j = 0,\dots,\mcl J$, traversed at unit speed. Since the squares $S_j$ for $j=1,\dots,\mcl J$ are distinct, it follows that $\wt P^{-1}(S_j) = [t_j , t_{j+1}]$ for each $j =1,\dots,\mcl J$. 
 
We next show that
\eqb \label{eqn-approx-lfpp-path}
\int_0^{\wt T} e^{\xi \wh h_\delta( \wt P (t))}   \, dt \leq 2 \delta^{-\zeta/2} D_{h^{\BB S(1)},\LFPP}^\delta \left(z,w  ; \BB S \right) .
\eqe 
For this purpose, let $\wt t_j$ for $j \in \BB N_0$ be the unique time for which $\wt P(\wt t_j) = P(t_j)$. 
Then $\wt P|_{[\wt t_{j }, \wt t_{j+1}]}$ is a straight line segment contained in the square $S_j$, so the Euclidean length of $\wt P|_{[\wt t_{j }, \wt t_{j+1}]}$ is at most the Euclidean length of $P|_{[t_{j },t_{j+1}]}$. Furthermore, since $\wt P$ is parameterized by unit speed, $\wt t_{j+1} - \wt t_{j} \leq \sqrt 2 \times 2^{-m_\delta}  \leq 4\delta$, so by~\eqref{eqn-use-circle-avg-approx} and since $P$ has unit speed,
\eqbn
\max_{t \in [\wt t_{j-1} , \wt t_j]} \wh h_\delta(\wt P(t)) \leq  \frac{\zeta}{2\xi} \log \delta^{-1} + \min_{t \in [ t_{j-1} , t_{j-1}  + \wt t_j - \wt t_{j-1} ]}  h^{\BB S(1)}_\delta(P(t)) .
\eqen
By combining this with~\eqref{eqn-lfpp-geo}, we get~\eqref{eqn-approx-lfpp-path}. 

We will now argue that, for the squares $S_j$ defined above, 
\eqb \label{eqn-approx-lfpp-square}
\sum_{j=0}^{\mcl J } \delta e^{\xi \wh h_\delta(v_{S_j})} \leq 4 \delta^{-\zeta/2} \int_0^{\wt T} e^{\xi \wh h_\delta( \wt P (t))}   \, dt ,
\eqe
which combined with~\eqref{eqn-approx-lfpp-path} gives the first inequality in~\eqref{eqn-lfpp-approx} (after adjusting $\zeta$ appropriately). 

To prove~\eqref{eqn-approx-lfpp-square}, we need to deal with the squares $S_j$ for which $\wt t_{j+1} - \wt t_j$ is very small, in which case $\delta e^{\xi \wh h_\delta(v_{S_j})}$ is a poor approximation for the integral of $ e^{\xi \wh h_\delta(P(t))}$ over $[\wt t_{j } , \wt t_{j+1}]$. To this end, for $j \in \BB N$ we let $J_j$ be the largest $j'\leq j$ for which $\wt t_{j+1} - \wt t_{j } \geq \delta^{1+\zeta/2}$. We claim that $j-J_j \leq 5$ for $j=1,\dots,\mcl J$. Indeed, if $j-J_j \geq 6$, then $\wt P$ travels Euclidean distance at most $4\delta^{1+\zeta/2}$ between times $\wt t_{J_{j }+1}$ and $\wt t_{J_{j }+6}$, so can hit at most 4 possible squares during this time, which contradicts the fact that the squares $S_i$ for $i=J_j,\dots,J_j+5$ are distinct. It therefore follows from~\eqref{eqn-use-circle-avg-approx} that 
\eqbn
\delta e^{\xi \wh h_\delta(v_{S_j})} \leq \delta^{-\zeta/2} \int_{\wt t_{J_{j }}}^{\wt t_{J_j+1}} e^{\xi \wh h_\delta(\wt P(t) )} \, dt .
\eqen
We now sum over all $j$ with $\wt t_j \leq T$ and use~\eqref{eqn-approx-lfpp-path} and the fact that each term on the right is counted at most 4 times (since $j-J_j\leq 5$) to get~\eqref{eqn-approx-lfpp-square}.
\end{proof}

We can now prove the analogue of Proposition~\ref{prop-lfpp-lower} for the zero-boundary GFF.

\begin{figure}[t!]
 \begin{center}
\includegraphics[scale=1]{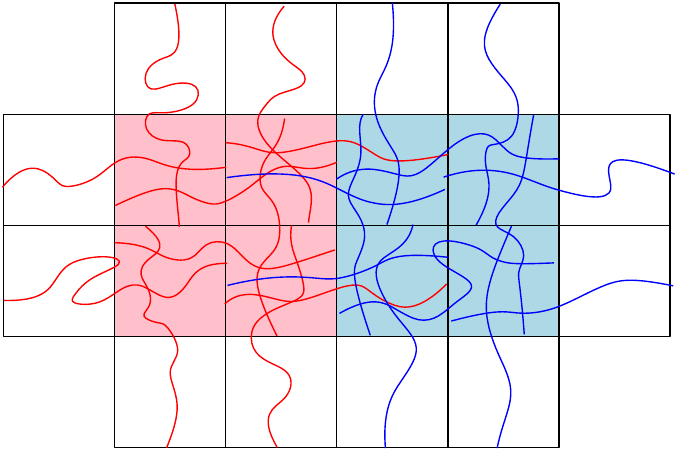}
\vspace{-0.01\textheight}
\caption{ The sets $Y_S$ used in the proof of Proposition~\ref{prop-lfpp-lower0} for two adjacent $\delta_\ep\times \delta_\ep$ squares $S$ (pink and light blue). Each of these sets is the union of 12 paths which cross the $\delta_\ep\times (\delta_\ep/2)$ or $(\delta_\ep/2) \times \delta_\ep$ rectangles which intersect the square. Each $Y_S$ is connected and the sets corresponding to adjacent squares intersect. 
}\label{fig-lfpp-lower}
\end{center}
\vspace{-1em}
\end{figure}

\begin{prop} \label{prop-lfpp-lower0}
Let $ K \subset U \subset \BB S$ with $K$ compact and $U$ open. Let $h^{\BB S(1)}$ be a zero-boundary GFF on $\BB S(1)$. 
For each $\zeta \in (0,1)$, it holds with polynomially high probability as $\delta \rta 0$ that 
\eqb \label{eqn-lfpp-lower0}
D_{h,\LFPP}^\delta\left(K , \bdy U \right) \geq \delta^{1 - \frac{2}{d_\gamma}  - \frac{\gamma^2}{2 d_\gamma}  + \zeta } .
\eqe
\end{prop}
\begin{proof} 
Let $\beta  \in \left( 0, \frac{2}{(2+\gamma)^2} \right)$ be arbitrary (e.g., we could take $\beta=\frac{1}{(2+\gamma)^2}$). 
We will compare $\ep^\beta$-LFPP distances to $\ep$-Liouville graph distances (with respect to $\wh h$), then set $\delta=\ep^\beta$. By Proposition~\ref{prop-lfpp-approx}, it suffices to prove a lower bound for \emph{approximate} $\ep^\beta$-Liouville graph distances, i.e., it is enough to show that for each fixed $K \subset U  \subset \BB S$ as in the statement of the lemma, it holds with polynomially high probability as $\ep\rta 0$ that
\eqb \label{eqn-lfpp-lower-show}
 \wh D_{\wh h ,\LFPP}^{\ep^\beta}\left( K , \bdy U  \right) \geq \ep^{\beta\left( 1 - \frac{2}{d_\gamma} - \frac{\gamma^2}{2d_\gamma} \right) + \zeta} .
\eqe
Note that the error term $ \ep^\beta e^{\frac{\gamma}{d_\gamma} \wh h_{\ep^\beta}(v_{S_z}) }$ coming from the left side of~\eqref{eqn-lfpp-approx} does not pose a problem here: indeed, Lemma~\ref{lem-one-scale-max} shows that with polynomially high probability as $\ep\rta 0$, this term is at most $\ep^{\beta(1- 2\gamma/d_\gamma)}$ uniformly over all $z\in \BB S$ and we have $1-2\gamma/d_\gamma > 1-2/d_\gamma -\gamma^2/(2d_\gamma)$. 

Let $ \wt\zeta \in (0,1)$ which we will choose later, in a manner depending only on $\beta$, $\zeta$, and $\gamma$. 
Also set
\eqbn
\delta_\ep := 2^{- \lceil \log_2  \ep^{-\beta } \rceil} .
\eqen 

\noindent \textit{Step 1: regularity events.} We first define a regularity event, giving bounds for $\wh h_{\ep^\beta}$ and the $\ep$-approximate Liouville graph distance.
By Lemma~\ref{lem-mid-scale-max} (applied with $A =  \ep^\beta / \delta_\ep \leq 2$) and Lemma~\ref{lem-one-scale-max}, it holds with polynomially high probability as $\ep\rta 0$ that 
\eqb \label{eqn-field-control}
 |\wh h_{\delta_\ep}(z)| \vee |\wh h_{\ep^\beta}(z)|  \leq (2\beta +\wt\zeta)\log\ep^{-1} \quad \op{and} \quad
  |\wh h_{\delta_\ep}(z) - \wh h_{\ep^\beta}(z)| \leq \wt\zeta \log \ep^{-1} ,\quad \forall z\in\BB S .
\eqe
By Lemma~\ref{lem-rectangle-dist} (applied with $N=2$, $m = \lfloor \log_2  \ep^{-\beta } \rfloor -1$, and a sufficiently small choice of $\zeta$), it holds with polynomially high probability as $\ep \rta 0$ that for each $  \delta_\ep \times (\delta_\ep/2) $ rectangle $R\subset \BB S$ with corners in $\frac{\delta_\ep}{2} \BB Z^2$, 
\eqb \label{eqn-use-rectangle-dist}
D_{\wh h}^\ep\left( \bdy_{\op{L}} R , \bdy_{\op{R}} R ; R' \right)  
\leq \max\left\{ (\log \ep^{-1})^3 , \, \ep^{-\frac{1}{d_\gamma }  + \frac{\beta}{d_\gamma } \left(2  + \frac{\gamma^2}{2} \right)  - \wt\zeta }  \exp\left(   \frac{\gamma}{d_\gamma}  \min_{z\in R'} \wh h_{\delta_\ep }(z)  \right) \right\} ;
\eqe
and the same holds with $(\delta_\ep/2) \times  \delta_\ep$ rectangles and with $\bdy_{\op{B}}$ and $\bdy_{\op{T}}$ in place of $\bdy_{\op{L}}$ and $\bdy_{\op{R}}$. 
Henceforth assume that this is the case and that~\eqref{eqn-field-control} holds. 
\medskip

\noindent\textit{Step 2: bounding Liouville graph distances along paths of squares.}
Since $\beta  <  2/(2+\gamma)^2$, if we choose $\wt\zeta$ sufficiently small (in a manner depending only on $\beta$ and $\gamma$) then the first inequality in~\eqref{eqn-field-control} shows that the second term in the maximum on the right side of~\eqref{eqn-use-rectangle-dist} is larger than the first. Using this together with the second inequality in~\eqref{eqn-field-control}, we see that~\eqref{eqn-use-rectangle-dist} can be replaced with
\eqb \label{eqn-use-rectangle-dist'}
D_{\wh h}^\ep\left( \bdy_{\op{L}} R , \bdy_{\op{R}} R ; R' \right)  
\leq  \ep^{-\frac{1}{d_\gamma  }  + \frac{\beta}{d_\gamma } \left(2  + \frac{\gamma^2}{2}  \right)  - o_{\wt\zeta}(1) }  
\exp\left(   \frac{\gamma}{d_\gamma}  \min_{z\in R'} \wh h_{\ep^\beta }(z)   \right) ,
\eqe 
with the rate of the $o_{\wt\zeta}(1)$ deterministic and $\ep$-independent.

Recall that we are assuming that the event described above~\eqref{eqn-use-rectangle-dist} occurs.
Let $N$ be the right side of~\eqref{eqn-use-rectangle-dist'}. 
Then we can choose for each $ \delta_\ep \times (\delta_\ep/2)$ (resp.\ $(\delta_\ep/2) \times  \delta_\ep$) rectangle $R \subset \BB S$ with corners in $\frac{\delta_\ep}{2} \BB Z$ a simple path $P_R$ in $R$ from $\bdy_{\op{L}} R$ to $\bdy_{\op{R}} R$ (resp.\ $\bdy_{\op{B}} R$ to $\bdy_{\op{T}} R$) which can be covered by at most $N$ Euclidean balls of $\mu_{\wh h}$-mass at most $\ep$, each of which is contained in $R'$.

For each of the $\delta_\ep$-side length squares $S\in \mcl S_{\delta_\ep}$, let $Y_S$ be the union of the paths $P_R$ over the at most twelve $ \delta_\ep \times (\delta_\ep/2)$ or $(\delta_\ep/2) \times  \delta_\ep$ rectangles $R$ as above which overlap with $S$. See Figure~\ref{fig-lfpp-lower} for an illustration. Then $Y_S$ is connected (but not contained in $ S$) and, since the center $v_S$ is contained in each of the above rectangles $R$ and $R'\subset S(1/2)$ for each such rectangle $R$,  
\eqb \label{eqn-lfpp-max-Y}
\max_{z,w\in Y_S} D_{\wh h}^\ep\left( z,w ; S(1/2) \right)
\leq 12 \ep^{-\frac{1}{d_\gamma }  + \frac{\beta}{d_\gamma } \left(2  + \frac{\gamma^2}{2} \right) - o_{\wt\zeta}(1)  }  \exp\left(   \frac{\gamma}{d_\gamma}   \wh h_{\ep^\beta }(v_S)   \right)  .
\eqe 
Furthermore, if $S, \wt S \in \mcl S_{\delta_\ep}$ are two squares which share a side, then $Y_S\cap Y_{\wt S} \not=\emptyset$.  
\medskip

\noindent\textit{Step 3: comparison to approximate LFPP.}
Let $S_0,\dots,S_k \in \mcl S_{\delta_\ep}$ be a sequence of distinct squares such that $K  \cap S_0 \not=\emptyset$, $\bdy U \cap  S_k \not=\emptyset$, $S_j$ and $S_{j-1}$ share a side for each $j=1,\dots,k$, and
\eqb \label{eqn-approx-lfpp-setup}
\sum_{j=0}^k \ep^\beta \exp\left( \frac{\gamma}{d_\gamma} \wh h_{\ep^\beta}(v_{S_j}) \right) \leq 2 \wh D_{\wh h ,\LFPP}^{\ep^\beta}\left( K , \bdy U \right)  . 
\eqe

By~\eqref{eqn-lfpp-max-Y} and since $Y_{S_j}\cap Y_{S_{j-1}}\not=\emptyset$ for $j=1,\dots,k$,  
\eqb \label{eqn-lfpp-lower-sum}
D_{\wh h}^\ep\left(S_0 , S_k ; \BB S(1/2) \right) 
\preceq  \ep^{-\frac{1}{d_\gamma }  + \frac{\beta}{d_\gamma } \left(2  + \frac{\gamma^2}{2}  \right) - o_{\wt\zeta}(1)  }   \sum_{j=0}^k \exp\left(  \frac{\gamma}{d_\gamma} \wh h_{\ep^\beta}(v_{S_j})   \right)  
\eqe 
with a universal implicit constant. Comparing~\eqref{eqn-lfpp-lower-sum} to~\eqref{eqn-approx-lfpp-setup} shows that with polynomially high probability as $\ep\rta 0$,
\allb \label{eqn-lfpp-lower-last}
D_{\wh h}^\ep\left(S_0, S_k ; \BB S(1/2) \right) 
\preceq \ep^{-\frac{1}{d_\gamma }  + \frac{\beta}{d_\gamma } \left(2  + \frac{\gamma^2}{2} \right) - \beta - o_{\wt\zeta}(1) } \wh D_{\wh h ,\LFPP}^{\ep^\beta}\left( K , \bdy U\right) .
\alle
To lower-bound the left side of~\eqref{eqn-lfpp-lower-sum}, choose a compact set $K'$ containing $K$ in its interior and an open set $U'$ with $K' \subset U' \subset \ol U' \subset U$. Since $S_0$ and $S_k$ have side length $\delta_\ep$ and intersect $K $ and $\bdy U$, respectively, for small enough $\ep > 0$, we have $D_{\wh h}^\ep\left(S_0 , S_k ; \BB S(1/2) \right) \geq D_{\wh h}^\ep\left(K' , \bdy U' \right) $. We may therefore apply Lemma~\ref{lem-tr-compare-square} and Theorem~\ref{thm-diam} to find that with polynomially high probability as $\ep\rta 0$, the left side of~\eqref{eqn-lfpp-lower-last} is at least $\ep^{-\frac{1}{d_\gamma+\wt\zeta}}$. 
Plugging this into~\eqref{eqn-lfpp-lower-last}, re-arranging, and choosing $\wt\zeta$ sufficiently small (in a manner depending only on $\beta$, $\zeta$, and $\gamma$) yields~\eqref{eqn-lfpp-lower-show}. 
\end{proof}

\begin{proof}[Proof of Proposition~\ref{prop-lfpp-lower}]
Due to the conformal invariance of the law of the zero-boundary GFF, Proposition~\ref{prop-lfpp-lower0} implies the analogous statement with $\BB S$ replaced with any other square in $\BB C$. Lemma~\ref{lem-whole-plane-compare} then allows us to transfer this to the case of the whole-plane GFF. 
\end{proof}

\subsection{Upper bound for LFPP distances}
\label{sec-lfpp-upper}

We now conclude the proof of Theorem~\ref{thm-lfpp-compare} by proving an upper bound for LFPP distances. 

\begin{prop} \label{prop-lfpp-upper}
Let $h $ be a whole-plane GFF normalized so that its circle average over $\bdy\BB D$ is zero. 
For each open set $U\subset\BB C$, each compact set $K\subset U$, and each $\zeta \in (0,1)$, it holds with polynomially high probability as $\delta \rta 0$ that the LFPP distance with exponent $\xi=\gamma/d_\gamma$ satisfies
\eqbn
\max_{z,w\in K} D_{h ,\LFPP}^\delta \left( z,w ; U \right) \leq \delta^{1 - \frac{2}{d_\gamma} - \frac{\gamma^2}{2d_\gamma} -\zeta} .
\eqen
\end{prop}

The basic outline of the proof of Proposition~\ref{prop-lfpp-upper} is similar to that of the corresponding lower bound, but the proof is somewhat more direct since we do not prove an analogue of Proposition~\ref{prop-diam} along the way and we do not need to consider approximate LFPP distances. In this setting, we need lower bounds for Liouville graph distance instead of upper bounds. We start by using a percolation argument which is very similar to that of Lemma~\ref{lem-rectangle-perc} to get a lower bound for the $\ep$-Liouville graph distance between the inner and outer boundaries of a square annulus which holds with superpolynomially high probability (Lemma~\ref{lem-annulus-perc}). Here, the idea is to construct a path which disconnects the inner and outer boundary of the annulus consisting of squares $S$ such that the Liouville graph distance across the square annulus $S(1/2)\setminus S$ is bounded below. We will then use Lemma~\ref{lem-annulus-perc} and a union bound to show that if $\delta$ is a small positive power of $\ep$, then with high probability, we have a lower bound for the $\ep$-Liouville graph distance across the square annulus $S(1)\setminus S$ simultaneously for all squares of side length $\delta$ contained in $\BB S$ with corners in $\delta \BB Z^2$ (this is analogous to Lemma~\ref{lem-rectangle-dist}). We then consider a path between $z,w\in\BB S$ which can be covered by a minimal number of $\mu_{\wh h}$-mass $\ep$ disks and use the aforementioned lower bound for Liouville graph distance together with the upper bound in Theorem~\ref{thm-diam} to construct a path whose $\delta$-LFPP length can be bounded above.

\begin{lem} \label{lem-annulus-perc} 
For $n\in\BB N$, define the closed square annulus $\mcl A_n :=  [-n,2n]^2 \setminus (0,n)^2$ and its inner and outer boundaries $\bdy_{\op{in}}\mcl A_n := \bdy ([0,n]^2)$ and $\bdy_{\op{out}}\mcl A_n := \bdy([-n,2n]^2)$. 
For each fixed $\zeta  \in (0,1)$, there exists $a_0,a_1  , \ep_*   > 0$ (depending only on $\zeta$ and $\gamma$) such that for $n\in\BB N$ and $\ep \in (0,\ep_*]$, 
\eqb \label{eqn-annulus-perc}
\BB P\left[  D_{\wh h}^\ep\left( \bdy_{\op{in}}\mcl A_n ,  \bdy_{\op{out}}\mcl A_n     \right) \geq    e^{-n^{1/2}} \ep^{-\frac{1}{d_\gamma +\zeta}  }    \right] \geq 1 -  a_0 e^{-a_1 n} .
\eqe 
\end{lem}

As in the case of Lemma~\ref{lem-rectangle-perc}, the starting point of the proof of Lemma~\ref{lem-annulus-perc} is an estimate for a single square.

\begin{lem} \label{lem-annulus-lower}
Recall the truncated field $\wh h^\tr$ from~\eqref{eqn-wn-truncate} and its associated Liouville graph distance. 
For each $\zeta\in (0,1)$, it holds with probability tending to 1 as $\ep\rta 0$ that
\eqbn
 D_{\wh h^\tr}^\ep\left( \BB S , \bdy \BB S(1/2)   \right) \geq \ep^{-\frac{1}{d_\gamma+\zeta}} .
\eqen
\end{lem}
\begin{proof}
This follows from~\cite[Proposition 3.17 and Lemma 6.1]{dzz-heat-kernel} (which give the analogous statement for Liouville graph distances with respect to a zero-boundary GFF on a square of appropriate side length) combined with Lemma~\ref{lem-tr-compare-square}. 
\end{proof}

\begin{proof}[Proof of Lemma~\ref{lem-annulus-perc}]
The proof is similar to that of Lemma~\ref{lem-rectangle-perc}.  
We will show that there are constants $a_0,a_1 , \ep_* > 0$ as in the statement of the lemma such that for $n\in\BB N$ and $\ep \in (0,\ep_*]$, 
\eqb \label{eqn-annulus-perc-truncated}
\BB P\left[  D_{\wh h^\tr}^\ep\left( \bdy_{\op{in}}\mcl A_n ,  \bdy_{\op{out}}\mcl A_n     \right) \geq  \ep^{-\frac{1}{d_\gamma +\zeta}  }   \right] \geq 1 -  a_0 e^{-a_1 n} .
\eqe  
Combining this with~\eqref{eqn-gff-compare} (applied with $A =  c n^{1/2} $ for an appropriate constant $c>0$) and taking a union bound over $O_n(n^2)$ Euclidean balls of radius 1 whose union covers $\mcl A_n$ yields~\eqref{eqn-annulus-perc}. 
 
Let $p\in (0,1)$ be a small universal constant to be chosen later. 
For $n\in\BB N$, let $\mcl S(\mcl A_n)$ be the set of unit side length squares with corners in $\BB Z^2$ such that $S(1)\subset \mcl A_n$. For $S\in \mcl S(\mcl A_n)$, let $E_S^\ep$ be the event that the following is true. 
\begin{enumerate}
\item $D^\ep_{\wh h^\tr}\left(  S , \bdy S(1/2) \right) \geq  \ep^{- \frac{1}{d_\gamma +\zeta  } } $. 
\item Each disk which intersects $S(1/2)$ and has $\mu_{\wh h^\tr}$-mass at most $\ep$ is contained in $S(3/4)$.
\end{enumerate}
The reason for the second condition is to make it so that $E_S^\ep$ is determined by $\wh h^\tr|_{S(3/4)}$. 
For each $S\in \mcl S(\mcl A_n)$, the field $  \wh h^\tr(\cdot - v_S + v_{\BB s}) $ agrees in law with $\wh h^\tr  $. 
It therefore follows from Lemma~\ref{lem-annulus-lower} and the fact that $\mu_{\wh h^\tr}$ assigns positive mass to every open set that we can find $\ep_* = \ep_*(p,\zeta,\gamma) > 0$ such that  
\eqb \label{eqn-perc-prob'}
\BB P[E_S^\ep ] = \BB P[E_{\BB S}^\ep] \geq 1-p \quad \forall S\in \mcl S(\mcl A_n) ,\quad \forall \ep\in (0,\ep_*]    .
\eqe 

View $\mcl S(\mcl A_n)$ as a graph with two squares considered to be adjacent if they share an edge.  
We claim that if $p$ is chosen sufficiently small, in a manner depending only on $\zeta$ and $\gamma$, then for appropriate constants $a_0,a_1 > 0$ as in the statement of the lemma, it holds for each $\ep \in (0,\ep_*]$ and $n\in\BB N$ that with probability at least $1- a_0 e^{-a_1 n}$, we can find a path $\mcl P$ in $\mcl S(\mcl A_n)$ which disconnects $\bdy_{\op{in}}\mcl A_n$ from $\bdy_{\op{out}} \mcl A_n$ such that $E_S^\ep$ occurs for each $S\in\mcl P$. 

Assume the claim for the moment. If a path $\mcl P$ as in the claim exists, then each Euclidean path from $\bdy_{\op{in}}\mcl A_n$ to $\bdy_{\op{out}} \mcl A_n$ must pass through one of the squares $S\in\mcl P$. Since $S(1/2)\subset\mcl A_n$ for each $S\in \mcl S(\mcl A_n)$, any path from $\bdy_{\op{in}}\mcl A_n$ to $\bdy_{\op{out}} \mcl A_n$ must cross one of the annuli $S(1/2) \setminus S $ for some $S\in \mcl P$. Since $E_S^\ep$ occurs for each such $S$, it follows that~\eqref{eqn-annulus-perc-truncated} holds. 
   
It remains only to prove the claim. Since $E_S^\ep$ is determined by $\wh h^\tr|_{S(1)}$,  the claim follows from exactly the same percolation-type argument given at the end of the proof of Lemma~\ref{lem-rectangle-perc}.  
\end{proof}

The following is an analogue of Lemma~\ref{lem-rectangle-dist} in the present setting, and is proven in a similar way.

\begin{lem} \label{lem-square-dist} 
For each $\beta \in \left(0,\frac{2}{(2+\gamma)^2}\right)$ and $\zeta \in (0,1)$, there exists $\ep_* = \ep_*(\zeta,\beta,\gamma) > 0$ such that the following is true. 
Let $\delta_\ep := 2^{- \lceil \log_2  \ep^{-\beta} \rceil}$. 
It holds with polynomially high probability as $\ep\rta 0$ that for each square $S\subset \BB S(1)$ with side length $\delta_\ep  $ and corners in $\delta_\ep \BB Z^2$, 
\eqb \label{eqn-square-dist}
 D_{\wh h}^\ep\left(  S , \bdy S(1) \right)  
\geq \ep^{-\frac{1}{d_\gamma } + \frac{\beta}{d_\gamma  } \left(2  + \frac{\gamma^2}{2}  \right) + \zeta    }  \exp\left(   \frac{\gamma}{d_\gamma} \max_{z\in S(1)} \wh h_{\ep^\beta }(z)   \right) .
\eqe 
\end{lem}
\begin{proof}
Fix $\wt\zeta\in (0,1)$ to be chosen later, in a manner depending only on $\zeta$.
Also set 
\eqbn
n_\ep :=  \lfloor  (\log \ep^{-1})^{3/2} \rfloor .
\eqen
By Lemmas~\ref{lem-one-scale-max} and~\ref{lem-mid-scale-max} (the latter applied with $A = (\delta_\ep/\ep^\beta) n_\ep$), it holds with polynomially high probability as $\ep\rta 0$ that 
\eqb \label{eqn-field-control'}
\max_{z\in\BB S } |\wh h_{\ep^\beta}(z) | \leq (2\beta+\wt\zeta) \log\ep^{-1} 
 \quad \text{and} \quad 
 \max_{z,w\in \BB S : |z-w| \leq 4 \ep^\beta} |\wh h_{\delta_\ep/ n_\ep}(z) - \wh h_{\ep^\beta}(w)| \leq \wt\zeta \log \ep^{-1} . 
\eqe

If $S$ is a square as in the statement of the lemma and $u_S$ denotes its bottom-left corner, then the re-scaled translated square annulus $n_\ep  \delta_\ep^{-1} \left( \ol{S(1) \setminus S} - u_S \right)$ is equal to the annulus $\mcl A_{n_\ep}$ of Lemma~\ref{lem-annulus-perc}. 
Moreover, the field $ (\wh h - \wh h_{\delta_\ep/n_\ep})  ((\delta_\ep /n_\ep) \cdot + u_S ) $ agrees in law with $\wh h$ and is independent from $\wh h_{\delta_\ep/n_\ep }$, so the associated Liouville graph distance $D_{(\wh h - \wh h_{\delta_\ep/n_\ep})  ((\delta_\ep /n_\ep) \cdot + u_S ) }^\ep$ is independent from $\wh h_{\delta_\ep/n_\ep}$ and agrees in law with $D_{\wh h}^\ep$.
By~\eqref{eqn-dist-scaling} applied with $\delta = \delta_\ep / n_\ep$ and $U $ equal to the interior of $S$, it therefore follows that the conditional law of $D_{\wh h}^\ep\left( \bdy S  ,\bdy S(1) \right)$ given $\wh h_{\delta_\ep/n_\ep }$ stochastically dominates the law of
\eqbn
 D_{\wh h }^{T_S\ep}\left(\bdy_{\op{in}}\mcl A_{n_\ep} , \bdy_{\op{out}} \mcl A_{n_\ep} \right) 
\quad \text{for} \quad 
T_S := (n_\ep / \delta_\ep )^{2+\frac{\gamma^2}{2}} \exp\left( -  \gamma \min_{z\in S}  \wh h_{\delta_\ep/n_\ep }(z)  \right) .
\eqen

If~\eqref{eqn-field-control'} holds, then 
\allb \label{eqn-square-dist-T}
T_S  
&\leq    \ep^{ -  \beta \left( 2 + \frac{\gamma^2}{2} \right)  + o_{\wt\zeta}(1) + o_\ep(1)}   \exp\left( -  \gamma \max_{z\in S(1)} \wh h_{\ep^\beta }(z) \right) \notag \\
&\leq    \ep^{ -  \beta \left( 2 + \frac{\gamma^2}{2} \right)  + o_{\wt\zeta}(1) + o_\ep(1)}   \exp\left( - \frac{d_\gamma-\wt\zeta}{d_\gamma} \gamma \max_{z\in S(1)} \wh h_{\ep^\beta }(z) \right)  ,
\alle 
with the rate of the $o_{\wt\zeta}(1)$ and the $o_\ep(1)$ deterministic, and the rate of the former independent of $\ep$. Note that here we have used~\eqref{eqn-field-control'} to replace $\wh h_{\delta_\ep/n_\ep}$ by $\wh h_{\ep^\beta}$ and to replace a min by a max. By~\eqref{eqn-square-dist-T} and the first inequality in~\eqref{eqn-field-control'} and since $\beta < 2/(2+\gamma)^2$ (which implies $2+\gamma^2/2 - 2\gamma > 0$) if $\wt\zeta$ is chosen sufficiently small (in a manner depending only on $\zeta$, $\beta$, and $\gamma$) then $\max_S T_S \ep  = o_\ep(1)$ at a deterministic rate. In particular, if $\ep_*$ is as in Lemma~\ref{lem-annulus-perc} then for a small enough deterministic $\ep > 0$ we have $T_S \ep \leq \ep_*$ for all squares $S$ as above whenever~\eqref{eqn-field-control'} holds. By Lemma~\ref{lem-annulus-perc} (applied with $T_S \ep$ in place of $\ep$ and $n_\ep$ in place of $n$) and a union bound over $O_\ep(\ep^{-2\beta})$ squares $S$, we obtain the statement of the lemma provided we choose $\wt\zeta$ sufficiently small. 
\end{proof}

We will now prove the zero-boundary GFF analogue of Proposition~\ref{prop-lfpp-upper}. 

\begin{prop} \label{prop-lfpp-upper0}
Let $h^{\BB S(1)}$ be a zero-boundary GFF on $\BB S(1)$ and for $\delta>0$ let $D_{h^{\BB S(1)},\LFPP}^\delta$ be the associated LFPP metric with $\xi =\gamma/d_\gamma$. 
For each $\zeta \in (0,1)$, it holds with polynomially high probability as $\delta \rta 0$ that 
\eqbn
\max_{z,w\in \BB S} D_{h^{\BB S(1)},\LFPP}^\delta \left( z,w ; \BB S(1/2)    \right) \leq \delta^{1 - \frac{2}{d_\gamma} - \frac{\gamma^2}{2d_\gamma} - \zeta} .
\eqen
\end{prop}
\begin{proof}
Let $\beta \in \left(0,\frac{2}{(2+\gamma)^2} \right)$ be arbitrary (e.g., $\beta = \frac{1}{(2+\gamma)^2}$ would suffice). 
We will compare $\ep^\beta$-LFPP distances to $\ep$-Liouville graph distances with respect to $\wh h$, then set $\delta=\ep^\beta$. By Lemma~\ref{lem-circle-avg-approx}, it suffices to prove an upper bound for $\ep^\beta$-Liouville graph distances defined with $\wh h_{\ep^\beta}$ in place of $h_{\ep^\beta}$, i.e., it is enough to show that with polynomially high probability as $\ep\rta 0$, there exists for each $z,w\in\BB S$ a Euclidean unit-speed path $P_{\LFPP} : [0,T]\rta\BB S(1/2)$ such that 
\eqb \label{eqn-lfpp-upper-show}
\int_0^T e^{\frac{\gamma}{d_\gamma} \wh h_{\ep^\beta}(P_{\LFPP}(t)) } \, dt \leq \ep^{\beta\left( 1 - \frac{2}{d_\gamma} - \frac{\gamma^2}{2d_\gamma} \right) -  \zeta} .
\eqe

Let $ \wt\zeta \in (0,1)$ which we will choose later, in a manner depending only on $\beta$, $\zeta$, and $\gamma$. 
Also set $\delta_\ep := 2^{- \lfloor  \log_2 \ep^{-\beta } \rfloor}$,
as in Lemma~\ref{lem-square-dist}. 

Let $E^\ep = E^\ep(\beta,\wt\zeta)$ be the event of Lemma~\ref{lem-square-dist} but with $\wt\zeta$ in place of $\zeta$, i.e., $E^\ep$ is the event that~\eqref{eqn-square-dist} holds (with $\wt\zeta$ in place of $\zeta$) for each square $S\subset \BB S(1)$ with side length $\delta_\ep  $ and corners in $\delta_\ep \BB Z^2$. By Lemmas~\ref{lem-square-dist} and~\ref{lem-max-ball-radius}, it holds with polynomially high probability as $\ep\rta 0$ that
\eqb \label{eqn-lfpp-lower-event}
E^\ep \: \text{occurs} \quad \text{and} \quad \min_{z\in\BB S(1)} \mu_{\wh h}(B_{\delta_\ep/2}(z) ) \geq \ep .
\eqe
We note that the second condition in~\eqref{eqn-lfpp-lower-event} implies that any two points of $\BB S(1 )$ which lie at Euclidean distance at least $\delta_\ep$ from each other lie at $D_{\wh h}^\ep $-distance at least 2. 

Henceforth assume that~\eqref{eqn-lfpp-lower-event} holds and fix $z,w\in\BB S$. 
By the definition of $  D_{\wh h}^\ep$, we can find a continuous Euclidean path $P : [0,1] \rta \BB S(1/2) $ from $z$ to $w$ whose range can be covered by at most $D_{\wh h}^\ep(z,w;\BB S(1/2) )$ Euclidean balls of $\mu_{\wh h}$-mass at most $\ep$ which are contained in $\BB S(1/2)$.
We will use this path $P$ to construct a path from $z$ to $w$ whose $\ep^\beta$-LFPP length can be bounded above.

Let $\mcl S_{\delta_\ep}(1/2)$ be the set of squares $S\subset \BB S(1/2)$ with side length $\delta_\ep$ and corners in $\delta_\ep\BB Z^2$. We will inductively define a sequence of squares in $\mcl S_{\delta_\ep}(1/2)$. Let $t_0 = 0$ and let $S_0$ be a square of $\mcl S_{\delta_\ep}(1/2)$ which contains $z$. Inductively, if $j\in\BB N$ and $t_{j-1} \in [0, 1]$ and $S_{j-1} \in \mcl S_{\delta_\ep}(1/2)$ have been defined, let $t_j$ be the infimum of the times $t \in [ t_{j-1} , 1]$ for which $P(t)$ is not contained in the expanded square $S_{j-1}(1)$, or let $t_j = 1$ if no such $t$ exists. If $t_j = 1 $, let $S_j $ be a square of $\mcl S_{\delta_\ep}(1/2)$ containing $w$. Otherwise, choose $S_j \in \mcl S_{\delta_\ep}(1/2)$ so that $P$ enters $S_j$ immediately after time $t$.  

Let $\mcl J$ be the smallest $j\in\BB N$ for which $t_{j+1} = 1$ ($\mcl J$ must be finite since $P$ is continuous). Let  $P_{\LFPP} : [0,T] \rta \BB S$ be the concatenation of the straight line segments $[P(t_j) , P(t_{j+1})]$ for $j\in [0 ,\mcl J ]_{\BB Z}$, traversed at unit speed.
Then $P_{\LFPP}$ is a path from $z$ to $w$. Since each $t_{j+1}$ for $j\leq \mcl J-1$ is the first time after $t_j$ at which $P$ exits $S_j(1)$, each of the segments $[P(t_j) , P(t_{j+1})]$ is contained in $S_j(1)$, and in particular its Euclidean length is at most a universal constant times $\ep^\beta$. Hence
\eqb \label{eqn-lfpp-upper-path}
\int_0^T e^{\frac{\gamma}{d_\gamma} \wh h_{\ep^\beta}(P_{\LFPP}(t)) } \, dt
\preceq \sum_{j=0}^{\mcl J} \ep^\beta \exp\left(\frac{\gamma}{d_\gamma} \max_{z\in S_j(1)} \wh h_{\ep^\beta}(z) \right) . 
\eqe 

For each $j\in [0,\mcl J-1]_{\BB Z}$ the path $P$ travels across the square annulus $S_j(1) \setminus S_j$ during the time interval $[t_j,t_{j+1}]$. By~\eqref{eqn-square-dist}, it follows that 
\eqb \label{eqn-lfpp-upper-inc}
D_{\wh h}^\ep\left( P(t_j) , P(t_{j+1})   ; \BB S(1/2) \right)   \geq  \ep^{-\frac{1}{d_\gamma } + \frac{\beta}{d_\gamma  } \left(2  + \frac{\gamma^2}{2}  \right) + \wt\zeta    }  \exp\left(   \frac{\gamma}{d_\gamma} \max_{z\in S_j(1)} \wh h_{\ep^\beta }(z)   \right) .
\eqe 
By our choice of $P$, the sum of the left side of this inequality over all $j \in [0 , \mcl J-1]_{\BB Z}$ is at most $   D_{\wh h}^\ep(z,w;\BB S) + 2\mcl J$ (the term $2\mcl J$ comes from double-counting the disks which contain the points $P(t_j)$). Since $|P(t_j) - P(t_{j+1})| \geq \delta_\ep$, the second condition in~\eqref{eqn-lfpp-lower-event} shows that none of the pairs of points $(P(t_j) , P(t_{j+1}))$ can be contained in a single Euclidean ball of $\mu_{\wh h}$-mass $\ep$.
Therefore, $\mcl J \leq  D_{\wh h}^\ep\left(z,w;\BB S(1/2) \right)$. 
Consequently, summing~\eqref{eqn-lfpp-upper-inc} over all $j$ gives
\eqb \label{eqn-lfpp-upper-sum}
 \sum_{j=0}^{\mcl J}  \exp\left(\frac{\gamma}{d_\gamma} \max_{z\in S_j(1)} \wh h_{\ep^\beta}(z) \right)
 \leq  3 \ep^{ \frac{1}{d_\gamma } - \frac{\beta}{d_\gamma  } \left(2  + \frac{\gamma^2}{2}  \right) - \wt\zeta    }   D_{\wh h}^\ep\left(z,w;\BB S(1/2) \right) .
\eqe 
By Proposition~\ref{prop-diam}, we have $\max_{z,w\in\BB S}  D_{\wh h}^\ep\left(z,w;\BB S(1/2) \right) \leq \ep^{-\frac{1}{d_\gamma}- \wt\zeta}$ with polynomially high probability as $\ep\rta 0$. Combining this with~\eqref{eqn-lfpp-upper-path} and~\eqref{eqn-lfpp-upper-sum} and choosing $\wt\zeta$ sufficiently small, in a manner depending only on $\zeta$, $\beta$, and $\gamma$, shows that~\eqref{eqn-lfpp-upper-show} holds with polynomially high probability as $\ep\rta0$. 
\end{proof}

\begin{proof}[Proof of Proposition~\ref{prop-lfpp-upper}]
Proposition~\ref{prop-lfpp-upper0} implies the analogous statement with $\BB S$ replaced with any other square $S\subset \BB C$ (with the rate of convergence of the probability depending on the square). Lemma~\ref{lem-whole-plane-compare} then implies that the same is true with a whole-plane GFF in place of a zero-boundary GFF on $S(1)$. We obtain the proposition statement from this by covering $K$ by a finite union of squares $S$ such that $S(1/2)$ is contained in $U$. 
\end{proof}

\begin{proof}[Proof of Theorem~\ref{thm-lfpp-compare}]
The lower bound for the point-to-point distance in~\eqref{eqn-lfpp-compare} follows from Proposition~\ref{prop-lfpp-lower} applied with $K=\{z\}$ and $U$ an open set containing $z$ but not $w$. The upper bound follows from Proposition~\ref{prop-lfpp-upper} applied with $K$ chosen so that $z,w\in K$. The bounds for other LFPP distances in~\eqref{eqn-lfpp-diam} are immediate from Propositions~\ref{prop-lfpp-lower} and~\ref{prop-lfpp-upper} together with~\eqref{eqn-lfpp-compare}. 
\end{proof}

\section{Connection to random planar maps}
\label{sec-planar-map}

In this section we will relate Liouville graph distance to random planar maps and thereby prove Theorem~\ref{thm-ball-size}. 
The basic strategy for doing so is discussed in Section~\ref{sec-planar-map-discussion}. 
To carry out this approach, we will first need to provide some background on SLE and LQG which will allow us to express the mated-CRT map as the adjacency graph of a certain random collection of $\ep$-LQG mass ``cells" in the plane (Section~\ref{sec-lqg-prelim}). 
In Section~\ref{sec-sle-compare}, we compare distances in this adjacency graph of cells to Liouville graph distance. Actually, we will prove upper and lower bounds for distances in the adjacency graph in terms of two minor variants of Liouville graph distance defined using slightly different types of Euclidean balls. 
In Section~\ref{sec-ball-size} we complete the proof of Theorem~\ref{thm-ball-size} in the setting of the mated-CRT map, then transfer to other random planar maps using the results of~\cite{ghs-map-dist}.

\subsection{Background on SLE and LQG}
\label{sec-lqg-prelim}

\subsubsection{Liouville quantum gravity surfaces}

Fix $\gamma \in (0,2)$. Following~\cite{shef-kpz,shef-zipper,wedges}, we define a \emph{$\gamma$-Liouville quantum gravity surface} to be an equivalence class of pairs $(\mcl D,h)$ where $\mcl D\subset\BB C$ and $h$ is a distribution on $\mcl D$ (which we will always take to be a realization of some variant of the Gaussian free field), with two such pairs $(\mcl D,h)$ and $(\wt{\mcl D} , \wt h)$ considered to be equivalent if there is a conformal map $f : \wt{\mcl D} \rta \mcl D$ such that
\eqb \label{eqn-lqg-coord}
\wt h = h \circ f + Q\log |f'| \quad \text{for} \quad Q = \frac{2}{\gamma}  +\frac{\gamma}{2} .
\eqe
One motivation for this definition is that by~\cite[Proposition 2.1]{shef-kpz}, if $\wt h$ and $h$ are related as in~\eqref{eqn-lqg-coord} then a.s.\ the $\gamma$-LQG measure $\mu_h$ is the pushforward of $\mu_{\wt h}$ under $f$, so $\mu_h$ is a well-defined functional of the LQG surface (as noted in Section~\ref{sec-lfpp-exponent}, we expect the conjectural $\gamma$-LQG metric to satisfy a similar invariance property). If $(\mcl D , h)$ is an equivalence class representative, then the distribution $h$ is called an \emph{embedding} of the $\gamma$-LQG surface into $\mcl D$. 

In this paper, the only types of $\gamma$-LQG surfaces which we will be interested in are the ones corresponding to whole-plane and zero-boundary GFF's and the so-called \emph{$\gamma$-quantum cone} which is defined in~\cite[Definition 4.10]{wedges}. We will not need the precise definition. Roughly speaking, the $\gamma$-quantum cone describes the local behavior of a general $\gamma$-LQG surface at a point sampled from its $\gamma$-LQG measure~\cite[Proposition 4.13(ii) and Lemma~A.10]{wedges}. 

We will only ever work with a particular embedding of the $\gamma$-quantum cone called the \emph{circle-average embedding}, which is the random distribution $h$ on $\BB C$ defined in~\cite[Definition 4.10]{wedges}. Aside from its connection to the mated-CRT map (as we discuss below), the most important property of this distribution for our purposes is that $h|_{\BB D}$ agrees in law with the corresponding restriction of a whole-plane GFF plus $-\gamma\log|\cdot|$, normalized so that its circle average over $\bdy\BB D$ is zero. The $\gamma$-log singularity arises from the fact that a $\gamma$-LQG surface has such a log singularity at a typical point from the perspective of the $\gamma$-LQG measure~\cite[Section 3.3]{shef-kpz}. 

\subsubsection{Space-filling SLE$_\kappa$}

In this subsection we will review the construction of whole-plane space-filling SLE$_\kappa$ from $\infty$ to $\infty$ for $\kappa > 4$, which first appeared in~\cite[Section 1.2.3]{ig4}. This is a random space-filling curve in $\BB C$ which a.s.\ hits each fixed point exactly once (but has an uncountable fractal set of multiple points). In the case when $\kappa \geq 8$, ordinary SLE$_\kappa$ is already space-filling~\cite{schramm-sle} and whole-plane space-filling SLE$_\kappa$ from $\infty$ to $\infty$ is just a two-sided variant of ordinary SLE$_\kappa$. For $\kappa \in (4,8)$, however, ordinary SLE$_\kappa$ is not space-filling and space-filling SLE$_\kappa$ can be obtained from ordinary SLE$_\kappa$ by, roughly speaking, iteratively filling in the ``bubbles" which it disconnects from its target point by SLE$_\kappa$-type curves.

We will not need many properties of space-filling SLE$_\kappa$ here (only some estimates from~\cite{ghm-kpz} which the reader can take as a black box), but we provide a moderately detailed review for the sake of context. 
The basic idea of the construction is to first construct the outer boundary of the curve $\eta$ stopped at the first time it hits each $z\in\BB Q^2$, then interpolate these boundaries to get a space-filling curve. By SLE duality~\cite{zhan-duality1,zhan-duality2,dubedat-duality,ig1,ig4} the outer boundaries should be SLE$_{\ul\kappa}$-type curves for $\ul\kappa=16/\kappa$. We will define these curves using the theory of imaginary geometry~\cite{ig1,ig4}. 

Let $\chi^{\op{IG}} := 2/\sqrt{\ul\kappa} -\sqrt{\ul\kappa}/2$ and let $h^{\op{IG}}$ be a whole-plane GFF viewed modulo a global additive multiple of $2\pi\chi^{\op{IG}}$, as in~\cite{ig4} (here IG stands for ``Imaginary Geometry" and is used to distinguish the field $h^{\op{IG}}$ from the field $h$ used to construct the LQG measure).  
By~\cite[Theorem 1.1]{ig4}, for $z\in \BB C$, we can construct the \emph{flow lines} $\eta_z^L$ and $\eta_z^R$ of $h$ started from $z$ with angles $\pi/2$ and $-\pi/2$, respectively. These curves will be the left and right boundaries of $\eta$ stopped upon hitting $z$. 

For distinct $z,w \in \BB Q^2$, the flow lines $\eta_z^L$ and $\eta_w^L$ a.s.\ merge upon intersecting, and similarly with $R$ in place of $L$. The two flow lines $\eta_z^L$ and $\eta_z^R$ started at the same point a.s.\ do not cross, but these flow lines bounce off each other without crossing if and only if $\kappa \in (4,8)$~\cite[Theorem~1.7]{ig4}. 

We define a total order on $\BB Q^2$ by declaring that $z$ comes before $w$ if and only if $w$ is in a connected component of $\BB C\setminus (\eta_z^L\cup \eta_z^R)$ which lies to the right of $\eta_z^L$ (equivalently, to the left of $\eta_z^R$).  The whole-plane analogue of~\cite[Theorem~4.12]{ig4} (which can be deduced from the chordal case; see~\cite[Section~1.4.1]{wedges}) shows that there is a.s.\ a well-defined continuous curve $\eta : \BB R\rta \BB C$ such that the following is true. The curve $\eta$ traces the points of $\BB Q^2$ in the above order, is such that $\eta^{-1}(\BB Q^2)$ is a dense set of times, and is continuous when parameterized by Lebesgue measure. This curve $\eta$ is defined to be the whole-plane space-filling SLE$_{\kappa}$ from $\infty$ to $\infty$. 

In the case when $\kappa \geq 8$, the left/right boundary curves $\eta_z^L$ and $\eta_z^R$ do not bounce off each other, so for $a < b$ the set $\eta([a,b])$ has the topology of a closed disk. In contrast, for $\kappa \in (4,8)$ the curves $\eta_z^L$ and $\eta_z^R$ intersect in an uncountable fractal set and for $a<b$ the interior of the set $\eta([a,b])$ a.s.\ has countably many connected components, each of which has the topology of a disk (see Figure~\ref{fig-LR-def}, right). 
 
\begin{figure}[t!]
 \begin{center}
\includegraphics[scale=1]{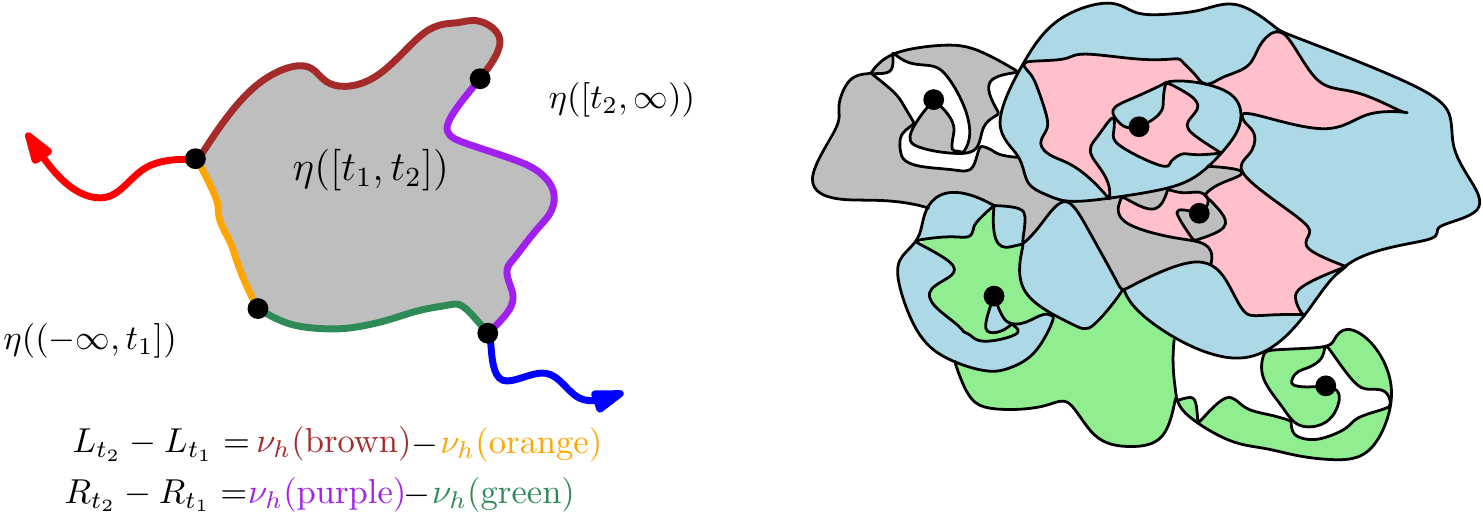}  
\vspace{-0.01\textheight}
\caption{ \textbf{Left.} Definition of the left/right quantum boundary length process $(L,R)$ for the space-filling SLE$_\kappa$ curve $\eta$, which is shown to be a pair of correlated Brownian motions in~\cite[Theorem 1.9]{wedges}. This figure corresponds to the case $\kappa \geq 8$ ($\gamma \in (0,\sqrt 2]$) since the image of each interval under $\eta$ is simply connected.
\textbf{Right.} Illustration of four typical space-filling SLE cells of the form $\eta([x-\ep,x])$ in the case $\kappa \in (4,8)$ ($\gamma \in (\sqrt 2 , 2)$). The picture is slightly misleading since the set of ``pinch points" where the left and right boundaries of each cell meet is actually uncountable, with no isolated points, but has Hausdorff dimension less than 2. The points where $\eta$ starts and finishes filling in each cell are shown with black dots. The grey and green cells intersect at several points, but do not share a connected boundary arc so are \emph{not} considered to be adjacent. This is natural since one can think of the blue cell as lying in between the grey and green cells. In fact, two cells which intersect, but do not share a connected boundary arc, will always be separated by another cell in this manner.  
}\label{fig-LR-def}
\end{center}
\vspace{-1em}
\end{figure} 

\subsubsection{Mated-CRT maps and SLE-decorated LQG}
\label{sec-peanosphere}

As in Section~\ref{sec-planar-map-discussion}, let $\gamma \in (0,2)$, let $h$ be the circle-average embedding of a $\gamma$-quantum cone, and let $\eta$ be a whole-plane space-filling SLE$_\kappa$ curve with $\kappa = 16/\gamma^2 > 4$ sampled independently from $h$ and then parametrized by $\gamma$-LQG mass with respect to $h$, i.e., so that $\eta(0) = 0$ and $\mu_h(\eta([s,t])) = t-s$ whenever $-\infty < s < t < \infty$. Let $\nu_h$ be the $\gamma$-LQG boundary length measure associated with $h$ (as in~\cite[Section 6]{shef-kpz}). We let $L : \BB R \rta \BB R$ be the process such that $L_0 = 0$ and for $t_1 < t_2$, 
\allb \label{eqn-LR-def}
L_{t_2} - L_{t_1} &=\nu_h\left( \text{left boundary of $ \eta([t_1,t_2]) \cap \eta([t_2,\infty))$} \right) \notag \\
&\qquad - \nu_h\left( \text{left boundary of $ \eta([t_1,t_2]) \cap \eta((-\infty,t_1])$} \right)  .
\alle
Similarly define $R $ with ``right" in place of ``left". See Figure~\ref{fig-LR-def}, left, for an illustration. Then~\cite[Theorem 1.9]{wedges} (and~\cite[Theorem 1.1]{kappa8-cov} in the case $\gamma  < \sqrt 2$) shows that $(L,R)$ has the law of a correlated two-sided two-dimensional Brownian motion with $\op{Corr}(L,R) = -\cos(\pi \gamma^2/4)$. In other words, $(L,R)$ is the same as the process used to construct the mated-CRT map in Section~\ref{sec-planar-map-discussion}. 

If we let $\mcl G^\ep$ for $\ep  >0$ be the mated-CRT map constructed from $(L,R)$, then one sees from~\eqref{eqn-LR-def} that two cells $\eta([x_1 - \ep  , x_1])$ and $\eta([x_2-\ep ,x_2])$ for $x_1,x_2 \in \ep \BB Z $ intersect along a non-trivial connected boundary arc if and only if the mated-CRT map adjacency condition~\eqref{eqn-inf-adjacency} holds for either $L$ or $R$.  Thus the mated-CRT map is isomorphic to the adjacency graph of space-filling SLE cells, with cells considered to be adjacent if they intersect along a non-trivial connected boundary arc. Note that for $\kappa \in (4,8)$, it is possible for two cells to intersect along a cantor-like set, but not a non-trivial connected boundary arc, in which case the cells are not considered to be adjacent (see Figure~\ref{fig-LR-def}, right).

\subsection{Comparing Liouville graph distance and SLE cell distance}
\label{sec-sle-compare}

In this subsection we will prove a proposition which allows us to compare distances in the adjacency graph of space-filling SLE cells discussed above with Liouville graph distances (see Proposition~\ref{prop-sle-metric-compare}). For this purpose we first need to introduce a few variants of Liouville graph distance. We start with the analogue of Liouville graph distance with SLE cells used in place of Euclidean balls.

\begin{defn} \label{def-sle-metric}
Let $h$ be some variant of the GFF on the whole plane such that $\mu_h(\BB C) =\infty$ a.s.\ and let $\eta$ be an independent whole-plane space-filling SLE$_\kappa$ curve, sampled independently from $h$ and then parametrized by $\gamma$-LQG mass with respect to $h$, i.e., so that $\eta(0) = 0$ and $\mu_h(\eta([s,t])) = t-s$ whenever $-\infty < s < t < \infty$. For $U\subset\BB C$, $\ep>0$, and $z_1,z_2\in \BB C$, we let $D_{h,\eta}^\ep(z_1,z_2 ; U)$ be equal to 1 plus the graph distance from the cell containing $z_1$ to the cell containing $z_2$ in the graph of cells of the form $\eta([x-\ep,x])$ for $x\in\ep\BB Z$ which are contained in $\ol U$, with two such cells considered to be adjacent if they intersect along a non-trivial connected boundary arc.

We also let $\wt D_{h,\eta}^\ep(z_1,z_2; U)$ be the minimum number of SLE segments of the form $\eta([a,b])$ for $0 < b-a \leq \ep$ which are contained in $\ol U$ and whose union contains a Euclidean path from $z_1$ to $z_2$. 

We abbreviate $D_{h,\eta}^\ep(\cdot,\cdot) := D_{h,\eta}^\ep(\cdot,\cdot ; \BB C)$ and $\wt D_{h,\eta}^\ep(\cdot,\cdot  ) :=  D_{h,\eta}^\ep(\cdot,\cdot  ; \BB C ) $. 
\end{defn}

We are mostly interested in $D_{h,\eta}^\ep$ since in the case when $h$ is the field corresponding to a $\gamma$-quantum cone, we know from the preceding subsection that $D_{h,\eta}^\ep(z_1,z_2 )$ is equal to 1 plus the graph distance in the mated-CRT map $\mcl G^\ep$ between the vertices corresponding to the cells containing $z_1$ and $z_2$. 
On the other hand, the definition of $\wt D_{h,\eta}^\ep$ is more closely analogous to the definition of Liouville graph distance and in particular 
\eqb \label{eqn-sle-dist-mono}
\wt D_{h,\eta}^\ep(z_1,z_2 ; U) \leq \wt D_{h,\eta}^{\ep'}(z_1,z_2 ; U) ,\quad \forall z_1,z_2\in U ,\quad \forall \ep' \in (0,\ep] .
\eqe
The analogous relationship does not hold for $D_{h,\eta}^\ep$. 

It is obvious that 
\eqb 
\wt D_{h,\eta}^\ep(z_1,z_2 ; U) \leq D_{h,\eta}^\ep(z_1,z_2  ; U) ,\quad \forall z_1,z_2\in U .
\eqe
One also has the following reverse relationship. 

\begin{lem} \label{lem-sle-metric}
Suppose we are in the setting of Definition~\ref{def-sle-metric}. There is a deterministic constant $c > 0$, depending only on $\gamma$, such that the following is true. Let $U \subset V\subset \BB C$ be connected open sets. On the event that each cell $\eta([x-\ep,x])$ for $x\in\ep\BB Z$ which intersects $\ol U$ is contained in $\ol V$,  
\eqb \label{eqn-sle-metric} 
D_{h,\eta}^\ep(z_1,z_2 ; V) \leq c \wt D_{h,\eta}^\ep(z_1,z_2; U) ,\quad \forall z_1,z_2 \in U .
\eqe 
In particular, if $h$ is a whole-plane GFF normalized so that $h_1(0) = 0$ or the circle average embedding of a $\gamma$-quantum cone and $\ol U \subset V \subset \ol V\subset \BB D$, then~\eqref{eqn-sle-metric} holds with polynomially high probability as $\ep\rta 0$. 
\end{lem}
\begin{proof}
Suppose we are working on the event that each cell $\eta([x-\ep,x])$ for $x\in\ep\BB Z$ which intersects $\ol U$ is contained in $\ol V$. We will prove~\eqref{eqn-sle-metric}. Each segment $\eta([a,b])$ with $0 < b- a \leq \ep$ which is contained in $U$ is contained in the union of at most two segments of the form $\eta([x-\ep,x])$ for $x\in\ep\BB Z$ which intersect $\ol U$, hence are contained in $\ol V$. 
It follows that the minimum number of cells of the form $\eta([x-\ep,x])$ for $x\in\ep\BB Z$ which are contained in $\ol V$ and whose union contains a Euclidean path in $\ol U$ from $z_1$ to $z_2$ is bounded above by $c \wt D_{h,\eta}^\ep(z_1,z_2 ; U)$. 

The above minimum is not the same as $D_{h,\eta}^\ep(z_1,z_2 ;  V )$ when $\gamma \in (\sqrt 2 , 2)$ since in this case two cells can intersect but not share a non-trivial connected boundary arc (see Figure~\ref{fig-LR-def}, right), in which case they do not count as being adjacent for the purposes of defining $D_{h,\eta}^\ep(z_1,z_2; V)$. However, the number of times that $\eta$ can hit any fixed point of $\BB C$ is at most a deterministic, $\gamma$-dependent constant $c'$ (see~\cite[Section 6]{ghm-kpz} or~\cite[Section 8.2]{wedges}), so any two cells which intersect at a point $w \in U$ can be joined by a path of at most $c'$ cells which also intersect $w$ and such that any two successive cells in the path share a non-trivial boundary arc. Each cell in this path must be contained in $\ol V$ by assumption, so we get~\eqref{eqn-sle-metric} with $c = 2c'$. 

The last statement follows since standard SLE/LQG estimates show that there is a $q = q(\gamma) > 0$ such that the maximal diameter of the cells $\eta([x-\ep,x])$ which intersect $\ol U$ is at most $\ep^q$ with polynomially high probability as $\ep\rta 0$ (see, e.g.,~\cite[Lemma 3.3]{ghs-dist-exponent}), which implies that each such cell is contained in $\ol V$ with polynomially high probability as $\ep\rta 0$. 
\end{proof}

In the case of the whole-plane GFF, we will bound $D_{h,\eta}^\ep$-distances above and below by distances with respect to two minor variants of Liouville graph distances which we now define.  
For $z\in\BB C$ and $\ep > 0$, let 
\eqb \label{eqn-lgd-upper}
\ol R^\ep(z):= \sup\left\{ r > 0 : \mu_h(B_{2r}(z)) \leq \ep \right\} \quad \op{and} \quad \ol B^\ep(z):= B_{\ol R^\ep(z)}(z). 
\eqe 
Also let 
\eqb \label{eqn-lgd-lower}
\ul R^\ep(z) := \sup\left\{r > 0 : e^{\gamma h_r(z)} r^{2+\frac{\gamma^2}{2}} \leq \ep \right\} \quad \op{and} \quad
\ul B^\ep(z) := B_{\ul R^\ep(z)}(z) .
\eqe  
For an open set $U\subset\BB C$ and $z_1,z_2\in U$, we define the modified Liouville graph distance $\ol D_h^\ep\left(z_1,z_2 ; U \right)$ to be the minimum number of balls of the form $B_r(w)$ with $r\leq \ol R^\ep(w)$, $w \in U$, and $B_r(w) \subset \ol U$ whose union contains a path from $z_1$ to $z_2$. We similarly define $\ul D_h^\ep(z_1,z_2 ; U)$ but with $\ul R^\ep(w)$ in place of $\ol R^\ep(w)$. 

We will need the following partial analogue of Theorem~\ref{thm-diam} for the above variants of Liouville graph distance. 

\begin{prop} \label{prop-lgd-variant}
Let $h$ be a whole-plane GFF normalized so that its circle average over $\bdy\BB D$ is zero. 
For each open set $U\subset\BB C$, each compact connected set $K\subset U$, and each $\zeta \in (0,1)$, it holds with polynomially high probability as $\ep\rta 0$ that
\eqb \label{eqn-lgd-variant}
 \max_{z,w \in K} \ol D_h^\ep(z,w; U ) \leq \ep^{-\frac{1}{d_\gamma-\zeta}}  
\quad \op{and} \quad
  \ul D_h^\ep( K  , \bdy U  ) \geq \ep^{-\frac{1}{d_\gamma+\zeta}}  .
\eqe  
\end{prop} 
\begin{proof}
By \cite[Proposition 6.2]{dzz-heat-kernel} together with Lemma~\ref{lem-whole-plane-compare} and its variants for $\ol D_h^\ep$ and $\ul D_h^\ep$ (which are proven in the same way), for any bounded connected set $U\subset\BB C$, any fixed disjoint compact sets $K_1 ,K_2 \subset U $, and any $\zeta \in (0,1)$, it holds with polynomially high probability as $\ep\rta 0$ that
\eqb \label{eqn-lgd-variant-compare}
\ep^\zeta D_h^\ep(K_1,K_2 ; U) \leq  \ol D_h^\ep(K_1,K_2 ; U)  \leq   \ep^{-\zeta} D_h^\ep(K_1,K_2 ; U) 
\eqe  
and the same holds with $\ul D_h^\ep$ in place of $\ol D_h^\ep$. The lower bound for $\ul D_h^\ep(K,\bdy U)$ in~\eqref{eqn-lgd-variant} is immediate from this and~\eqref{eqn-diam-proof-lower} from the proof of Theorem~\ref{thm-diam}. The desired upper bound for $\ol D_h^\ep(z,w;U)$ follows from exactly the same argument given in Section~\ref{sec-diam} (with~\eqref{eqn-lgd-variant-compare} used to prove the needed analogue of Lemma~\ref{lem-local-dist}).  
\end{proof}

The main result of this subsection is the following proposition. 

\begin{prop} \label{prop-sle-metric-compare}
Let $h$ be a whole-plane GFF normalized so that its circle average over $\bdy \BB D$ is zero and fix $\zeta\in (0,1)$. 
Also let $U_1 \subset U_2 \subset U_3 \subset\BB C$ be bounded, connected open sets with $\ol U_1\subset U_2$ and $\ol U_2 \subset U_3$. It holds with polynomially high probability as $\ep\rta 0$ that
\eqb \label{eqn-sle-metric-compare}
\ul D_h^{\ep^{1-\zeta}}\left( z ,w ; U_3 \right) \leq  D_{h,\eta}^\ep\left( z ,w ; U_2 \right)   \leq \ep^\zeta \ol D_h^\ep\left( z ,w ; U_1 \right) ,\quad \forall z,w \in U_1 .
\eqe 
In particular, for any open set $U\subset\BB C$ and any compact set $K \subset U$, it holds with polynomially high probability as $\ep\rta 0$ that
\eqb \label{eqn-sle-metric-prob}
 \max_{z,w \in K} D_{h,\eta}^\ep(z,w; U ) \leq \ep^{-\frac{1}{d_\gamma-\zeta}}  
\quad \op{and} \quad
 D_{h,\eta}^\ep( K  , \bdy U  ) \geq \ep^{-\frac{1}{d_\gamma+\zeta}}   ,
\eqe  
and the same is true with $\ul D_h^\ep$ or $\ol D_h^\ep$ in place of $D_{h,\eta}^\ep$. 
Moreover, the conclusion of Theorem~\ref{thm-diam} remains true with any of $\ul D_h^\ep$, $\ol D_h^\ep$, or $D_{h,\eta}^\ep$ in place of $D_h^\ep$. 
\end{prop}

Proposition~\ref{prop-sle-metric-compare} does \emph{not} apply directly in the setting of Section~\ref{sec-peanosphere} since we are working with a whole-plane GFF instead of a $\gamma$-quantum cone. We will transfer to the case of a $\gamma$-quantum cone in Proposition~\ref{prop-cone-diam} below. 
For the proof of the lower bound for $D_{h,\eta}^\ep$ in Proposition~\ref{prop-sle-metric-compare}, we need the following basic estimate for the $\gamma$-LQG measure.

\begin{lem} \label{lem-circle-avg-ball-mass}
Let $h$ and $U$ be as in Proposition~\ref{prop-sle-metric-compare} and let $\ul R^\ep(z)$ for $\ep>0$ and $z\in\BB C$ be as in~\eqref{eqn-lgd-upper}. For each $\zeta ,  \xi \in (0,1)$, 
\eqbn
\BB P\left[ \mu_h\left(B_{ \ep^{ \zeta} \ul R^\ep(z)}(z) \right) \geq  \ep^{1+\zeta\left(2 + \frac{\gamma^2}{2} \right) + \xi} \right] \geq 1 - O_\ep\left( \ep^{\frac{\xi^2}{8\gamma^2 \zeta}}  \right) ,
\eqen
at a rate which is uniform over all $z\in U$. 
\end{lem}
\begin{proof}
By~\cite[Proposition 3.2]{shef-kpz}, $ h_{\ep^\zeta \ul R^\ep(z)}(z) - h_{\ul R^\ep(z)}(z)$ is centered Gaussian with variance $\log \ep^{-\zeta}$. By the Gaussian tail bound,
\eqbn
\BB P\left[  | h_{\ep^\zeta \ul R^\ep(z)}(z) - h_{\ul R^\ep(z)}(z) | \leq \frac{\xi}{2\gamma } \log \ep^{-1} \right] \geq 1 - O_\ep\left( \ep^{\frac{\xi^2}{8\gamma^2 \zeta} }  \right) .
\eqen
By~\cite[Lemma 4.6]{shef-kpz}, 
\eqbn
\BB P\left[ \mu_h\left(B_{ \ep^{ \zeta} \ul R^\ep(z)}(z) \right) \geq \ep^{\frac{\xi}{2}} \left( \ep^\zeta \ul R^\ep(z)   \right)^{2 + \frac{\gamma^2}{2}} \exp\left(  \gamma    h_{\ep^\zeta \ul R^\ep(z)}(z) \right) \right] \geq 1 - O_\ep(\ep^p) ,\: \forall p > 0 .
\eqen
We conclude by combining these estimates and recalling that $\ul R^\ep(z)^{2+\gamma^2/2} e^{\gamma h_{\ul R^\ep(z)}} = \ep$ by definition. 
\end{proof}

\begin{proof}[Proof of Proposition~\ref{prop-sle-metric-compare}]
Fix a small parameter $\xi\in (0,1)$ to be chosen later, in a manner depending only on $\zeta$ and $\gamma$. 
\medskip

\noindent\textit{Step 1: regularity event for balls and cells.}
By Lemma~\ref{lem-max-ball-radius} (and a union bound over dyadic values of $\ep$), there exists $\wt p_2 > \wt p_1 > 0$ depending only on $\gamma$ such that with polynomially high probability as $\ep\rta 0$, 
\eqb \label{eqn-compare-event}
\text{Each Euclidean ball $B\subset U_3$ with $\mu_h(B) = \delta \in (0,\ep^\xi]$ has radius in $[\delta^{\wt p_2} , \delta^{\wt p_1}]$.} 
\eqe 
By~\cite[Proposition 3.4 and Remark 3.9]{ghm-kpz}, for each $\xi \in (0,1)$, it holds with superpolynomially high probability as $\ep\rta 0$ that the following is true: for each $\delta \in (0,\ep^\xi]$ and each $a,b\in\BB R$ with $a<b$, $\eta([a,b])\subset U_3$, and $\op{diam}\eta([a,b]) \geq \delta $, the set $\eta([a,b])$ contains a Euclidean ball of radius at least $\delta^{1+\xi}$. Let $E^\ep = E^\ep(\xi)$ be the event that this is the case and~\eqref{eqn-compare-event} holds, so that $E^\ep$ occurs with polynomially high probability as $\ep\rta 0$, with the exponent depending only on $\gamma$. 

We first argue that if $\xi$ is chosen sufficiently small (in a manner depending only on $\gamma$) then there exists $p_2 > p_1 > 0$ (depending only on $\gamma$) such that on $E^\ep$,  
\eqb \label{eqn-compare-max-diam}
\ep^{p_2} \leq \ol R^\ep(z)  \leq \ep^{p_1}  ,\: \forall z\in U_3  
\quad \op{and} \quad 
\ep^{p_2} \leq \op{diam}\eta([x-\ep,x]) \leq \ep^{p_1} ,\: \forall x \in \ep\BB Z \: \text{with} \: \eta([x-\ep,x])\subset U_3 .
\eqe 
Indeed, the bounds for $\ol R^\ep(z)$ in~\eqref{eqn-compare-max-diam} (for any $p_1 < \wt p_1 < \wt p_2 < p_2$) are immediate from~\eqref{eqn-compare-event} since $\mu_h( B_{2\ol R^\ep(z)}(z) )  = \ep$. Since $\mu_h(\eta([x-\ep,x])) =  \ep$, the lower bound for $\op{diam}\eta([x-\ep,x])$ is also immediate from~\eqref{eqn-compare-event}. To get the upper bound for $\op{diam}(\eta([x-\ep,x]))$, we first use the condition on $\eta$ in the definition of~\eqref{eqn-compare-event} to get that each of the cells $\eta([x-\ep,x])$ which is contained in $U_3$ must contain a Euclidean ball of radius at least $(\ep^\xi \wedge \op{diam}(\eta([x-\ep,x]))^{1+\xi}$. This ball has $\mu_h$-mass at most $\ep$, so by~\eqref{eqn-compare-event} has radius at most $\ep^{\wt p_1}$. This gives the upper bound in~\eqref{eqn-compare-max-diam} for any $p_1 > \wt p_1$ provided $\xi$ is chosen sufficiently small. 
\medskip

\noindent\textit{Step 2: upper bound for $D_{h,\eta}^\ep$.}
We now compare $D_{h,\eta}^\ep$ and $\ol D_h^\ep$. Assume that $E^\ep $ occurs. 
To lighten notation, let $2\ol B^\ep(z) := B_{2\ol R^\ep(z)}(z)$, so that $\mu_h(2\ol B^\ep(z)) = \ep$. 
We assume that $\ep$ is chosen sufficiently small such that on $E^\ep$, each ball of the form $2\ol B^\ep(z)$ which intersects $U_1$ and each cell $\eta([x-\ep,x])$ for $x\in\ep\BB Z$ which intersects such a ball is contained in $U_2$ (this is the case for small enough $\ep$ by~\eqref{eqn-compare-max-diam}).

For $z\in U_2$, none of the SLE cells $\eta([x-\ep,x])$ for $x\in\ep\BB Z$ (which each have $\mu_h$-mass $\ep$) is properly contained in $2 \ol B^\ep(z)$, so each such cell which intersects $\ol B^\ep(z) $ must cross the annulus $2\ol B^\ep(z) \setminus \ol B^\ep(z)$. Since $ \ol R^\ep(z) \leq \ep^{p_1} \leq \ep^\xi$ by~\eqref{eqn-compare-max-diam}, the condition on $\eta$ in the definition of $E^\ep$ together with~\eqref{eqn-compare-max-diam} (applied to a segment of $\eta|_{[x-\ep,x]}$ which crosses the annulus) shows that each such cell contains a Euclidean ball of radius at least $\ol R^\ep(z)^{1+\xi} \geq \ep^{ p_2 \xi } \ol R^\ep(z) $ which is itself contained in $2\ol B^\ep(z)$. Such a ball has Lebesgue measure at least $\pi \ep^{ 2 p_2 \xi } \ol R^\ep(z)^2$, so there can be at most $4 \ep^{-2 p_2 \xi}$ such balls contained in $2\ol B^\ep(z)$. Hence there can be at most $4 \ep^{-2 p_2 \xi}$ cells of the form $\eta([x-\ep,x])$ which intersect $\ol B^\ep(z)$. In particular, any connected subset of $U_1$ which can be covered by $N$ balls of the form $\ol B^\ep(z)$ can be covered by $\ep^{-2p_2\xi} N$ cells of the form $\eta([x-\ep,x])$ for $x\in\ep\BB Z$. 
Each such cell is contained in $U_2$ by the assumption in the preceding paragraph.  
If we choose $\xi \leq \zeta/(2p_2)$, this shows that with polynomially high probability $ D_{h,\eta}^\ep(z,w ; U_2)   \leq \ep^\zeta \ol D_h^\ep(z , w ; U_1)$ for all $z ,w \in U_1$.
 \medskip

\noindent\textit{Step 3: lower bound for $D_{h,\eta}^\ep$.}
It remains to compare $\ul D_h^\ep$ and $D_{h,\eta}^\ep$. On the event $E^\ep$ above, each cell $\eta([x-\ep  ,x])$ which is contained in $U_2$ contains a Euclidean ball $B_x^\ep$ of radius at least $(\op{diam}\eta([x-\ep  ,x]) )^{1+\xi} \geq \ep^{p_2 \xi} \op{diam} \eta([x-\ep,x])$. By~\eqref{eqn-compare-max-diam}, this ball $B_x^\ep$ has radius at least $\ep^{p_2(1+ \xi)}$ and hence contains a point of $(\ep^{2 p_2 } \BB Z^2) \cap U_3$ provided $\ep$ is small enough that $\ep^{p_1} \leq \op{dist}(\bdy U_2 ,\bdy U_3)$. 

By Lemma~\ref{lem-circle-avg-ball-mass} (applied with $\ep^{1-\zeta}$ in place of $\ep$, $  p_2 \xi$ in place of $\zeta$, and $\xi^{1/4}$, say, in place of $\xi$)  and a union bound over $(\ep^{2p_2} \BB Z^2) \cap U_3$, if $\xi$ is chosen sufficiently small, in a manner depending only on $\gamma$, then with polynomially high probability as $\ep\rta 0$, 
\eqb \label{eqn-min-circle-avg}
  \mu_h\left(B_{ \ep^{ p_2\xi(1-\zeta) } \ul R^{\ep^{1-\zeta}}(z)}(z) \right)  
  \geq \ep^{1-\zeta +   p_2 \xi (1-\zeta) \left(2 + \frac{\gamma^2}{2} \right) + \xi^{1/4} (1-\zeta) }  
  \geq \ep,
  \quad \forall z \in (\ep^{2 p_2 } \BB Z^2) \cap U_3 .
\eqe 
Furthermore, by a standard Gaussian estimate (see, e.g.,~\cite[Proposition 2.4]{lqg-tbm2}) it holds with polynomially high probability as $\ep\rta 0$ that each of the balls $\ul B^{\ep^{1-\zeta}}(z)$ for $z\in U_2$ is contained in $U_3$.

Henceforth assume that $E^\ep$ occurs,~\eqref{eqn-min-circle-avg} holds, and the event described just after~\eqref{eqn-min-circle-avg} occurs. 
Since the radius of $B_x^\ep$ is at least $\ep^{p_2(1+\xi)}$ for each $x \in \ep\BB Z$ with $\eta([x-\ep,x]) \subset U_2$, for each such $x$ we can find $z\in (\ep^{2 p_2 } \BB Z^2) \cap B_x^\ep$ which lies at Euclidean distance at least $\frac14 \op{diam} B_x^\ep$ from $\bdy B_x^\ep$. Since $\mu_h(B_x^\ep) \leq \ep$ and by~\eqref{eqn-min-circle-avg}, the ball $B_{ \ep^{2p_2\xi(1-\zeta)} \ul R^{\ep^{1-\zeta}}(z)}(z)$ cannot be contained in $B_x^\ep$, which means that $\ul R^{\ep^{1-\zeta}}(z) \geq \frac14 \ep^{- p_2\xi(1-\zeta)} \op{diam}(B_x^\ep) \geq \op{diam} \eta([x-\ep,x])$. In other words, $\eta([x-\ep,x]) \subset \ul B^{\ep^{1-\zeta}}(z)$. We also have $\ul B^{\ep^{1-\zeta}}(z) \subset U_3$ by our above assumption. Since a $z\in \eta([x-\ep,x])$ with this property can be found for any $x\in \ep\BB Z$, we obtain the left inequality in~\eqref{eqn-sle-metric-compare}.

The bound~\eqref{eqn-sle-metric-prob} and its variants for $\ul D_h^\ep$ and $\ol D_h^\ep$ is immediate from~\eqref{eqn-sle-metric-compare} and Proposition~\ref{prop-lgd-variant}. The analogue of Theorem~\ref{thm-diam} for $D_{h,\eta}^\ep$ is immediate from~\eqref{eqn-sle-metric-prob} and a union bound over dyadic values of $\ep$ (here we also need to use~\eqref{eqn-sle-dist-mono} and Lemma~\ref{lem-sle-metric} since $D_{h,\eta}^\ep$ is not monotone in $\ep$). Similar statements hold for $\ul D_h^\ep$ and $\ol D_h^\ep$. 
\end{proof}

\subsection{Ball growth exponent for random planar maps}
\label{sec-ball-size}

In order to study mated-CRT maps, we need to transfer the conclusion of Proposition~\ref{prop-sle-metric-compare} from the case of a whole-plane GFF to the case of a $\gamma$-quantum cone. We restrict attention to balls contained in the unit disk to avoid technicalities related to our choice of embedding for the $\gamma$-quantum cone. 

\begin{prop} \label{prop-cone-diam}
Let $h$ be the circle average embedding of a $\gamma$-quantum cone. 
For each $\zeta\in (0,1)$ and each $\rho \in (0,1)$, it holds with polynomially high probability as $\ep\rta 0$ (at a rate depending on $\rho$, $\zeta$, and $\gamma$) that  
\eqb \label{eqn-cone-diam}
D_{h,\eta}^\ep\left( 0 , \bdy B_\rho(0) \right) \geq \ep^{-\frac{1}{d_\gamma+\zeta}} \quad \op{and} \quad
\max_{z,w\in B_\rho(0) } D_{h,\eta}^\ep\left( z,w;  \BB D \right) \leq \ep^{-\frac{1}{d_\gamma -\zeta}} .
\eqe  
\end{prop}
\begin{proof}
Recall that $h|_{\BB D}$ agrees in law with the corresponding restriction of a whole-plane GFF plus $\gamma \log (|\cdot|^{-1})$, normalized so that its circle average over $\bdy\BB D$ is 0.  
Hence it is enough to prove the lemma with $h$ replaced with a whole-plane GFF plus $\gamma \log (|\cdot|^{-1})$. We assume that this replacement has been made throughout the proof. 

The lower bound in~\eqref{eqn-cone-diam} is immediate from Proposition~\ref{prop-sle-metric-compare} for $D_{h,\eta}^\ep$ (applied with $K = \bdy B_{\rho/2}(0)$ and $U = B_{\rho}(0)\setminus B_{\rho/4}(0)$, say) since $h|_{\BB D \setminus B_{\rho/2}(0)}$ differs from the corresponding restriction of a whole-plane GFF by a deterministic, bounded function. 

To get the upper bound, we will bound the distance across dyadic annuli centered at 0, then sum over the annuli. Recall that $\wt D_{h ,\eta }^{ \ep} $ is the modified version of $D_{h,\eta}^\ep$ from Definition~\ref{def-sle-metric}. For most of the argument we will use this distance instead of $D_{h ,\eta }^{ \ep}$ since the former is monotone in $\ep$ (recall~\eqref{eqn-sle-dist-mono}). 
We will switch back to $D_{h,\eta}^\ep$ at the end by means of Lemma~\ref{lem-sle-metric}. 
 
Let $\wt\zeta \in (0,1)$ be a small parameter to be chosen later, in a manner depending only on $\zeta$ and $\gamma$. 
Also fix $r \in (0,  \rho \wedge ( 1-\rho) )$. For $n\in\BB N_0$, define the annulus $\mcl A_n := B_{2^{-n} \rho}(0) \setminus B_{2^{-n-1}\rho}(0)$ and the slightly larger annulus $\mcl A_n' := B_{2^{-n} (\rho +r) }(0) \setminus B_{2^{-n-1}(\rho-r) }(0) $. 
Define the re-scaled field/curve pair 
\eqb \label{eqn-scaled-field}
h^n := h(2^n\cdot) - h_{2^{-n}}(0) \quad \op{and} \quad \eta^n :=  2^{-n} \eta\left( T_n^{-1} \cdot     \right) \quad \text{for} \quad T_n := 2^{(2+\gamma^2/2) n} e^{\gamma h_{2^{-n}}(0)} .
\eqe 
Then $(h^n,\eta^n) \eqD (h,\eta)$ (note that $\eta^n$ is parametrized by $\mu_{h^n}$-mass by the $\gamma$-LQG coordinate change formula~\cite[Proposition 2.1]{shef-kpz}). Furthermore, each segment $\eta^n([a,b])$ with $0 <b-a \leq \ep$ is equal to the re-scaled segment $2^{-n} \eta([T_n^{-1} a , T_n^{-1} b])$. 
Therefore,
\eqb \label{eqn-cone-compare-scale}
\wt D_{h,\eta}^\ep(z,w ; \mcl A_n') = \wt D_{h^n,\eta^n}^{T_n \ep}(2^n z , 2^n w ; \mcl A_0')  ,\quad \forall z,w \in \mcl A_n' .
\eqe

By standard estimates for the $\gamma$-LQG measure (see, e.g.,~\cite[Lemma A.3]{ghs-dist-exponent}), we can find $q = q(\gamma) > 0$ such that with polynomially high probability as $\ep\rta 0$, we have $\mu_h(B_{2\ep^q}(0)) \leq \ep$, which means that $B_{2\ep^q}(0)$ does not contain any $\ep$-LQG mass segment of $\eta$. It follows from~\cite[Proposition 3.4 and Remark 3.9]{ghm-kpz} that with superpolynomially high probability as $\ep\rta 0$, the number of crossings of $B_{2\ep^q}(0)\setminus B_{\ep^q}(0)$ by $\eta$ is at most $\ep^{-\wt\zeta}$. Consequently, with polynomially high probability as $\ep\rta 0$, $B_{\ep^q}(0)$ can be covered by at most $\ep^{-\wt\zeta}$ segments of $\eta$ which are contained in $B_{2\ep^q}(0)$ and hence $\mu_h$-mass at most $\ep$ and so
\eqb \label{eqn-cone-compare-ball}
\max_{z,w\in B_{\ep^q}(0)}  \wt D_{h,\eta}^\ep\left(z,w ; \BB D \right) \leq \ep^{-\wt\zeta}  .
\eqe

We will now estimate $\wt D_{h^n,\eta^n}^{T_n\ep} |_{\mcl A_0'}$ for $n\in\BB N_0$ with $2^{-n}\geq \ep^q$, then sum over all such $n$. 
For each such $n$, let $E^n = E^n(\ep)$ be the event that 
\eqbn
|h_{2^{-n}}(0)| \leq \gamma \log 2^n  + \frac{\wt\zeta}{ \gamma} \log \ep^{-1}
\quad \op{and} \quad 
\max_{z,w\in \mcl A_0} \wt D_{h^n,\eta^n}^{T_n \ep }(z,w ; \mcl A_0') \leq \ep^{-\frac{1}{d_\gamma-\zeta/2}} .
\eqen   

The random variable $h_{2^{-n}}(0)$ is centered Gaussian with mean $\gamma \log 2^n$ and variance $\log 2^n$~\cite[Section 3.1]{shef-kpz}, so the probability that the first condition in the definition of $E^n$ fails decays polynomially in $\ep$, uniformly over all $n\in\BB N$ with $2^{-n} \geq \ep^q$. If $|h_{2^{-n}}(0)| \leq \left(\gamma + \frac{\wt\zeta}{ \gamma} \right) \log 2^n$, then if $\wt\zeta$ is chosen sufficiently small, in a manner depending only on $\gamma$ we have (in the notation of~\eqref{eqn-scaled-field}) $T_n  \geq 2^{(2+\gamma^2/2 - \gamma  ) n} \ep^{\wt\zeta}  \geq  \ep^{\wt\zeta}$. Since $(h^n,\eta^n) \eqD (h,\eta)$ and $h|_{\mcl A_0'}$ differs from the corresponding restriction of a whole-plane GFF by a deterministic function which is bounded independently of $\ep$ and $n$, we infer from Proposition~\ref{prop-sle-metric-compare} and~\eqref{eqn-sle-dist-mono} (to compare $\wt D_{h^n,\eta^n}^{T_n \ep}$ to $\wt D_{h^n,\eta^n}^{\ep^{1 + \wt \zeta}}$) that if $\wt \zeta$ is chosen sufficiently small, in a manner depending only on $\gamma$ and $\zeta$, then $E^n$ occurs with polynomially high probability as $\ep\rta 0$, uniformly over all $n\in\BB N_0$ with $2^{-n} \geq \ep^q$. By a union bound over logarithmically many values of $n$, we see that with polynomially high probability as $\ep\rta 0$, $E^n$ occurs for every such $n$. Combining this with~\eqref{eqn-cone-compare-scale} and~\eqref{eqn-cone-compare-ball}, and summing over all $n$ with $2^{-n}\geq\ep^q$, we see that the upper bound in~\eqref{eqn-cone-diam} holds with $\wt D_{h,\eta}^\ep$ in place of $D_{h,\eta}^\ep$ with polynomially high probability as $\ep\rta 0$. We then convert from $\wt D_{h,\eta}^\ep$ to $D_{h,\eta}^\ep$ by means of Lemma~\ref{lem-sle-metric}.
\end{proof}

As a consequence of Proposition~\ref{prop-cone-diam}, we obtain Theorem~\ref{thm-ball-size} in the case of the mated-CRT map. 

\begin{prop} \label{prop-crt-ball}
Let $\mcl G = \mcl G^1$ be the $\gamma$-mated-CRT map with unit increment size. For each $\zeta \in (0,1)$, it holds with polynomially high probability as $r \rta \infty$ (at a rate depending only on $\zeta$ and $\gamma$) that the volume of the metric ball of radius $r$ satisfies
\eqb \label{eqn-crt-ball}
r^{d_\gamma - \zeta} \leq \#\mcl B_r^{\mcl G}(0) \leq r^{d_\gamma + \zeta } .
\eqe  
\end{prop}
\begin{proof}
Recall that for $\ep > 0$, the mated-CRT map $\mcl G^\ep$ agrees in law with $\mcl G$. 
Furthermore, the map $\ep\BB Z \ni x \mapsto \eta([x-\ep,x])$ is a graph isomorphism from $\mcl G^\ep$ to the adjacency graph of cells $\eta([x-\ep,x])$ for $x\in\ep\BB Z$. 
Proposition~\ref{prop-cone-diam} implies that for each $\rho \in (0,1)$, it holds with polynomially high probability as $\ep\rta 0$ that 
\eqb \label{eqn-quantum-euc-balls}
\eta\left(  \mcl B_{\ep^{-1/(d_\gamma+\zeta)} }^{\mcl G^\ep}(0)  \right)   \subset B_\rho(0) \quad \op{and} \quad 
 \mcl B_{\ep^{-1/(d_\gamma-\zeta)} }^{\mcl G^\ep}(0) \supset \eta^{-1}\left( B_\rho(0) \cap \eta(\ep\BB Z) \right) .
\eqe 
Since $h|_{\BB D}$ agrees in law with the corresponding restriction of a whole-plane GFF plus $ \gamma \log |\cdot|$, it is easily seen (see, e.g.,~\cite[Lemmas A.2 and A.3]{ghs-dist-exponent}) that $\mu_h(B_\rho(0) )$ has finite moments of all negative orders and a finite moment of some positive order, so by Markov's inequality it holds with polynomially high probability as $\ep\rta 0$ that $\ep^\zeta \leq \mu_h\left(B_\rho(0) \right) \leq \ep^{-\zeta}$. Since the cells $\eta([x-\ep,x])$ have $\mu_h$-mass $\ep$, it holds with polynomially high probability as $\ep\rta 0$ that
\eqbn
\ep^{-1+\zeta} \leq  \#\left(  B_\rho(0) \cap \eta(\ep\BB Z) \right)   \leq \ep^{-1-\zeta} .
\eqen
Combining this with~\eqref{eqn-quantum-euc-balls} (applied with $\ep = r^{-d_\gamma-\zeta}$ and with $\ep = r^{-d_\gamma+\zeta}$), possibly shrinking $\zeta$, and recalling that $\mcl G^\ep \eqD \mcl G$ shows that~\eqref{eqn-crt-ball} holds with polynomially high probability as $r \rta \infty$.
\end{proof}

\begin{proof}[Proof of Theorem~\ref{thm-ball-size}]
The theorem statement in the case when $\mcl M =\mcl G$ is a mated-CRT map follows from Proposition~\ref{prop-crt-ball} and a union bound over dyadic values of $r$. If $\mcl M$ is one of the other planar maps listed above the theorem statement, let $\mcl G$ be the mated-CRT map with the same value of $\gamma$ as $\mcl M$. Proposition~\ref{prop-crt-ball} together with the coupling results for $\mcl M$ and $\mcl G$ established in~\cite{ghs-map-dist} (see in particular~\cite[Theorem 1.5 and Lemma 1.12]{ghs-dist-exponent}) shows that for each $\zeta\in (0,1)$, it holds with polynomially high probability as $r \rta\infty$ that 
\eqb \label{eqn-map-ball}
r^{d_\gamma - \zeta} \leq \#\mcl B_r^{\mcl M}(0) \leq r^{d_\gamma + \zeta } .
\eqe  
We now conclude as above by means of a union bound over dyadic values of $r$. 
\end{proof}

\appendix
 
\section{Proof of Lemma~\ref{lem-gff-compare}}
\label{sec-gff-compare}
 
We will compare $\wh h^\tr$ and $h^U$, and then $\wh h^\tr$ and $\wh h$. The comparison of $h^U$ and $h$ follows from Lemma~\ref{lem-whole-plane-compare}.

\begin{lem} \label{lem-h-wn-diff}
If $U\subset \BB C$ is a bounded Jordan domain then we can couple $\wh h^\tr$ with the zero-boundary GFF $h^U$ on $U$ in such a way that the following is true. 
If we let $K$ be the set of points in $U$ which lie at distance at least $\frac{1}{10}$ from $\bdy U$, then $(h - \wh h^\tr)|_K$ a.s.\ admits a modification which is a continuous Gaussian random function and there are constants $c_0,c_1 >0$ depending only on $U$ such that for $A>1$,
\eqb \label{eqn-h-wn-diff}
\BB P\left[ \max_{z\in K} |(h^U - \wh h^\tr)(z)| \leq A \right] \geq 1 - c_0 e^{-c_1 A^2} .
\eqe
\end{lem}
\begin{proof}
Recall the white noise $W$ used to define $\wh h^\tr$ in~\eqref{eqn-wn-truncate} and for $0 < t < \wt t \leq\infty $, let
\eqb
 h^U_{t,\wt t}(z) := \sqrt \pi \int_{t^2}^{\wt t^2} p_U(s/2; z,w) \, W(dw,ds)  .
\eqe
It is easily checked using the Kolmogorov continuity criterion that $h^U_{t,\infty}$ for $t>0$ a.s.\ admits a continuous modification. Furthermore, the distributional limit $h^U := \lim_{t\rta 0} h^U_{t,\infty}$ is the zero-boundary GFF on $U$~\cite[Lemma 5.4]{rhodes-vargas-review}. 
This gives a coupling of $h^U$ with $\{\wh h_t^\tr\}_{t \in [0,1]}$. 

Set $f_t(z) := \wh h_{t,1}^U(z ) - \wh h_t^\tr(z )$, so that
\eqbn
f_t(z) = \sqrt\pi \int_{t^2}^1 \int_U  q_s(z,w) \,W(dw,ds) \quad \op{for} \quad q_s(z,w) :=  p_U(s/2 ;z ,w) - p_{B_{1/10}(z )}(s/2 ; z  ,w) . 
\eqen 
Since Brownian motion started from $z$ is extremely unlikely to travel distance further than $\frac{1}{10} \wedge \op{dist}(z,\bdy U)$ before time $t$ when $t$ is small, $\op{Var} f_t(z) = \int_{t^2}^1 \int_U q_s(z,w) \,dw\,ds$ converges as $t\rta 0$ for each fixed $z\in U$.
This shows that the function $f(z) := \lim_{t\rta 0} f_t(z)$ is a.s.\ defined for Lebesgue-a.e.\ $z \in U$ and is Gaussian with covariances
\allb \label{eqn-h-wn-diff-cov}
 \op{Cov} \left( f  (z_1) , f(z_2)   \right) 
  =  \pi \int_0^1  \int_{U} q_s(z_1,w) q_s(z_2,w) \,dw  \, ds   , \quad \forall  z_1,z_2 \in U . 
\alle

To show that $f|_K$ admits a continuous modification, we will show that 
\eqb \label{eqn-h-wn-diff-kc}
\op{Var}(f(z_1) - f(z_2)) \preceq |z_1-z_2|  ,\quad \text{uniformly over all $z_1,z_2 \in K$}.
\eqe
Since $f$ is Gaussian, this together with the Kolmogorov continuity criterion will show that a.s.\ $f$ is locally H\"older continuous with any exponent less than $1/2$. 
This shows that $(h^U - \wh h^\tr)|_K = \wh h^U_{1,\infty} + f$ a.s.\ admits a continuous modification.
Furthermore, since $K$ is compact a.s.\ $\max_{z\in K} |(h^U - \wh h^\tr)(z) |  <\infty$ so the Borell-TIS inequality~\cite{borell-tis1,borell-tis2} (see, e.g.,~\cite[Theorem 2.1.1]{adler-taylor-fields}) shows that $\BB  E[\max_{z\in K} |(h^U - \wh h^\tr)(z) |] < \infty$ and that~\eqref{eqn-h-wn-diff} holds for an appropriate choice of $c_0$ and $c_1$. 

It remains only to prove~\eqref{eqn-h-wn-diff-kc}. 
This is done by an elementary but somewhat tedious calculation.
By translating and rotating $U$, it suffices prove~\eqref{eqn-h-wn-diff-kc} in the case when $U$ is such that $\ep,-\ep \in K$ with $z_1 = \ep$ and $z_2 = -\ep$, with the implicit constant depending only on the size and shape of $U$. 
Here and throughout the proof, we identify $\BB C$ with $\BB R^2$, so $\ep,-\ep \in \BB C$ correspond to the points $(\ep,0)$ and $(-\ep,0)$.
Using~\eqref{eqn-h-wn-diff-cov}, we find that
\allb \label{eqn-h-wn-int0}
\op{Var}\left(  f(\ep) - f(-\ep) \right)  
&=  \op{Var} f(\ep)  + \op{Var} f(-\ep) - 2\op{Cov}(f(\ep) , f(0) )  \notag \\
&= \pi \int_0^1  \int_{U} q_s(w,\ep)^2 + q_s(w,-\ep)^2 - 2q_s(w,\ep)q_s(w,-\ep) \,dw  \, ds  .
\alle
Since $|q_s(z,w)| \leq 2 p(s/2;z,w)$, it is clear that $\int_0^\ep \int_{U} q_s(w,\ep)^2 + q_s(w,-\ep)^2 - 2q_s(w,\ep)q_s(w,-\ep) \,dw  \, ds   = O_\ep(\ep)$. 
We therefore only need to bound the integral from $\ep$ to 1. 

If $\mcl B^z$ denotes a standard planar Brownian motion started from $z$, then the law of $\mcl B_{s/2}^z$ is $\frac{1}{ \pi s} e^{ -   |w|^2 /s } \,dw$ and the conditional law of $\mcl B^z|_{[0,s]}$ given $\{\mcl B_s^z = w\}$ is that of a Brownian bridge from $z$ to $w$ in time $s/2$. Hence, if $\mcl B^{s,z,w}$ denotes such a Brownian bridge, then
\eqb \label{eqn-bridge-heat-kernel}
p_U(s/2 ;z,w)  = \frac{1}{ \pi s} e^{ - \frac{1}{s}  |w-z|^2 } \BB P\left[ \mcl B^{s,z,w}([0,s/2]) \subset U \right]. 
\eqe 
By~\eqref{eqn-bridge-heat-kernel} (applied for $U$ and with $B_{1/10}(z)$ in place of $U$) we see that for $z\in K$,  
\allb \label{eqn-h-wn-bridge}
&q_s(z,w) = \frac{1}{ \pi s} e^{ - \frac{1}{s} |w-z|^2 } \wt q_s(z,w) , 
&\qquad \text{for} \: \wt q_s(z,w) := \BB P\left[ \mcl B^{s,z,w} \: \text{exits $B_{1/10}(z)$ but not $U$ before time $s/2$}\right] .
\alle

Plugging~\eqref{eqn-h-wn-bridge} into~\eqref{eqn-h-wn-int0} shows that
\allb \label{eqn-h-wn-int1}
&\pi \int_{\ep}^1  \int_{U} q_s(w,\ep)^2 + q_s(w,-\ep)^2 - 2q_s(w,\ep)q_s(w,-\ep) \,dw  \, ds \notag \\ 
&= \int_{\ep}^1  \int_{U} \frac{1}{\pi  s^2} \left( e^{ - \frac{2}{s}   |w-\ep|^2 }  \wt q_s(w,\ep)^2 + e^{ - \frac{2}{s}  |w+\ep|^2  }   \wt q_s(w,-\ep)^2 
- 2 e^{-\frac{1}{s} (|w-\ep|^2 + |w+\ep|^2)}\wt q_s(w,\ep)\wt q_s(w,-\ep) \right) \,dw  \, ds  \notag\\
&= \int_{\ep}^1  \int_{U} \frac{1}{\pi  s^2} e^{-\frac{2}{s} (|w|^2 + \ep^2)} \left( e^{ - \frac{\ep}{s}   \re w }  \wt q_s(w,\ep)^2 + e^{ \frac{\ep}{s}    \re w }   \wt q_s(w,-\ep)^2  - 2  \wt q_s(w,\ep)\wt q_s(w,-\ep) \right) \,dw  \, ds  .
\alle
To bound this last integrand, we couple the Brownian bridges from~\eqref{eqn-h-wn-bridge} for $s=\pm\ep$ in such a way that $\mcl B^{s,\ep,w}_u = \mcl B_u - \frac{2u}{s} \mcl B_{s/2} + \ep  + \frac{2u}{s}(w - \ep)$ and $\mcl B^{s, -\ep,w}_u = \mcl B_u - \frac{2u}{s} \mcl B_{s/2} - \ep  + \frac{2u}{s}(w + \ep)$ for $\mcl B$ a standard linear Brownian motion on $[0,s/2]$. 
Then
\eqbn
| \mcl B^{s,\ep,w}_u - \mcl B^{s,-\ep,w}_u | = 2\left( 1 - \frac{2 u}{s} \right) \ep \leq 2\ep .
\eqen
Let $E_{\ep}$ be the event that $B^{s,\ep,w}$ exits $B_{1/10}(\epsilon )$ but not $U$ before time $s/2$, and define $E_{-\ep}$ similarly with $-\ep$ in place of $\ep$. 
Then on $E_{\ep}\setminus E_{-\ep}$, either $\mcl B^{s,\ep,w}$ exits $B_{1/10}(\ep)$ without exiting $B_{1/10+4\ep}(\ep)$ or $\mcl B^{s,\ep,w}$ enters the $4\ep$-neighborhood of $\bdy U$ without exiting $U$.
The probability that it does so is of order $O_\ep(\ep)$, uniformly over $w \in U$ and $s > 0$. 
A similar statement holds with the roles of $\ep$ and $-\ep$ reversed. Therefore,
\eqbn
| \wt q_s(w,\ep) -  \wt q_s(w,-\ep) | = O_\ep(\ep) .
\eqen

Plugging this last bound into~\eqref{eqn-h-wn-int1} and recalling~\eqref{eqn-h-wn-int0} and the sentence just after, we get
\allb
& \op{Var}\left(  f (\ep) - f(-\ep) \right) \notag \\
&\leq \int_{\ep}^1  \int_{U} \frac{ e^{-\frac{2}{s} (|w|^2 + \ep^2)}}{ \pi  s^2}  \left(    \frac{ \wt q_s(w,\ep)^2 }{  e^{   \frac{\ep}{s}  \re w } } + e^{  \frac{\ep}{s}  \re w } \left(   \wt q_s(w, \ep)^2  + O_\ep(\ep) \right) - 2  \wt q_s(w,\ep)^2  \right) \,dw  \, ds   + O_\ep(\ep)  . \label{eqn-h-wn-int2}
\alle 
For $s\in [\ep,1]$, we have that $|(1/s) \ep \re w|  \leq |w|$ and the integral of $\frac{1}{\pi s^2} e^{-\frac{2}{s} |w|^2 + \frac{1}{s} |w|}$ over $(w,s) \in U\times [\ep,1]$ is finite. 
This allows us to move the $O_\ep(\ep)$ inside the integral in~\eqref{eqn-h-wn-int2} to outside the integral, so we get that the right side of~\eqref{eqn-h-wn-int2} is bounded above by
\alb
&\int_\ep^1  \int_{U} \frac{1}{ \pi  s^2} e^{-\frac{2}{s} (|w|^2 + \ep^2)} \wt q_s(w,\ep)^2 \left(   e^{ - \frac{\ep}{s}  \re w }   + e^{  \frac{\ep}{s}   \re w }  - 2 \right) \,dw  \, ds  + O_\ep(\ep) \notag  \\
&\leq \int_\ep^1  \int_{U} \frac{1}{\pi  s^2} e^{-\frac{2}{s} (|w|^2 + \ep^2)} \wt q_s(w,\ep)^2 \left(   e^{ - \frac{\ep}{2s} \re w }   - e^{  \frac{\ep}{2s}     \re w }  \right)^2 \,dw  \, ds  + O_\ep(\ep)  \notag \\
&\preceq \int_\ep^1  \int_{U} \frac{1}{s^3} e^{-\frac{2}{s} (|w|^2 + \ep^2)}  \ep^2 (\re w)^2   \,dw  \, ds  + O_\ep(\ep)   
 = O_\ep(\ep) , \label{eqn-h-wn-int3}
\ale
where in the second inequality we use that $\wt q_s(w,\ep) \leq 1$ and that $|e^x - e^{-x}| \preceq |x|$ for $|x|\preceq 1$. 
This gives~\eqref{eqn-h-wn-diff-kc} for $z_1=\ep$ and $z_2 = -\ep$, as desired.
\end{proof}

\begin{lem} \label{lem-wn-tr-diff}
If $\wh h$ and $\wh h^\tr$ are defined using the same white noise, then $\wh h- \wh h^\tr$ a.s.\ admits a continuous modification and for any compact set $K\subset \BB C$, there are constants $c_0 , c_1 > 0$ (depending only on $K$) such that for $A>1$,
\eqbn
\BB P\left[ \max_{z \in K} |(\wh h - \wh h^\tr)(z)| > A \right] \leq c_0 e^{-c_1 A^2} .
\eqen
\end{lem}
\begin{proof}
This follows from exactly the same argument used to prove Lemma~\ref{lem-h-wn-diff}.
\end{proof}

\begin{proof}[Proof of Lemma~\ref{lem-gff-compare}]
Combine Lemmas~\ref{lem-whole-plane-compare},~\ref{lem-h-wn-diff}, and~\ref{lem-wn-tr-diff}.
\end{proof}

\def\cprime{$'$}

\end{document}